\documentclass[11pt,a4paper,leqno]{article}
\usepackage{a4wide}
\usepackage{latexsym}
\usepackage{amsmath}
\usepackage{amssymb}
\usepackage{graphicx}
\newtheorem{defin}{Definition}
\newtheorem{lemma}{Lemma}
\newtheorem{prop}{Proposition}
\newtheorem{theo}{Theorem}

\newenvironment{proof}{\medskip\par\noindent{\bf Proof}}{\hfill $\Box$
\medskip\par}
\begin{document}
\title{On singularly perturbed linear initial value problems with mixed irregular and Fuchsian time singularities}
\author{{\bf A. Lastra, S. Malek}\\
University of Alcal\'{a}, Departamento de F\'{i}sica y Matem\'{a}ticas,\\
Ap. de Correos 20, E-28871 Alcal\'{a} de Henares (Madrid), Spain,\\
University of Lille 1, Laboratoire Paul Painlev\'e,\\
59655 Villeneuve d'Ascq cedex, France,\\
{\tt alberto.lastra@uah.es}\\
{\tt Stephane.Malek@math.univ-lille1.fr }}
\date{January, 11 2019}
\maketitle
\thispagestyle{empty}
{ \small \begin{center}
{\bf Abstract}
\end{center}
We consider a family of linear singularly perturbed PDE relying on a complex perturbation parameter $\epsilon$. As in the former study
\cite{lama1} of the authors, our problem possesses an irregular singularity in time located at the origin but, in the present work, it entangles
also differential operators of Fuchsian type acting on the time variable. As a new feature, a set of sectorial holomorphic solutions are built
up through {\it iterated} Laplace transforms and Fourier inverse integrals following a classical multisummability procedure introduced by
W. Balser. This construction has a direct issue on the Gevrey bounds of their asymptotic expansions w.r.t $\epsilon$ which are shown to bank
on the order of the leading term which combines both irregular and Fuchsian types operators.\medskip

}
\bigskip \bigskip

\section{Introduction}

In this paper, we aim attention at a family of singularly perturbed linear partial differential equations which combines two varieties of
differential operators acting on the time variable of so-called irregular and Fuchsian types. The definition of irregular type operators
in the context of PDE can be found in the paper \cite{man2} by T. Mandai and we refer to the excellent textbook \cite{geta} by R. G\'{e}rard and H. Tahara for an extensive study of Fuchsian ordinary and partial differential equations.

The problem under study can be displayed as follows
\begin{multline}
Q(\partial_{z})u(t,z,\epsilon) = R_{D}(\partial_{z})\epsilon^{k\delta_{D}} (t^{k+1}\partial_{t})^{\delta_{D}}
(t\partial_{t})^{m_D}u(t,z,\epsilon)\\
+ P(z,\epsilon,t^{k+1}\partial_{t},t\partial_{t},\partial_{z})u(t,z,\epsilon) + f(t,z,\epsilon) \label{main_PDE_u_intro}
\end{multline}
for vanishing initial data $u(0,z,\epsilon) \equiv 0$, where $k,\delta_{D},m_{D} \geq 1$ are integers, $Q(X),R_{D}(X)$ stand for polynomials
with complex coefficients and $P(z,\epsilon,V_{1},V_{2},V_{3})$ represents a polynomial in the arguments $V_{1},V_{2},V_{3}$ with holomorphic
coefficients w.r.t the perturbation parameter $\epsilon$ in the vicinity of the origin in $\mathbb{C}$ and holomorphic relatively to the
space variable $z$ on a horizontal strip in $\mathbb{C}$ with the shape $H_{\beta} = \{ z \in \mathbb{C} / |\mathrm{Im}(z)| < \beta \}$, for some
given $\beta > 0$. The forcing term $f(t,z,\epsilon)$ relies analytically on $\epsilon$ near the origin and holomorphically on $z$ on
$H_{\beta}$ and defines either an analytic function near 0 or an entire function with (at most) exponential growth of prescribed order w.r.t the time $t$.

This work can be seen as a continuation of our previous study \cite{lama1} where we focused at the next problem (in the linear setting)
\begin{multline}
 Q(\partial_{z})t^{\kappa + 1}\partial_{t}y(t,z,\epsilon) = \epsilon^{(\delta_{D}-1)\kappa}
 t^{\delta_{D}(\kappa + 1)}\partial_{t}^{\delta_{D}}R_{D}(\partial_{z})y(t,z,\epsilon)\\
 + \sum_{l=1}^{D-1} \epsilon^{\Delta_{l}} t^{d_{l} + \kappa + 1}\partial_{t}^{\delta_l}R_{l}(\partial_{z})y(t,z,\epsilon) +
 t^{\kappa + 1}f(t,z,\epsilon) \label{former_main_PDE_y_intro}
\end{multline}
for vanishing initial data $y(0,z,\epsilon) \equiv 0$, where $Q,R_{D},R_{l}$, $l=1,\ldots,D-1$ stand for polynomials, $D \geq 2$,
$\delta_{D},\kappa \geq 1$, $\Delta_{l},d_{l},\delta_{l} \geq 0$ are integers and $f(t,z,\epsilon)$ represents a holomorphic function near
the origin w.r.t $(t,\epsilon)$ which is holomorphic on $H_{\beta}$ w.r.t $z$ as above. This equation involves exclusively time differential
operators of irregular type which carry one single level (named also rank in the literature) $\kappa$, meaning that all operators
$t^{\delta_{D}(\kappa+1)}\partial_{t}^{\delta_{D}}$ and $t^{d_{l} + \kappa + 1}\partial_{t}^{\delta_{l}}$ appearing in
(\ref{former_main_PDE_y_intro}) can be expressed as $P(t,t^{\kappa+1}\partial_{t})$ for some polynomials $P(t,V_{1}) \in \mathbb{C}[t,V_{1}]$
through the expansion (\ref{Tahara_formula}) stated in Lemma 4, under the requirements
$$ d_{l} + \kappa + 1 \geq \delta_{l}(\kappa+1) $$
for all $1 \leq l \leq D-1$. For our present problem (\ref{main_PDE_u_intro}), this condition is in general not fulfilled. Namely, in
the example treated after
Theorem 2, the operator $t^{4}\partial_{t}(t\partial_{t}) = t^{4}\partial_{t} + t^{5}\partial_{t}^{2}$ writes as a sum of two irregular
operators that possess two different ranks, namely the rank of $t^{4}\partial_{t}$ is $3$ and $t^{5}\partial_{t}^{2}$ is of rank 1 since
$t^{5}\partial_{t}^{2} = t(t^{2}\partial_{t})^{2} - 2t^{2}(t^{2}\partial_{t})$.

Under appropriate conditions on the building blocks of (\ref{former_main_PDE_y_intro}), we constructed a set of genuine bounded holomorphic solutions in the form of Laplace transforms of order $\kappa$ in time $t$ and Fourier inverse transform in space $z$,
$$ y_{p}(t,z,\epsilon) = \frac{\kappa}{(2\pi)^{1/2}} \int_{-\infty}^{+\infty} \int_{L_{d_p}} \omega^{d_p}(u,m,\epsilon)
\exp( -(\frac{u}{\epsilon t})^{\kappa} ) e^{izm} \frac{du}{u} dm $$
for $0 \leq p \leq \varsigma-1$, with $\varsigma \geq 2$, where $\omega^{d_p}(u,m,\epsilon)$ stands for a function with (at most)
exponential growth of order $\kappa$ containing the halfline of integration $L_{d_p} = \mathbb{R}_{+}\exp(\sqrt{-1}d_{p})$
for some well chosen directions $d_{p} \in \mathbb{R}$ and holomorphic near 0 w.r.t $u$, owning exponential decay w.r.t $m$ on $\mathbb{R}$
and relying analytically on $\epsilon$ near 0. The resulting maps $y_{p}(t,z,\epsilon)$ define bounded holomorphic functions on
domains $\mathcal{T} \times H_{\beta} \times \mathcal{E}_{p}$ for a suitable bounded sector $\mathcal{T}$ at 0 and
$\underline{\mathcal{E}} = \{ \mathcal{E}_{p} \}_{0 \leq p \leq \varsigma-1}$ is a set of sectors whose union contains a full neighborhood
of 0 in $\mathbb{C}^{\ast}$ and is called a good covering (see Definition 6). Furthermore, precise information about their
asymptotic expansions as $\epsilon$ tends to 0 is provided. Namely, all the solutions $\epsilon \mapsto y_{p}(t,z,\epsilon)$ share on
$\mathcal{E}_{p}$ a common asymptotic expansion $\hat{y}(t,z,\epsilon) = \sum_{n \geq 0} y_{n}(t,z) \epsilon^{n}$ with
bounded holomorphic coefficients $y_{n}(t,z)$ on $\mathcal{T} \times H_{\beta}$. Besides, this asymptotic expansion appears to be (at most)
of Gevrey order $1/\kappa$ (see Definition 8 for a description of this notion). In the special configuration where the aperture of
$\mathcal{E}_{p}$ can be chosen slightly larger than $\pi/\kappa$, the function $\epsilon \mapsto y_{p}(t,z,\epsilon)$ becomes the
so-called $\kappa-$sum of $\hat{y}$ on $\mathcal{E}_{p}$ as described in Definition 8.

Throughout the present study, our goal is to achieve a comparable statement, that is the construction of a set of sectorial holomorphic
solutions to (\ref{main_PDE_u_intro}) and the description of their asymptotic expansions as $\epsilon$ tends to 0 with dominated Gevrey bounds.
However, the presence of the Fuchsian operators modifies radically our approach in comparison with our previous investigation \cite{lama1}.
Indeed, according to the appearance of time differential operators of irregular type with different ranks as noticed above, we witness that
the set of solutions $u_{p}(t,z,\epsilon)$, $0 \leq p \leq \varsigma - 1$ of (\ref{main_PDE_u_intro}) (detailed later in the introduction) cannot
be built up as a single Laplace transform in time $t$ but as {\it iterated} Laplace transforms which entangle two orders $k$ and $k'$
(which can be different) that are related to the leading term in (\ref{main_PDE_u_intro}), see (\ref{rel_k_deltaD_mD_k_prime}). Moreover,
this construction has a direct effect on their asymptotic expansion w.r.t $\epsilon$ whose Gevrey bounds are sensitive to the contributions of
both irregular and Fuchsian operators and depends on the pair $(k,k')$, see Theorem 2.

A similar phenomenon has already been observed in a different context by the authors and J. Sanz in \cite{lamasa1} for some Cauchy problem
of the form (in the linear setting)
\begin{equation}
\epsilon^{r_3}(t^{2}\partial_{t})^{r_2}(z\partial_{z})^{r_1}\partial_{z}^{S}X(t,z,\epsilon) =
P(t,z,\epsilon,\partial_{t},\partial_{z})X(t,z,\epsilon) \label{former_CP_fuch_irreg_intro}
\end{equation}
for given initial Cauchy data
$$ (\partial_{z}^{j}X)(t,0,\epsilon) = \varphi_{j}(t,\epsilon) \ \ , \ \ 0 \leq j \leq S-1 $$
where $r_{1} \geq 0$, $r_{2},r_{3},S \geq 1$ are integers, $P$ stands for a polynomial and the functions $\varphi_{j}(t,\epsilon)$
are bounded holomorphic on domains $\mathcal{T} \times \mathcal{E}_{p}$, $0 \leq p \leq \varsigma-1$, for $\mathcal{T}$,$\mathcal{E}_{p}$
sectors given as above. In this context, the Fuchsian operator $(z\partial_{z})^{r_1}$ acts on the space variable $z$ near 0 in
$\mathbb{C}$ and contributes to the Gevrey order of the asymptotic expansion $\hat{X}(t,z,\epsilon) = \sum_{n \geq 0} X_{n}(t,z) \epsilon^{n}$
of the genuine holomorphic solutions $X_{p}(t,z,\epsilon)$ of (\ref{former_CP_fuch_irreg_intro}) on
$\mathcal{T} \times D(0,r) \times \mathcal{E}_{p}$ w.r.t $\epsilon$ which turns out to be $\frac{r_{1}+r_{2}}{r_3}$. Here, the mechanism of
enlargement of the Gevrey order caused by the Fuchsian operators appears through the presence of small divisors in the Borel plane.

Under proper restrictions on the shape of (\ref{main_PDE_u_intro}) detailed in the statement of Theorem 1, we can select
\begin{enumerate}
 \item[i)] a set $\underline{\mathcal{E}}$ of bounded sectors $\mathcal{E}_{p}$ as described above, which constitutes a good covering
 in $\mathbb{C}^{\ast}$ (see Definition 6),
 \item[ii)] a bounded sector $\mathcal{T}$ centered at 0,
 \item[iii)] a set of directions $\mathfrak{d}_{p} \in \mathbb{R}$, $0 \leq p \leq \varsigma-1$, chosen in a way that the halflines
 $L_{\mathfrak{d}_{p}} = \mathbb{R}_{+}\exp( \sqrt{-1} \mathfrak{d}_{p})$ bypass all the roots of the
 polynomial $u \mapsto Q(im) - R_{D}(im)k^{\delta_D}(k')^{m_D}u^{k\delta_{D}}$ whenever $m \in \mathbb{R}$,
\end{enumerate}
for which we can model a family of bounded holomorphic solutions $u_{p}(t,z,\epsilon)$ on the domains
$\mathcal{T} \times H_{\beta} \times \mathcal{E}_{p}$. Each solution $u_{p}$ is expressed as a Laplace transform of order $k$ in time $t$
and Fourier inverse integral in space $z$,
$$ u_{p}(t,z,\epsilon) = \frac{k}{(2\pi)^{1/2}}\int_{-\infty}^{+\infty} \int_{L_{\gamma_p}} W^{\mathfrak{d}_{p}}(\tau,m,\epsilon)
\exp( - (\frac{\tau}{\epsilon t})^{k} ) e^{izm} \frac{d\tau}{\tau} dm $$
where the {\it Borel/Fourier} map $W^{\mathfrak{d}_{p}}(\tau,m,\epsilon)$ is itself represented as a Laplace transform of order $k'$
in the Borel plane,
$$ W^{\mathfrak{d}_{p}}(\tau,m,\epsilon) = k'\int_{L_{\gamma_p}} w^{\mathfrak{d}_{p}}(u,m,\epsilon)
\exp( - (\frac{u}{\tau})^{k'} ) \frac{du}{u} $$
where $w^{\mathfrak{d}_{p}}(u,m,\epsilon)$ stands for an analytic function near $u=0$ with (at most) exponential growth of some
order $k_{1} < k'$ on a sector containing $L_{\mathfrak{d}_{p}}$ w.r.t $u$, suffering exponential decay w.r.t $m$ on $\mathbb{R}$, with analytic
dependence on $\epsilon$ near $\epsilon=0$ (see Theorem 1).

Furthermore, as detailed in Theorem 2, all the functions $\epsilon \mapsto u_{p}(t,z,\epsilon)$ share a common asymptotic expansion
$\hat{u}(t,z,\epsilon) = \sum_{m \geq 0} h_{m}(t,z) \epsilon^{m}$ on $\mathcal{E}_{p}$ with bounded holomorphic coefficients
$h_{m}(t,z)$ on $\mathcal{T} \times H_{\beta}$. The essential point that needs to be stressed is that this asymptotic expansion turns out
to be of Gevrey order (at most) $\frac{1}{\kappa}= \frac{1}{k} + \frac{1}{k'}$. When the aperture of one $\mathcal{E}_{p}$ can be chosen
a bit larger than $\pi/\kappa$, the map $\epsilon \mapsto u_{p}(t,z,\epsilon)$ is elected as the $\kappa-$sum of $\hat{u}$ on
$\mathcal{E}_{p}$, a configuration that can actually arise as shown in the example treated after Theorem 2.

The manner we build up our solutions as iterated Laplace transforms is known in the literature as a {\it multisummability} procedure
as described in the classical textbooks by W. Balser, \cite{ba}, \cite{ba2}. Namely, there exist three equivalent approaches to
multisummability, the first is based on acceleration kernels and goes back to the seminal works by J. \'{E}calle (see Chapter 5 of \cite{ba}),
the second, due to W. Balser, is performed through a finite number of iterations of Laplace transforms (described in Section 7.2 of \cite{ba})
and the third, known as Malgrange-Ramis approach, is based on sheaf theory aspects and is very clearly explained in Chapter 7 of the
recent lectures notes by M. Loday-Richaud, see \cite{lod}. In this paper, the second of these methods appears naturally. It is worth noticing
the two other procedures have been successfully applied by the authors to show parametric multisummability of formal solutions to
singularly perturbed equations of the shape (\ref{former_main_PDE_y_intro}) written in factorized forms, see \cite{lama2}. We observe that
in our setting (\ref{main_PDE_u_intro}), no situation of parametric multisummability w.r.t $\epsilon$ is reached for our solutions
$u_{p}(t,z,\epsilon)$.

The multisummable structure of formal solutions to linear and nonlinear ODE has been revealed two decades ago, for that we refer to some
outstanding fundamental works \cite{ba1}, \cite{babrrasi}, \cite{br}, \cite{lori}, \cite{malra}, \cite{rasi}. These last years,
applications of these notions attract a lot of attention in the framework of PDE. Not pretending to be exhaustive, we just mention
some recent references among the growing literature somehow related to our recent contributions. In the linear case of two complex variables
involving constant coefficients, we quote the important paper by W. Balser, \cite{ba4}, extended lately by interesting works by
K. Ichinobe, \cite{ich}, \cite{ich1} and S. Michalik, \cite{mi}, \cite{mi1}. In the case of general time dependent coefficients, H. Tahara
and H. Yamazawa have recently shown the multisummability for formal solutions expanded in the time variable provided that the forcing term belongs to a suitable class of entire functions with finite exponential order in the space variables, see \cite{taya}.

\noindent Our paper is organized as follows.\\
In Section 2, we state the definition of Laplace transforms of order $k$ among the positive integers and classical identities for the
Fourier inverse transform acting on exponentially decaying functions are formulated.\\
In Section 3, we present our main problem (\ref{main_PDE_u}) and display the full strategy leading to its resolution. We describe the
structure of the building blocks of (\ref{main_PDE_u}), especially the forcing term which is supposed to be assembled as iterated Laplace
transforms of functions with appropriate exponential growth. Then, in a first step, possible candidates for solutions are selected among
Laplace transforms of order $k$ and Fourier inverse integrals of Borel maps $W$ with exponential growths on large enough unbounded
sectors and with exponential decay on the real line, giving rise to an integro-differential equation (\ref{main_diff_conv_W}) that $W$ needs
to satisfy. In a second undertaking, we assume that $W$ itself is represented as a Laplace transform of suitable order $k'$ of a second Borel
map $w$ with again convenient growth on unbounded sectors and exponential decay on $\mathbb{R}$. The expression $w$ is then adjusted to solve
an integral equation (\ref{main_conv_w}).\\
In Section 4, we first analyze bounds for linear convolution operators acting on Banach spaces of analytic functions on sectors and then
we solve the main convolution problem (\ref{main_conv_w}) within these spaces by means of a fixed point argument.\\
In Section 5, leaning on the resolution of (\ref{main_conv_w}) performed in Section 4, we build up genuine holomorphic solutions $W$
of equation (\ref{main_diff_conv_W}) fulfilling the required bounds.\\
In Section 6, we provide a set of actual holomorphic solutions (\ref{Laplace_up_theo}) to our initial equation (\ref{main_PDE_u}) by realizing
rearward the two steps of constructions described in Section 3.\\
At last, in Section 7, we achieve the existence of a common asymptotic expansion of Gevrey order (at most) $\frac{1}{k}+\frac{1}{k'}$
for the set of solutions mentioned above based on the crucial flatness estimates (\ref{exp_flat_difference_up}) as an application of a theorem
by Ramis and Sibuya.

\section{Laplace transforms of order $k$ and Fourier inverse maps}

Let $k \geq 1$ be an integer. We recall the definition of the Laplace transform of order $k$ as introduced in \cite{lama1}.
\begin{defin} We set
$S_{d,\delta} = \{ \tau \in \mathbb{C}^{\ast} : |d - \mathrm{arg}(\tau)| < \delta \}$ as some unbounded sector with bisecting direction
$d \in \mathbb{R}$ and aperture $2\delta > 0$ and $D(0,\rho)$ as a disc centered at 0 with radius $\rho>0$.
Consider a holomorphic function $w : S_{d,\delta} \cup D(0,\rho) \rightarrow \mathbb{C}$ that vanishes at 0 and satisfies the bounds :
there exist $C>0$ and $K>0$ such that
\begin{equation}
|w(\tau)| \leq C |\tau| \exp( K |\tau|^{k} ) \label{w_exp_bds}
\end{equation}
for all $\tau \in S_{d,\delta}$. We define the Laplace transform of $w$ of order $k$ in the direction $d$ as the integral transform
$$ \mathcal{L}_{k}^{d}(w)(T) = k \int_{L_{\gamma}} w(u) \exp( -(\frac{u}{T})^{k} ) \frac{du}{u} $$
along a half-line $L_{\gamma} = \mathbb{R}_{+}e^{\sqrt{-1}\gamma} \subset S_{d,\delta} \cup \{ 0 \}$, where $\gamma$ depends on
$T$ and is chosen in such a way that $\cos(k(\gamma - \mathrm{arg}(T))) \geq \delta_{1}$, for some fixed real number $\delta_{1}>0$.
The function $\mathcal{L}^{d}_{k}(w)(T)$ is well defined, holomorphic and bounded on any sector
$$ S_{d,\theta,R^{1/k}} = \{ T \in \mathbb{C}^{\ast} : |T| < R^{1/k} \ \ , \ \ |d - \mathrm{arg}(T) | < \theta/2 \},$$
where $0 < \theta < \frac{\pi}{k} + 2\delta$ and
$0 < R < \delta_{1}/K$.

If one sets $w(\tau) = \sum_{n \geq 1} w_{n} \tau^n$, the Taylor expansion of $w$, which converges on the disc $D(0,\rho/2)$, the Laplace
transform $\mathcal{L}_{k}^{d}(w)(T)$ has the formal series
$$ \hat{X}(T) = \sum_{n \geq 1} w_{n} \Gamma( \frac{n}{k}) T^{n} $$
as Gevrey asymptotic expansion of order $1/k$. This means that for all $0 < \theta_{1} < \theta$, two constants $C,M>0$ can be
selected with the bounds
$$ | \mathcal{L}_{k}^{d}(w)(T) - \sum_{p=1}^{n-1} w_{p} \Gamma( \frac{p}{k} ) T^{p} | \leq CM^{n}\Gamma( 1 + \frac{n}{k}) |T|^{n} $$
for all $n \geq 2$, all $T \in S_{d,\theta_{1},R^{1/k}}$.

In particular, if $w(\tau)$ represents an entire function w.r.t $\tau \in \mathbb{C}$ with the bounds (\ref{w_exp_bds}), its Laplace transform
$\mathcal{L}_{k}^{d}(w)(T)$ does not depend on the direction $d$ in $\mathbb{R}$ and represents a bounded holomorphic function on
$D(0,R^{1/k})$ whose Taylor expansion is represented by the convergent series $X(T) = \sum_{n \geq 1} w_{n}\Gamma( \frac{n}{k} ) T^{n}$
on $D(0,R^{1/k})$.
\end{defin}

We restate the definition of some family of Banach spaces mentioned in \cite{lama1}.
\begin{defin} Let $\beta, \mu \in \mathbb{R}$. We set
$E_{(\beta,\mu)}$ as the vector space of continuous functions $h : \mathbb{R} \rightarrow \mathbb{C}$ such that
$$ ||h(m)||_{(\beta,\mu)} = \sup_{m \in \mathbb{R}} (1+|m|)^{\mu} \exp( \beta |m|) |h(m)| $$
is finite. The space $E_{(\beta,\mu)}$ endowed with the norm $||.||_{(\beta,\mu)}$ becomes a Banach space.
\end{defin}

Finally, we remind the reader the definition of the inverse Fourier transform acting on the latter Banach spaces and some of its
handy formulas relative to derivation and convolution product as stated in
\cite{lama1}.

\begin{defin}
Let $f \in E_{(\beta,\mu)}$ with $\beta > 0$, $\mu > 1$. The inverse Fourier transform of $f$ is given by
$$ \mathcal{F}^{-1}(f)(x) = \frac{1}{ (2\pi)^{1/2} } \int_{-\infty}^{+\infty} f(m) \exp( ixm ) dm $$
for all $x \in \mathbb{R}$. The function $\mathcal{F}^{-1}(f)$ extends to an analytic bounded function on the strips
\begin{equation}
H_{\beta'} = \{ z \in \mathbb{C} / |\mathrm{Im}(z)| < \beta' \}. \label{strip_H_beta}
\end{equation}
for all given $0 < \beta' < \beta$.\\
a) Define the function $m \mapsto \phi(m) = imf(m)$ which belongs to the space $E_{(\beta,\mu-1)}$. Then, the next identity
\begin{equation}
\partial_{z} \mathcal{F}^{-1}(f)(z) = \mathcal{F}^{-1}(\phi)(z) \label{Fourier_derivation}
\end{equation}
occurs.\\
b) Take $g \in E_{(\beta,\mu)}$ and set 
$$ \psi(m) = \frac{1}{(2\pi)^{1/2}} \int_{-\infty}^{+\infty} f(m-m_{1})g(m_{1}) dm_{1} $$
as the convolution product of $f$ and $g$. Then, $\psi$ belongs to $E_{(\beta,\mu)}$ and moreover,
\begin{equation}
\mathcal{F}^{-1}(f)(z) \mathcal{F}^{-1}(g)(z) = \mathcal{F}^{-1}(\psi)(z) \label{Fourier_product}
\end{equation}
for all $z \in H_{\beta}$.
\end{defin}

\section{Outline of the main initial value problem and related auxiliary problems}

We set $k \geq 1$ as an integer. Let $m_{D},\delta_{D} \geq 1$ be integers. We assume the existence of an integer $k' \geq 1$ such
that
\begin{equation}
 k \delta_{D} = m_{D}k' \label{rel_k_deltaD_mD_k_prime}
\end{equation}
We consider a finite set $I$ of $\mathbb{N}^{2}$ that fulfills the next feature,
\begin{equation}
kl_{1} \geq 1 + l_{2}k' \label{rel_kl_one_l_two_k_prime}
\end{equation}
whenever $(l_{1},l_{2}) \in I$ and we set non negative integers $\Delta_{\bf l} \geq 0$ with
\begin{equation}
 \Delta_{\bf l} - kl_{1} \geq 0 \label{rel_Deltal_k}
\end{equation}
for all ${\bf l}=(l_{1},l_{2}) \in I$.

Let $Q(X),R_{D}(X),R_{\bf l}(X) \in \mathbb{C}[X]$, ${\bf l} \in I$, be polynomials such that
\begin{equation}
\mathrm{deg}(Q) = \mathrm{deg}(R_{D}) \geq \mathrm{deg}(R_{\bf l}) \ \ , \ \ 
Q(im) \neq 0 \ \ , \ \ R_{D}(im) \neq 0 \label{constraints_degree_coeff}
\end{equation}
for all $m \in \mathbb{R}$, all ${\bf l} \in I$.

We consider a family of linear singularly perturbed initial value problems
\begin{multline}
Q(\partial_{z})u(t,z,\epsilon) = R_{D}(\partial_{z}) \epsilon^{k \delta_{D}}
(t^{k+1}\partial_{t})^{\delta_{D}}(t\partial_{t})^{m_{D}}u(t,z,\epsilon)
\\+ \sum_{{\bf l}=(l_{1},l_{2}) \in I} \epsilon^{\Delta_{{\bf l}}} c_{{\bf l}}(z,\epsilon) R_{{\bf l}}(\partial_{z})
(t^{k+1}\partial_{t})^{l_1}(t\partial_{t})^{l_2}u(t,z,\epsilon) + f(t,z,\epsilon) \label{main_PDE_u}
\end{multline}
for vanishing initial data $u(0,z,\epsilon) \equiv 0$.

The coefficients $c_{{\bf l}}(z,\epsilon)$ are built in the following manner.
For each ${\bf l} \in I$, we consider a function $m \mapsto C_{{\bf l}}(m,\epsilon)$ that belongs to the Banach space $E_{(\beta,\mu)}$
for some $\beta,\mu>0$, depends holomorphically on the parameter $\epsilon$ on some disc $D(0,\epsilon_{0})$ with radius $\epsilon_{0}>0$
and for which one can find a constant $C_{\bf l}>0$ with
\begin{equation}
\sup_{\epsilon \in D(0,\epsilon_{0})} ||C_{\bf l}(m,\epsilon)||_{(\beta,\mu)} \leq C_{\bf l} \label{defin_Cl}
\end{equation}
We construct
$$ c_{{\bf l}}(z,\epsilon) = \frac{1}{(2\pi)^{1/2}} \int_{-\infty}^{+\infty} C_{{\bf l}}(m,\epsilon) e^{izm} dm $$
as the inverse Fourier transform of the map $C_{{\bf l}}(m,\epsilon)$ for all ${\bf l} \in I$. As a result, $c_{{\bf l}}(z,\epsilon)$ is
bounded holomorphic w.r.t $\epsilon$ on $D(0,\epsilon_{0})$ and w.r.t
$z$ on any strip $H_{\beta'}$ for $0 < \beta' < \beta$ in view of Definition 3.

In order to display the forcing term, we need some preparation. We consider a sequence of functions $m \mapsto \psi_{n}(m,\epsilon)$,
for $n \geq 1$, that belong to the Banach space $E_{(\beta,\mu)}$ with the parameters $\beta,\mu>0$ given above and which relies
analytically and is bounded w.r.t $\epsilon$ on the disc $D(0,\epsilon_{0})$. We assume that the next bounds
\begin{equation}
\sup_{ \epsilon \in D(0,\epsilon_{0})} || \psi_{n}(m,\epsilon) ||_{(\beta,\mu)} \leq K_{0} (\frac{1}{T_{0}})^{n} \label{norm_bds_psi_n} 
\end{equation}
hold for all $n \geq 1$ and given constants $K_{0},T_{0}>0$. We define the formal series
$$ \psi(\tau,m,\epsilon) = \sum_{n \geq 1} \psi_{n}(m,\epsilon) \frac{\tau^{n}}{\Gamma( \frac{n}{k_1} + 1)} $$
for some integer $0 < k_{1} < k'$. According to the bounds of the Mittag-Leffler's function
$E_{\alpha}(z) = \sum_{n \geq 0} z^{n}/\Gamma(1 + \alpha n)$ for $\alpha \in (0,2)$ given in Appendix B of \cite{ba2}, we deduce 
that $\psi(\tau,m,\epsilon)$ represents an entire function w.r.t $\tau$ in $\mathbb{C}$ and we get the existence of
a constant $C_{k_1}>0$ (depending on $k_1$) such that
\begin{multline}
| \psi(\tau,m,\epsilon) | \leq K_{0} (1 + |m|)^{-\mu} \exp(-\beta |m|) \sum_{n \geq 1} (\frac{|\tau|}{T_{0}})^{n}/\Gamma( \frac{n}{k_1} + 1)\\
\leq K_{0}C_{k_1}(1 + |m|)^{-\mu} \exp( -\beta |m| ) \exp( (\frac{1}{T_{0}})^{k_1} |\tau|^{k_1} ) \label{psi_exp_bds}
\end{multline}
for all $\tau \in \mathbb{C}$, all $m \in \mathbb{R}$, all $\epsilon \in D(0,\epsilon_{0})$.

We set
$$ \Psi_{d}(\tau,m,\epsilon) = k' \int_{L_{d}} \psi(u,m,\epsilon) \exp( - (\frac{u}{\tau})^{k'} ) \frac{du}{u} $$
as the Laplace transform of $\psi(\tau,m,\epsilon)$ w.r.t $\tau$ of order $k'$ in direction $d \in \mathbb{R}$. Since
$\psi(\tau,m,\epsilon)$ defines an entire function w.r.t $\tau$ under the bounds (\ref{psi_exp_bds}), according to
Definition 1, we deduce that $\Psi_{d}$ does not depend on $d$ and can be written as a convergent series
$$ \Psi_{d}(\tau,m,\epsilon) = \sum_{n \geq 1} \psi_{n}(m,\epsilon)
\frac{\Gamma( \frac{n}{k'} )}{ \Gamma( \frac{n}{k_1} + 1)} \tau^{n} $$
From Appendix B of \cite{ba2}, we recall the Beta integral formula
\begin{equation}
B(\alpha,\beta) = \int_{0}^{1} (1-t)^{\alpha - 1} t^{\beta - 1} dt = \frac{\Gamma(\alpha) \Gamma(\beta)}{\Gamma( \alpha + \beta)} \label{Beta_integral}
\end{equation}
which is valid for all positive real numbers $\alpha,\beta > 0$. In particular, when $\alpha,\beta \geq 1$, we observe that
\begin{equation}
\Gamma(\alpha) / \Gamma( \alpha + \beta) \leq 1/\Gamma(\beta) \label{quotient_Gamma_alpha_beta}
\end{equation}
For the special case $\alpha = n/k'$ and $\beta = n(\frac{1}{k_{1}} - \frac{1}{k'}) + 1$, we obtain that
\begin{equation}
\frac{ \Gamma( \frac{n}{k'} ) }{\Gamma( \frac{n}{k_1} + 1)} \leq \frac{1}{\Gamma( n( \frac{1}{k_{1}} - \frac{1}{k'} ) + 1)}
\label{maj_quotient_Gamma_k_prime_k1}
\end{equation}
for all $n \geq k'$. In the following, we set $\kappa_{1}>1/2$ such as $\frac{1}{\kappa_{1}} = \frac{1}{k_{1}} - \frac{1}{k'}$. Again, in view of the bounds of the Mittag-Leffler's function, we deduce that $\Psi_{d}(\tau,m,\epsilon)$ represents an entire function w.r.t $\tau$ and
that there exist two constants $C_{\kappa_1},C_{\kappa_{1}}'>0$ (depending on $\kappa_{1}$) such that
\begin{multline}
 |\Psi_{d}(\tau,m,\epsilon)| \leq C_{\kappa_{1}}'K_{0} (1 + |m|)^{-\mu} e^{-\beta |m|} \sum_{n \geq 1}
 \frac{ (\frac{|\tau|}{T_{0}})^{n} }{ \Gamma( \frac{n}{\kappa_{1}} + 1) }\\
 \leq K_{0}C_{\kappa_{1}}(1 + |m|)^{-\mu} e^{-\beta |m|} \exp( (\frac{1}{T_{0}})^{\kappa_1} |\tau|^{\kappa_1} )
\end{multline}
for all $\tau \in \mathbb{C}$, all $m \in \mathbb{R}$, all $\epsilon \in D(0, \epsilon_{0})$. Let us assume that
\begin{equation}
\frac{1}{2} < \kappa_{1} \leq k \label{kappa1_less_k}
\end{equation}
In a last step, we set
$$ F_{d}(T,z,\epsilon) = \frac{k}{(2\pi)^{1/2}} \int_{-\infty}^{+\infty} \int_{L_{d}} \Psi_{d}(u,m,\epsilon)
\exp( -(\frac{u}{T})^{k} ) e^{izm} \frac{du}{u} dm
$$
as the Laplace transform of $\Psi_{d}(\tau,m,\epsilon)$ w.r.t $\tau$ of order $k$ and Fourier inverse transform w.r.t $m$. If we put
$$ F_{n}(z,\epsilon) = \frac{1}{(2\pi)^{1/2}} \int_{-\infty}^{+\infty} \psi_{n}(m,\epsilon) e^{izm} dm $$
for all $n \geq 1$, then owing to Definition 1, we notice that $F_{d}(T,z,\epsilon)$ can be written as a formal series
\begin{equation}
F_{d}(T,z,\epsilon) = \sum_{n \geq 1} F_{n}(z,\epsilon)
\frac{ \Gamma(\frac{n}{k}) \Gamma( \frac{n}{k'} )}{\Gamma( \frac{n}{k_1} + 1)} T^{n} 
\end{equation}
As a result, we see that $F_{d}$ does not depend on the direction $d$. We can provide bounds for $F_{n}(z,\epsilon)$ and get a constant $C_{\mu,\beta,\beta'}>0$ (depending on $\mu,\beta,\beta'$) with
\begin{multline}
|F_{n}(z,\epsilon)| \leq \frac{K_0}{(2\pi)^{1/2}} (\frac{1}{T_0})^{n} \int_{-\infty}^{+\infty} (1 + |m|)^{-\mu} e^{-\beta|m|}
e^{-\mathrm{Im}(z) m} dm \\
\leq \frac{K_0}{(2\pi)^{1/2}} (\frac{1}{T_0})^{n} \int_{-\infty}^{+\infty} (1 + |m|)^{-\mu}
e^{(\beta' - \beta)|m|} dm \leq \frac{C_{\mu,\beta,\beta'} K_{0}}{(2\pi)^{1/2}} (\frac{1}{T_0})^{n}
\end{multline}
for all $n \geq 1$, whenever $\epsilon \in D(0,\epsilon_{0})$ and $z$ belongs to the horizontal strip $H_{\beta'}$ for some $0 < \beta' < \beta$ (see Definition 3).
Bearing in mind (\ref{maj_quotient_Gamma_k_prime_k1}), we deduce a constant $C_{\kappa_{1}}'>0$ with
\begin{equation}
|F_{d}(T,z,\epsilon)| \leq \sum_{n \geq 1} \frac{C_{\kappa_{1}}'C_{\mu,\beta,\beta'} K_{0}}{(2\pi)^{1/2}}
\frac{\Gamma( \frac{n}{k} )}{ \Gamma( \frac{n}{\kappa_{1}} + 1) } (\frac{|T|}{T_0})^{n}
\end{equation}
In the case $\kappa_{1}=k$, we remark in particular that $F_{d}(T,z,\epsilon)$ is a convergent series on $D(0,T_{0}/2)$
w.r.t $T$, and defines a bounded holomorphic function w.r.t $z$ on $H_{\beta'}$ and w.r.t $\epsilon$ on $D(0,\epsilon_{0})$. On the other hand,
when $0 < \kappa_{1} < k$, we apply the inequality (\ref{quotient_Gamma_alpha_beta}) in the particular case $\alpha = n/k$ and
$\beta = n( \frac{1}{\kappa_{1}} - \frac{1}{k}) + 1$ and set $\kappa_{2}>1/2$ with $\frac{1}{\kappa_{2}} = \frac{1}{\kappa_1} - \frac{1}{k}$
in order to get
$$ \frac{\Gamma(\frac{n}{k})}{\Gamma( \frac{n}{k_1} + 1)} \leq \frac{1}{\Gamma( \frac{n}{\kappa_{2}} + 1)} $$
for all $n \geq k$. Again, calling back the bounds for the Mittag-Leffler's function, we deduce that $F_{d}(T,z,\epsilon)$
defines an entire function w.r.t $T$ with two constants $C_{\kappa_2},C_{\kappa_{2}}'>0$ such that
\begin{equation}
|F_{d}(T,z,\epsilon)| \leq \frac{C_{\kappa_{2}}'C_{\mu,\beta,\beta'} K_{0}}{(2\pi)^{1/2}}
\sum_{n \geq 1} \frac{1}{\Gamma( \frac{n}{\kappa_2} + 1)}
(\frac{|T|}{T_0})^{n} \leq \frac{C_{\kappa_2}C_{\mu,\beta,\beta'} K_{0}}{(2\pi)^{1/2}} \exp( (\frac{1}{T_0})^{\kappa_2} |T|^{\kappa_2} )
\end{equation}
for all $T \in \mathbb{C}$, all $z \in H_{\beta'}$ and $\epsilon \in D(0,\epsilon_{0})$.

Finally, we set the forcing term $f$ as a time rescaled version of $F_{d}$, namely
$$ f(t,z,\epsilon) = F_{d}(\epsilon t,z,\epsilon) $$
which defines a bounded holomorphic function on $D(0,r) \times H_{\beta'} \times D(0,\epsilon_{0})$ for any given
$0 < \beta' < \beta$ and radius $r>0$ such that $\epsilon_{0}r \leq T_{0}/2$ when $\kappa_{1}=k$ and represents an
entire function w.r.t $t$ provided that $0 < \kappa_{1} < k$.

Within this work, we are looking for time rescaled solutions of (\ref{main_PDE_u}) of the form
$$ u(t,z,\epsilon) = U(\epsilon t,z,\epsilon) $$
As a consequence, the expression $U(T,z,\epsilon)$, through the change of variable $T=\epsilon t$, is asked to solve the next
singular problem
\begin{multline}
Q(\partial_{z})U(T,z,\epsilon) = R_{D}(\partial_{z})(T^{k+1}\partial_{T})^{\delta_D}(T\partial_{T})^{m_D}U(T,z,\epsilon)\\
+ \sum_{{\bf l} = (l_{1},l_{2}) \in I} \epsilon^{\Delta_{{\bf l}} - kl_{1}}c_{{\bf l}}(z,\epsilon)R_{{\bf l}}(\partial_{z})
(T^{k+1}\partial_{T})^{l_1}(T\partial_{T})^{l_2}U(T,z,\epsilon) + F_{d}(T,z,\epsilon)   \label{main_PDE_U}
\end{multline}

We now recall the definition of Banach spaces already introduced in the paper \cite{lama2}.
\begin{defin}
Let $S_{d}$ be an unbounded sector centered at 0 with bisecting direction $d \in \mathbb{R}$. Let $\nu_{1},\beta,\mu,\kappa_{1}>0$
be positive real numbers. We set $E_{(\nu_{1},\beta,\mu,\kappa_{1})}^{d}$ as the vector space of continuous functions
$(\tau,m) \mapsto h(\tau,m)$ on $S_{d} \times \mathbb{R}$, which are holomorphic w.r.t $\tau$ on $S_{d}$ such that
$$ ||h(\tau,m)||_{(\nu_{1},\beta,\mu,\kappa_{1})} = \sup_{\tau \in S_{d},m \in \mathbb{R}} (1 + |m|)^{\mu} \frac{1}{|\tau|}
e^{\beta |m| - \nu_{1} |\tau|^{\kappa_1}} |h(\tau,m)| $$
is finite. The space $E_{(\nu_{1},\beta,\mu,\kappa_{1})}^{d}$ endowed with the norm $||.||_{(\nu_{1},\beta,\mu,\kappa_{1})}$ is a Banach space.
\end{defin}

In a {\bf first} step, we search for solutions $U(T,z,\epsilon)$ that can be expressed similarly to $F_{d}(T,z,\epsilon)$ as
integral representations through Laplace transforms of order $k$ and Fourier inverse transforms
$$ U_{\gamma}(T,z,\epsilon) = \frac{k}{(2\pi)^{1/2}} \int_{-\infty}^{+\infty} \int_{L_{\gamma}}
W(u,m,\epsilon) \exp( -(\frac{u}{T})^{k} ) e^{izm} \frac{du}{u} dm $$
where $L_{\gamma} = \mathbb{R}_{+}e^{\sqrt{-1}\gamma}$ stands for a halfline with direction $\gamma \in \mathbb{R}$
which belongs to the set $S_{d} \cup \{ 0 \}$ where $S_{d}$ represents a sector as given above in Definition 4.

Our target is the statement of a related problem fulfilled by the expression $W(u,m,\epsilon)$. Overall this section,
we assume that for all $\epsilon \in D(0,\epsilon_{0})$, the function $(\tau,m) \mapsto W(\tau,m,\epsilon)$ belongs to the Banach space
$E_{(\nu_{1},\beta,\mu,\kappa_{1})}^{d}$, where the constants $\beta,\mu,\kappa_{1}$ are fixed in the description of
the forcing term $f(t,z,\epsilon)$ given above and $\nu_{1}>0$ is some real number larger than $(1/T_{0})^{\kappa_1}$ (that will
be suitably chosen later on in Section 5).

We display some identities related to the action of differential operators
of irregular and fuchsian types.

\begin{lemma} The actions of the differential operators $T^{k+1}\partial_{T}$ and $T\partial_{T}$ on $U_{\gamma}$ are given by
\begin{multline}
T^{k+1}\partial_{T}U_{\gamma}(T,z,\epsilon) =
\frac{k}{(2\pi)^{1/2}} \int_{-\infty}^{+\infty} \int_{L_{\gamma}} ku^{k}W(u,m,\epsilon)
\exp( - (\frac{u}{T})^{k} ) e^{izm} \frac{du}{u} dm, \\
T\partial_{T}U_{\gamma}(T,z,\epsilon) = \frac{k}{(2\pi)^{1/2}} \int_{-\infty}^{+\infty} \int_{L_{\gamma}}
u\partial_{u}W(u,m,\epsilon) \exp( - (\frac{u}{T})^{k} ) e^{izm} \frac{du}{u} dm
\label{TkpartialTUgamma_andfuchsian}
\end{multline}
\end{lemma}
\begin{proof} The first identity is a direct consequence of derivation under the integral symbol w.r.t $T$. We now deal with the second formula.
Namely, by derivation under the integral followed by an integration by parts, we obtain
\begin{multline*}
T\partial_{T}U_{\gamma}(T,z,\epsilon) = \frac{k}{(2\pi)^{1/2}} \int_{-\infty}^{+\infty} \int_{L_{\gamma}}
\frac{ku^{k-1}}{T^{k}}W(u,m,\epsilon) \exp( - (\frac{u}{T})^{k} ) e^{izm} du dm\\
= \frac{k}{(2\pi)^{1/2}} \int_{-\infty}^{+\infty} [ -\exp( -(\frac{u}{T})^{k} )W(u,m,\epsilon) ]_{u=0}^{u=\infty} e^{izm} dm\\
+ \frac{k}{(2\pi)^{1/2}} \int_{-\infty}^{+\infty} \int_{L_{\gamma}}
\partial_{u}W(u,m,\epsilon) \exp( - (\frac{u}{T})^{k} ) e^{izm} du dm
\end{multline*}
which yields the announced formula in (\ref{TkpartialTUgamma_andfuchsian}) since $W(u,m,\epsilon)$ is vanishing at $u=0$ and possesses
an exponential growth of order at most $\kappa_{1} \leq k$ w.r.t $u$.
\end{proof}

By virtue of the formulas (\ref{TkpartialTUgamma_andfuchsian}), together with
(\ref{Fourier_derivation}) and (\ref{Fourier_product}), we are now in position to state the first main integro-differential equation fulfilled by the expression $W(\tau,m,\epsilon)$ provided that $U_{\gamma}(T,z,\epsilon)$ solves (\ref{main_PDE_U}), namely
\begin{multline}
Q(im)W(\tau,m,\epsilon) = R_{D}(im)( k \tau^{k})^{\delta_{D}} (\tau \partial_{\tau})^{m_D}W(\tau,m,\epsilon) \\
+
\sum_{ {\bf l} = (l_{1},l_{2}) \in I} \epsilon^{\Delta_{{\bf l}} -kl_{1}} \frac{1}{(2\pi)^{1/2}}
\int_{-\infty}^{+\infty} C_{{\bf l}}(m-m_{1},\epsilon) (k \tau^{k})^{l_1} (\tau \partial_{\tau})^{l_2}R_{{\bf l}}(im_{1})
W(\tau,m_{1},\epsilon) dm_{1} \\
+ \Psi_{d}(\tau,m,\epsilon) \label{main_diff_conv_W}
\end{multline}

In a {\bf second} step, we seek for solutions of the previous equation (\ref{main_diff_conv_W}) in the form of a Laplace transform of order $k'$ as it is the case for its forcing term $\Psi_{d}(\tau,m,\epsilon)$. We first need to
introduce some Banach spaces that are similar to those provided in Definition 4 except that the functions are furthermore bounded holomorphic
on some disc centered at the origin w.r.t the first variable.
\begin{defin}
Let $U_{d}$ denote an unbounded sector centered at 0 with bisecting direction $d \in \mathbb{R}$ and let
$D(0,r)$ be the disc of radius $r>0$ centered at 0.
Let $\nu_{2},\beta,\mu,k_{1}>0$
be positive real numbers. We set $F_{(\nu_{2},\beta,\mu,k_{1})}^{d}$ as the vector space of continuous functions
$(u,m) \mapsto h(u,m)$ on $(U_{d} \cup D(0,r)) \times \mathbb{R}$, which are holomorphic w.r.t $u$ on
$U_{d} \cup D(0,r)$ such that
$$ ||h(u,m)||_{(\nu_{2},\beta,\mu,k_{1})} = \sup_{u \in U_{d} \cup D(0,r),m \in \mathbb{R}} (1 + |m|)^{\mu} \frac{1}{|u|}
e^{\beta |m| - \nu_{2} |u|^{k_1}} |h(u,m)| $$
is finite. The space $F_{(\nu_{2},\beta,\mu,k_{1})}^{d}$ equipped with the norm
$||.||_{(\nu_{2},\beta,\mu,k_{1})}$ is a Banach space.
\end{defin}

In the following, we assume that
\begin{equation}
W(\tau,m,\epsilon) = k' \int_{L_{\gamma}} w(u,m,\epsilon) \exp( -(\frac{u}{\tau})^{k'} ) \frac{du}{u}
\end{equation}
where $L_{\gamma} = \mathbb{R}_{+}e^{\sqrt{-1}\gamma}$ stands for a halfline with direction $\gamma \in \mathbb{R}$
which belongs to $U_{d} \cup \{ 0 \}$ that represents an unbounded sector centered at 0 with bisecting direction $d$. We take for granted that
for all $\epsilon \in D(0,\epsilon_{0})$ the function $(u,m) \mapsto w(u,m,\epsilon)$ appertains to the Banach space
$F^{d}_{(\nu_{2},\beta,\mu,k_{1})}$, where the constants $\beta,\mu,k_{1}$ are set throughout the construction of the forcing term
$f(t,z,\epsilon)$ stated overhead and where $\nu_{2} = (1/T_{0})^{k_1}$.

As in Lemma 1 overhead, we present some formulas related to the action of differential opertors of irregular type and multiplication by
monomials
\begin{lemma} 1) The action of the differential operators $\tau^{k'+1}\partial_{\tau}$ on $W(\tau,m,\epsilon)$ is given by
\begin{equation}
\tau^{k'+1}\partial_{\tau}W(\tau,m,\epsilon) = k' \int_{L_{\gamma}} k'u^{k'}w(u,m,\epsilon)
\exp( -(\frac{u}{\tau})^{k'} ) \frac{du}{u}
\end{equation}
2) Let $m' \geq 1$ be an integer. The action of the multiplication by $\tau^{m'}$ on $W(\tau,m,\epsilon)$ is described through the next formula
\begin{equation}
\tau^{m'}W(\tau,m,\epsilon) = k' \int_{L_{\gamma}} \left( \frac{u^{k'}}{\Gamma(\frac{m'}{k'})}\int_{0}^{u^{k'}}
(u^{k'} -s)^{\frac{m'}{k'}-1} w(s^{1/k'},m,\epsilon) \frac{ds}{s} \right) \exp( -(\frac{u}{\tau})^{k'} ) \frac{du}{u} 
\end{equation}

\end{lemma}
\begin{proof} The first formula follows by mere derivation under the integral symbol and the proof of the second identity is similar to the one given in Lemma 2 of \cite{lama3} and will not be reproduced here.
 \end{proof}

We propose to display another related problem satisfied by the expression $w(\tau,m,\epsilon)$. We first need to recast the
equation (\ref{main_diff_conv_W}) in a well prepared form. For that purpose, the next lemma will be essential.

\begin{lemma} For all integers $l \geq 1$, there exist positive integers $a_{q,l} \geq 1$, $1 \leq q \leq l$ such that
\begin{equation}
(t\partial_{t})^{l} = \sum_{q=1}^{l} a_{q,l} t^{q}\partial_{t}^{q}
\end{equation}
\end{lemma}
\begin{proof} The above identity is obtained by induction on $l$ and one observes in particular that the sequence
$(a_{q,l})_{1 \leq q \leq l, l \geq 1}$ satisfies the recursion
$$ a_{q,l+1} = qa_{q,l} + a_{q-1,l} $$
for all $2 \leq q \leq l$ provided that $l \geq 2$ and that $a_{1,l}=a_{l,l}=1$ for all $l \geq 1$.
\end{proof}
As a result, Equation (\ref{main_diff_conv_W}) can be rephrased in the form
\begin{multline}
Q(im)W(\tau,m,\epsilon) = R_{D}(im) k^{\delta_{D}} \sum_{q=1}^{m_{D}} a_{q,m_{D}} \tau^{k \delta_{D} + q} \partial_{\tau}^{q}W(\tau,m,\epsilon)
\\+ \sum_{{\bf l} = (l_{1},l_{2}) \in I} \epsilon^{\Delta_{{\bf l}} -kl_{1}}
\frac{1}{(2\pi)^{1/2}} \int_{-\infty}^{+\infty} C_{{\bf l}}(m-m_{1},\epsilon) k^{l_1} R_{{\bf l}}(im_{1})
\sum_{q=1}^{l_2} a_{q,l_{2}} \tau^{kl_{1}+q} \partial_{\tau}^{q}W(\tau,m_{1},\epsilon) dm_{1} \\
+ \Psi_{d}(\tau,m,\epsilon) \label{main_diff_conv_1}
\end{multline}
We further need to expand the above expression in order to be able to apply the lemma 2 and deduce some integral problem fulfilled by
$w$. The next crucial lemma restates the formula (8.7) p. 3630 from \cite{taya}.
\begin{lemma} Let $k',\delta \geq 1$ be integers. Then, there exit real numbers $A_{\delta,p}$, $1 \leq p \leq \delta-1$ such that
\begin{equation}
\tau^{\delta(k'+1)} \partial_{\tau}^{\delta} = (\tau^{k'+1}\partial_{\tau})^{\delta}
+ \sum_{1 \leq p \leq \delta - 1} A_{\delta,p} \tau^{k'(\delta-p)} (\tau^{k'+1}\partial_{\tau})^{p} \label{Tahara_formula}
\end{equation}
By convention, we take for granted that the above sum $\sum_{1 \leq p \leq \delta-1} [..]$ vanishes when $\delta=1$.
\end{lemma}
Owing to our hypothesis (\ref{rel_k_deltaD_mD_k_prime}), we can rewrite
\begin{equation}
k\delta_{D} + m_{D} = m_{D}(1+k') \label{expand_k_deltaD_mD} 
\end{equation}
which implies also the next expansion
\begin{equation}
k\delta_{D} + q = q(1 + k') + d_{q,k,D} \label{expand_k_deltaD_q}
\end{equation}
where $d_{q,k,D} = k\delta_{D} - qk'=(m_{D}-q)k' \geq 1$, whenever $1 \leq q \leq m_{D}-1$. Besides, according to our assumption
(\ref{rel_kl_one_l_two_k_prime}) on the set $I$, we can represent the next integers
\begin{equation}
kl_{1} + q = q(1 + k') + e_{q,k,l_{1}} \label{expand_k_l1_q}
\end{equation}
in a specific way where $e_{q,k,l_{1}} = kl_{1} - qk' \geq 1$ for all $(l_{1},l_{2}) \in I$ and $1 \leq q \leq l_{2}$.

Owing to these expansions (\ref{expand_k_deltaD_mD}), (\ref{expand_k_deltaD_q}) and (\ref{expand_k_l1_q}), the lemma 4 allows us to
expand each piece of the equation (\ref{main_diff_conv_1}) in a final prepared form, namely

\begin{multline}
\tau^{k \delta_{D} + m_{D}} \partial_{\tau}^{m_D} W(\tau,m,\epsilon) =
\tau^{m_{D}(1+k')} \partial_{\tau}^{m_D} W(\tau,m,\epsilon) \\
=
\left( (\tau^{k'+1}\partial_{\tau})^{m_D} + \sum_{1 \leq p \leq m_{D}-1} A_{m_{D},p} \tau^{k'(m_{D}-p)} (\tau^{k'+1}\partial_{\tau})^{p} \right)
W(\tau,m,\epsilon) \label{Tahara_form_1}
\end{multline}
together with
\begin{multline}
\tau^{k \delta_{D} + q} \partial_{\tau}^{q} W(\tau,m,\epsilon) =
\tau^{d_{q,k,D}} \tau^{q(1+k')} \partial_{\tau}^{q}W(\tau,m,\epsilon) \\
=
\tau^{d_{q,k,D}} \left( (\tau^{k'+1}\partial_{\tau})^{q} + \sum_{1 \leq p \leq q-1} A_{q,p} \tau^{k'(q-p)}
(\tau^{k'+1}\partial_{\tau})^{p} \right)
W(\tau,m,\epsilon) \label{Tahara_form_2}
\end{multline}
for $1 \leq q \leq m_{D}-1$ and 
\begin{multline}
\tau^{kl_{1}+q}\partial_{\tau}^{q}W(\tau,m_{1},\epsilon) = \tau^{e_{q,k,l_{1}}} \tau^{q(1+k')} \partial_{\tau}^{q}W(\tau,m_{1},\epsilon)
\\
= \tau^{e_{q,k,l_{1}}} \left( (\tau^{k'+1}\partial_{\tau})^{q} + \sum_{1 \leq p \leq q-1} A_{q,p}
\tau^{k'(q-p)} (\tau^{k'+1}\partial_{\tau})^{p} \right)W(\tau,m_{1},\epsilon) \label{Tahara_form_3}
\end{multline}
for $1 \leq q \leq l_{2}$ when $(l_{1},l_{2}) \in I$.
 
Henceforth, we can rework the equation (\ref{main_diff_conv_1}) in its final suitable form for further computations. Namely,
\begin{multline}
Q(im)W(\tau,m,\epsilon) \\
= R_{D}(im) k^{\delta_D} a_{m_{D},m_{D}} \left( (\tau^{k'+1}\partial_{\tau})^{m_D} +
\sum_{1 \leq p \leq m_{D}-1} A_{m_{D},p} \tau^{k'(m_{D}-p)} (\tau^{k'+1}\partial_{\tau})^{p} \right)W(\tau,m,\epsilon)\\
+ R_{D}(im)k^{\delta_D} \sum_{q=1}^{m_{D}-1} a_{q,m_{D}} \tau^{d_{q,k,D}} \left( (\tau^{k'+1}\partial_{\tau})^{q}
+ \sum_{1 \leq p \leq q-1} A_{q,p} \tau^{k'(q-p)} (\tau^{k'+1}\partial_{\tau})^{p} \right)W(\tau,m,\epsilon)\\
+ \sum_{{\bf l} = (l_{1},l_{2}) \in I} \epsilon^{\Delta_{{\bf l}} - kl_{1}} \frac{1}{(2\pi)^{1/2}}
\int_{-\infty}^{+\infty} C_{{\bf l}}(m-m_{1},\epsilon) k^{l_1} R_{{\bf l}}(im_{1}) \\
\times \sum_{q=1}^{l_2} a_{q,l_{2}}
\tau^{e_{q,k,l_{1}}} \left( (\tau^{k'+1}\partial_{\tau})^{q} + \sum_{1 \leq p \leq q-1} A_{q,p} \tau^{k'(q-p)} (\tau^{k'+1}\partial_{\tau})^{p}
\right) W(\tau,m_{1},\epsilon) dm_{1} + \Psi_{d}(\tau,m,\epsilon) \label{main_diff_conv_prep_form}
\end{multline}
Owing to Lemma 2, we are now ready to state the main integral equation that shall fulfill the expression $w(u,m,\epsilon)$ provided that
$W(\tau,m,\epsilon)$ solves the integro-differential equation presented earlier (\ref{main_diff_conv_W})
\begin{multline}
Q(im)w(u,m,\epsilon) = R_{D}(im)k^{\delta_D}(k'u^{k'})^{m_D} w(u,m,\epsilon)
\\
+ \sum_{1 \leq p \leq m_{D}-1} R_{D}(im) k^{\delta_D} A_{m_{D},p}
\frac{u^{k'}}{\Gamma( \frac{k'(m_{D}-p)}{k'} )} \int_{0}^{u^{k'}} (u^{k'}-s)^{\frac{k'(m_{D}-p)}{k'} - 1}
(k')^{p} s^{p} w(s^{1/k'},m,\epsilon) \frac{ds}{s}\\
+ R_{D}(im) k^{\delta_D} \sum_{q=1}^{m_{D}-1} a_{q,m_{D}} \left( \frac{u^{k'}}{\Gamma(\frac{d_{q,k,D}}{k'})}
\int_{0}^{u^{k'}} (u^{k'} - s)^{\frac{d_{q,k,D}}{k'} - 1} (k')^{q} s^{q} w(s^{1/k'},m,\epsilon) \frac{ds}{s} \right.
\\
\left. + \sum_{1 \leq p \leq q-1} A_{q,p} \frac{u^{k'}}{\Gamma( \frac{d_{q,k,D} + k'(q-p)}{k'} )}
\int_{0}^{u^{k'}} (u^{k'} - s)^{\frac{d_{q,k,D} + k'(q-p)}{k'} - 1} (k')^{p} s^{p} w(s^{1/k'},m,\epsilon) \frac{ds}{s} \right)\\
+ \sum_{{\bf l} = (l_{1},l_{2}) \in I} \epsilon^{\Delta_{{\bf l}} - kl_{1}} \frac{1}{(2\pi)^{1/2}}
\int_{-\infty}^{+\infty} C_{{\bf l}}(m-m_{1},\epsilon) k^{l_1} R_{{\bf l}}(im_{1}) \\
\times \sum_{q=1}^{l_2} a_{q,l_{2}}
\left( \frac{u^{k'}}{\Gamma(\frac{e_{q,k,l_{1}}}{k'})} \int_{0}^{u^{k'}} (u^{k'} - s)^{\frac{e_{q,k,l_{1}}}{k'}-1}
(k')^{q} s^{q} w(s^{1/k'},m_{1},\epsilon) \frac{ds}{s} \right.\\
\left. + \sum_{1 \leq p \leq q-1} A_{q,p} \frac{u^{k'}}{\Gamma( \frac{e_{q,k,l_{1}} + k'(q-p)}{k'} )}
\int_{0}^{u^{k'}} (u^{k'} - s)^{\frac{e_{q,k,l_{1}} + k'(q-p)}{k'} - 1} (k')^{p} s^{p} w(s^{1/k'},m_{1},\epsilon) \frac{ds}{s} \right) dm_{1}\\
+ \psi(u,m,\epsilon) \label{main_conv_w}
\end{multline}
 
\section{Construction of solutions to an accessory integral equation relying in a complex parameter}

The main goal of this section is the manufacturing of a unique solution of the latter equation (\ref{main_conv_w}) for vanishing initial data
within the Banach spaces presented in Definition 5.

The next two propositions analyze the continuity of linear convolutions operators acting on the prior Banach spaces.

\begin{prop} Let $k' \geq 1$ be an integer and $\gamma_{1} > 0$, $\gamma_{2},\gamma_{3}$ be real numbers submitted to the next
assumption
\begin{multline}
\gamma_{2}+1 > 0 \ \ , \ \ \gamma_{3} + \frac{1}{k'} + 1 > 0 \ \ , \ \ \gamma_{2}+\gamma_{3}+2 \geq 0 \ \ , \ \
\gamma_{1} \geq k'(\gamma_{2}+\gamma_{3}+2) - k_{1}(\gamma_{2}+1) \label{cond_gamma23_k_prime} 
\end{multline}
We consider a function $(u,m) \mapsto f(u,m)$ that belongs to $F_{(\nu_{2},\beta,\mu,k_{1})}^{d}$ and a continuous function
$a_{\gamma_{1}}(u,m)$ on $(\bar{U}_{d} \cup \bar{D}(0,r)) \times \mathbb{R}$, holomorphic w.r.t $u$ on
$U_{d} \cup D(0,r)$ with the bounds
\begin{equation}
|a_{\gamma_{1}}(u,m)| \leq \frac{1}{(1 + |u|)^{\gamma_{1}}} \label{bds_a_gamma1k} 
\end{equation}
for all $u \in \bar{U}_{d} \cup \bar{D}(0,r)$, all $m \in \mathbb{R}$.

We set the next convolution operator
\begin{equation}
\mathcal{C}_{k',\gamma_{1},\gamma_{2},\gamma_{3}}(f)(u,m) = a_{\gamma_{1}}(u,m)u^{k'}\int_{0}^{u^{k'}}
(u^{k'} - s)^{\gamma_{2}} s^{\gamma_{3}} f(s^{\frac{1}{k'}},m) ds
\end{equation}
Then, the linear map $f \mapsto \mathcal{C}_{k',\gamma_{1},\gamma_{2},\gamma_{3}}(f)$ is continuous from the Banach space
$F_{(\nu_{2},\beta,\mu,k_{1})}^{d}$ into itself. In other words, a constant $C_{1}>0$ (depending on
$\nu_{2},\gamma_{1},\gamma_{2},\gamma_{3},k',k_{1},r$) can be chosen with
\begin{equation}
|| \mathcal{C}_{k',\gamma_{1},\gamma_{2},\gamma_{3}}(f)(u,m) ||_{(\nu_{2},\beta,\mu,k_{1})} \leq C_{1}
||f(u,m)||_{(\nu_{2},\beta,\mu,k_{1})} \label{continuity_conv_C_kprime_gamma} 
\end{equation}
for all $f \in F_{(\nu_{2},\beta,\mu,k_{1})}^{d}$.
\end{prop}

\begin{proof} The lines of arguments are akin to those appearing in the proof of Proposition 1 of \cite{lama2}. However, we provide
a detailed proof in order to explain fully the conditions imposed in (\ref{cond_gamma23_k_prime}).

First, let $f$ belong to $F_{(\nu_{2},\beta,\mu,k_{1})}^{d}$. We can rewrite $\mathcal{C}_{k',\gamma_{1},\gamma_{2},\gamma_{3}}(f)(u,m)$
using the parametrization $s=u^{k'}p$ for $0 \leq p \leq 1$.
Namely,
\begin{equation}
\mathcal{C}_{k',\gamma_{1},\gamma_{2},\gamma_{3}}(f)(u,m) = a_{\gamma_{1}}(u,m) u^{k'(\gamma_{2}+\gamma_{3}+2)}
\int_{0}^{1}(1 -p)^{\gamma_{2}}p^{\gamma_{3}} f(u p^{1/k'},m) dp \label{conv_C_k_gamma_param_repres}
\end{equation}
for all $u \in U_{d} \cup D(0,r)$, whenever $m \in \mathbb{R}$. Under the third constraint in (\ref{cond_gamma23_k_prime}), according to
the claim that $a_{\gamma_{1}}(u,m)$ is holomorphic on $U_{d} \cup D(0,r)$ w.r.t $u$, continuous on the adherence
$(\bar{U}_{d} \cup \bar{D}(0,r)) \times \mathbb{R}$ and the fact that $f(u,m)$ is holomorphic on
$U_{d} \cup D(0,r)$ w.r.t $u$ and continuous relatively to $m$ on
$\mathbb{R}$, the map $(u,m) \mapsto \mathcal{C}_{k',\gamma_{1},\gamma_{2},\gamma_{3}}(f)(u,m)$ inherits the same feature on
$(U_{d} \cup D(0,r)) \times \mathbb{R}$. Furthermore, we can provide local sharp bounds when $u$ stays in the disc $D(0,r)$. Indeed, since
$f$ belongs to $F_{(\nu_{2},\beta,\mu,k_{1})}^{d}$, we observe in particular that the next estimates
\begin{equation}
 |f(u,m)| \leq ||f(u,m)||_{(\nu_{2},\beta,\mu,k_{1})} \exp(\nu_{2}|u|^{k_1}) |u| (1 + |m|)^{-\mu} \exp(-\beta |m|)
\end{equation}
hold for all $u \in D(0,r)$, all $m \in \mathbb{R}$. Consequently, owing to the representation (\ref{conv_C_k_gamma_param_repres}) and keeping
in mind the Beta function formula (\ref{Beta_integral}), it follows
\begin{multline}
|\mathcal{C}_{k',\gamma_{1},\gamma_{2},\gamma_{3}}(f)(u,m)| \leq |a_{\gamma_{1}}(u,m)| |u|^{k'(\gamma_{2}+\gamma_{3}+2)}
|u| \int_{0}^{1} (1 -p)^{\gamma_{2}} p^{\gamma_{3} + \frac{1}{k'}} dp
||f(u,m)||_{(\nu_{2},\beta,\mu,k_{1})}\\
\times \exp( \nu_{2} |u|^{k_1} ) (1 + |m|)^{-\mu} \exp(-\beta |m|) \leq
\frac{\Gamma(\gamma_{2}+1) \Gamma(\gamma_{3} + \frac{1}{k'} + 1)}{\Gamma( \gamma_{2}+\gamma_{3} + \frac{1}{k'} + 2)}\\
\times
\sup_{u \in D(0,r),m \in \mathbb{R}} |a_{\gamma_{1}}(u,m)| r^{k'(\gamma_{2}+\gamma_{3}+2)}
||f(u,m)||_{(\nu_{2},\beta,\mu,k_{1})} |u| \exp( \nu_{2} |u|^{k_1}) (1 + |m|)^{-\mu} \exp(-\beta |m|) \label{conv_C_k_gamma_bds_near_origin}
\end{multline}
for all $u \in D(0,r)$, all $m \in \mathbb{R}$.

In a second step, we focus on the global behaviour of $\mathcal{C}_{k',\gamma_{1},\gamma_{2},\gamma_{3}}(f)$ on the domain
$U_{d} \times \mathbb{R}$. Since $f$ is taken within $F_{(\nu_{2},\beta,\mu,k_{1})}^{d}$, we get especially that
\begin{equation}
|f(u,m)| \leq ||f(u,m)||_{(\nu_{2},\beta,\mu,k_{1})} |u| \exp( \nu_{2} |u|^{k_1}) (1 + |m|)^{-\mu} \exp(- \beta |m|) 
\end{equation}
for all $u \in U_{d}$, all $m \in \mathbb{R}$. As a result, we deduce from the very definition of the convolution operator together with the assumption
(\ref{bds_a_gamma1k}) that
\begin{multline}
|\mathcal{C}_{k',\gamma_{1},\gamma_{2},\gamma_{3}}(f)(u,m)| \leq ||f(u,m)||_{(\nu_{2},\beta,\mu,k_{1})}
\frac{|u|^{k'}}{(1 + |u|)^{\gamma_1}} \\
\times \int_{0}^{|u|^{k'}} (|u|^{k'} - h)^{\gamma_2} h^{\gamma_{3} + \frac{1}{k'}}
\exp( \nu_{2} h^{\frac{k_1}{k'}} ) dh (1 + |m|)^{-\mu} e^{-\beta |m|}
\end{multline}
for all $u \in U_{d}$, all $m \in \mathbb{R}$. We consider the function
$$ B(x) = \int_{0}^{x} (x-h)^{\gamma_{2}} h^{\gamma_{3} + \frac{1}{k'}} \exp( \nu_{2}h^{\frac{k_1}{k'}} ) dh $$
The procedure that will lead to upper estimates for this function is similar to the one performed in the proof of Proposition 1 of
\cite{lama2}. Indeed, according to the uniform expansion
$$ \exp( \nu_{2} h^{\frac{k_1}{k'}} ) = \sum_{n \geq 0} \frac{(\nu_{2}h^{\frac{k_{1}}{k'}})^{n}}{n!} $$
on every compact interval $[0,x]$, $x > 0$, we can write
$$ B(x) = \sum_{n \geq 0} \frac{\nu_{2}^{n}}{n!} \int_{0}^{x} h^{\gamma_{3} + \frac{1}{k'} + n\frac{k_1}{k'}}
(x-h)^{\gamma_2} dh $$
Using the Beta integral formula (\ref{Beta_integral}), we can obtain the identity
\begin{equation}
\int_{0}^{x} (x-h)^{\alpha - 1} h^{\beta-1} dh = x^{\alpha + \beta - 1}\frac{\Gamma(\alpha)\Gamma(\beta)}{\Gamma(\alpha+\beta)}
\label{Beta_integral_convolution}
\end{equation}
which holds for any real number $x > 0$ whenever $\alpha,\beta >0$. Under our assumption (\ref{cond_gamma23_k_prime}),
we deduce that
$$ B(x) = \sum_{n \geq 0} \frac{\nu_{2}^{n}}{n!}
\frac{ \Gamma(\gamma_{2}+1)\Gamma(\gamma_{3} + \frac{1}{k'} + n\frac{k_{1}}{k'} + 1)}{
\Gamma( \gamma_{2} + \gamma_{3} + \frac{1}{k'} + n\frac{k_1}{k'} + 2) }
x^{\gamma_{2}+\gamma_{3} + \frac{1}{k'} + n\frac{k_1}{k'} + 1} $$
for all $x > 0$. Bearing in mind that
\begin{equation}
\Gamma(x)/\Gamma(x+a) \sim 1/x^{a} \label{quotient_Gamma}
\end{equation}
as $x$ tends to $+\infty$, for any real number $a>0$ (see for instance \cite{ba2}, Appendix B3), we deduce a constant
$K_{1.1}>0$ (depending on $\gamma_{2},\gamma_{3},k',k_{1}$) with
$$ B(x) \leq K_{1.1}x^{\gamma_{2} + \gamma_{3} + \frac{1}{k'} + 1}
\sum_{n \geq 0} \frac{1}{(n+1)^{\gamma_{2}+1}n!}(\nu_{2} x^{\frac{k_1}{k'}} )^{n} $$
for all $x>0$. Again, by (\ref{quotient_Gamma}), we check that
$$ \frac{1}{(n+1)^{\gamma_{2}+1}} \sim \frac{\Gamma(n+1)}{\Gamma(n + \gamma_{2} + 2)} $$
as $n$ tends to $+\infty$. Consequently, we get a constant $K_{1.2}>0$ (depending on $\gamma_{2}$) such that
$$ B(x) \leq K_{1.1}K_{1.2} x^{\gamma_{2}+\gamma_{3} + \frac{1}{k'} + 1}
\sum_{n \geq 0} \frac{1}{\Gamma(n + \gamma_{2} + 2)} (\nu_{2}x^{\frac{k_1}{k'}})^{n} $$
for all $x>0$. On the other hand, we remind the asymptotic property of the Wiman function $E_{\alpha,\beta}(z) =
\sum_{n \geq 0} z^{n}/\Gamma(\beta + \alpha n)$, for any $\alpha \in (0,2)$, $\beta>0$ (see \cite{erd}, expansion (22) p.210) which gives
rise to a constant
$C_{\alpha,\beta}>0$ (depending on $\alpha,\beta$) with
\begin{equation}
E_{\alpha,\beta}(z) \leq C_{\alpha,\beta} z^{\frac{1 - \beta}{\alpha}} \exp( z^{1/\alpha} ) \label{bds_E_alpha_beta} 
\end{equation}
for all $z \geq 1$. We deduce the existence of a constant
$K_{1.3}>0$ (depending on $\nu_{2},\gamma_{2}$) such that
\begin{equation}
B(x) \leq K_{1.1}K_{1.2}K_{1.3}x^{\gamma_{2}+\gamma_{3} + \frac{1}{k'} + 1}
x^{-\frac{k_1}{k'}(\gamma_{2}+1)} \exp( \nu_{2} x^{\frac{k_1}{k'}} ) \label{bds_Bx}
\end{equation}
for all $x >r/2$. Subsequently, a constant $K_{1}>0$ (depending on $\nu_{2},\gamma_{2},\gamma_{3},k',k_{1}$) can be chosen with
\begin{multline}
|\mathcal{C}_{k',\gamma_{1},\gamma_{2},\gamma_{3}}(f)(u,m)| \leq ||f(u,m)||_{(\nu_{2},\beta,\mu,k_{1})} K_{1}
\frac{ |u|^{k'(\gamma_{2}+\gamma_{3} + \frac{1}{k'} + 2)} |u|^{-k_{1}(\gamma_{2}+1)} }{(1 + |u|)^{\gamma_{1}}}
\exp( \nu_{2} |u|^{k_1} )\\
\times (1 + |m|)^{-\mu} e^{-\beta |m|} \label{conv_C_k_gamma_bds_on_Ud}
\end{multline}
for all $u \in U_{d}$, $|u| > r/2$, all $m \in \mathbb{R}$. In accordance with the last item of the assumption
(\ref{cond_gamma23_k_prime}), we get a constant $B_{1}$ (depending on $r,\gamma_{1},\gamma_{2},\gamma_{3},k_{1},k'$) with
\begin{equation}
 \sup_{u \in U_{d}, |u|>r/2} \frac{ |u|^{k'(\gamma_{2}+\gamma_{3}+2) - k_{1}(\gamma_{2}+1)} }{(1 + |u|)^{\gamma_1}}
 \leq B_{1} \label{quotient_u_gamma_bds_on_Ud}
\end{equation}
In the final step, we collect the three previous bounds (\ref{conv_C_k_gamma_bds_near_origin}),
(\ref{conv_C_k_gamma_bds_on_Ud}) and (\ref{quotient_u_gamma_bds_on_Ud}) from which we figure out that the map
$(u,m) \mapsto \mathcal{C}_{k',\gamma_{1},\gamma_{2},\gamma_{3}}(f)(u,m)$ belongs to $F_{(\nu_{2},\beta,\mu,k_{1})}^{d}$ with the
anticipated bounds (\ref{continuity_conv_C_kprime_gamma}).
\end{proof}

\begin{prop} Let $Q(X),R(X) \in \mathbb{C}[X]$ be polynomials such that
\begin{equation}
\mathrm{deg}(R) \geq \mathrm{deg}(Q) \ \ , \ \ R(im) \neq 0 \ \ , \ \ \mu > \mathrm{deg}(Q)+1 \label{cond_R_Q_mu}
\end{equation}
Let $(u,m,m_{1}) \mapsto b(u,m,m_{1})$ be a continuous function on $(U_{d} \cup D(0,r)) \times \mathbb{R} \times \mathbb{R}$,
holomorphic w.r.t $u$ on $U_{d} \cup D(0,r)$ fulfilling the bounds
\begin{equation}
\sup_{\stackrel{u \in U_{d} \cup D(0,r)}{m,m_{1} \in \mathbb{R}}} |b(u,m,m_{1})| \leq C_{b} \label{bds_bumm1}
\end{equation}
for some constant $C_{b}>0$. Then, there exists a constant $C_{2}>0$ (depending on $Q$,$R$ and $\mu$) such that
\begin{multline}
||\frac{1}{R(im)} \int_{-\infty}^{+\infty} f(m-m_{1}) Q(im_{1})b(u,m,m_{1})g(u,m_{1}) dm_{1} ||_{(\nu_{2},\beta,\mu,k_{1})}
\\
\leq C_{2}C_{b} ||f(m)||_{(\beta,\mu)} ||g(u,m)||_{(\nu_{2},\beta,\mu,k_{1})}
\end{multline}
whenever $f$ belongs to $E_{(\beta,\mu)}$ and $g$ belongs to $F_{(\nu_{2},\beta,\mu,k_{1})}^{d}$.
\end{prop}
\begin{proof} The proof is closely related to the one of Proposition 3 of \cite{lama2}. Again, we give a thorough explanation of the result.
We take $f$ inside $E_{(\beta,\mu)}$ and select $g$ belonging to $F^{d}_{(\nu_{2},\beta,\mu,k_{1})}$. We first recast the norm of the
convolution operator as follows
\begin{multline}
N_{2} = ||\frac{1}{R(im)} \int_{-\infty}^{+\infty} f(m-m_{1}) Q(im_{1})b(u,m,m_{1})g(u,m_{1}) dm_{1} ||_{(\nu_{2},\beta,\mu,k_{1})}\\
= \sup_{u \in U_{d} \cup D(0,r), m \in \mathbb{R}} (1 + |m|)^{\mu} \frac{1}{|u|} e^{\beta |m|} \exp(-\nu_{2} |u|^{k_1})\\
\times |\frac{1}{R(im)} \int_{-\infty}^{+\infty} \{ (1 + |m-m_{1}|)^{\mu} \exp(\beta|m-m_{1}|) f(m-m_{1}) \} b(u,m,m_{1}) \\
\times
\{ (1 + |m_{1}|)^{\mu} e^{\beta |m_1|} \frac{1}{|u|} \exp( - \nu_{2} |u|^{k_1} ) g(u,m_{1}) \} \mathcal{A}(u,m,m_{1}) dm_{1}| \label{N2}
\end{multline}
where
$$ \mathcal{A}(u,m,m_{1}) = \frac{Q(im_{1}) \exp(-\beta |m_{1}|) \exp(-\beta |m-m_{1}|) }{(1 + |m_{1}|)^{\mu} ( 1 + |m-m_{1}|)^{\mu}}
|u| \exp( \nu_{2} |u|^{k_1}) $$
By construction of the polynomials $Q$ and $R$, one can sort two constants $\mathfrak{Q},\mathfrak{R}>0$ with
\begin{equation}
|Q(im_{1})| \leq \mathfrak{Q}(1 + |m_{1}|)^{\mathrm{deg}(Q)} \ \ , \ \ |R(im)| \geq \mathfrak{R}(1 + |m|)^{\mathrm{deg}(R)} \label{bds_Q_R}
\end{equation}
for all $m,m_{1} \in \mathbb{R}$. As a consequence of (\ref{N2}), (\ref{bds_Q_R}) and (\ref{bds_bumm1}) with the help of the triangular inequality
$|m| \leq |m_{1}| + |m-m_{1}|$, we are led to the
bounds
$$ N_{2} \leq C_{2}C_{b} ||f(m)||_{(\beta,\mu)} ||g(u,m)||_{(\nu_{2},\beta,\mu,k_{1})} $$
where
$$ C_{2} = \frac{\mathfrak{Q}}{\mathfrak{R}} \sup_{m \in \mathbb{R}} (1 + |m|)^{\mu - \mathrm{deg}(R)}
\int_{-\infty}^{+\infty} \frac{1}{(1 + |m-m_{1}|)^{\mu} (1 + |m_{1}|)^{\mu - \mathrm{deg}(Q)} } dm_{1} $$
is a finite constant under the first and last restriction of (\ref{cond_R_Q_mu}) according to the estimates of
Lemma 2.2 from \cite{cota2} or Lemma 4 of \cite{ma2}.
\end{proof}

In the next step, we discuss further analytic assumptions on the leading polynomials $Q(X)$ and $R_{D}(X)$ in order to be able to
transform our problem (\ref{main_conv_w}) into a fixed point equation as described afterwards, see (\ref{fixed_pt_Hepsilon}).

We follow a similar roadmap as in our previous study \cite{lama1}. Namely, we take for granted that one can find a bounded sectorial annulus
$$ S_{Q,R_{D}} = \{ z \in \mathbb{C}^{\ast} / r_{Q,R_{D},1} < |z| < r_{Q,R_{D},2}, |\mathrm{arg}(z) - d_{Q,R_{D}}| \leq \eta_{Q,R_{D}} \} $$
with direction $d_{Q,R_{D}} \in \mathbb{R}$, aperture $\eta_{Q,R_{D}}>0$ for some given inner and outer radius
$0< r_{Q,R_{D},1} < r_{Q,R_{D},2}$ with the inclusion
\begin{equation}
\{ \frac{Q(im)}{R_{D}(im)} / m \in \mathbb{R} \} \subset S_{Q,R_{D}} \label{quotient_Q_RD_in_S}
\end{equation}
In the sequel, we need to factorize explicitely the polynomial
\begin{equation}
P_{m}(u) = Q(im) - R_{D}(im)k^{\delta_{D}} (k')^{m_D} u^{k \delta_{D}} \label{defin_Pm}
\end{equation}
as follows
\begin{equation}
P_{m}(u) = -R_{D}(im) k^{\delta_D} (k')^{m_D} \Pi_{l=0}^{k\delta_{D}-1} (u - q_{l}(m)) \label{factor_P_m}
\end{equation}
where the roots $q_{l}(m)$ are given by
$$ q_{l}(m) = (\frac{|Q(im)|}{|R_{D}(im)k^{\delta_D}(k')^{m_D}})^{\frac{1}{k\delta_{D}}}
\exp\left( \sqrt{-1}( \mathrm{arg}( \frac{Q(im)}{R_{D}(im)k^{\delta_{D}}(k')^{m_D}} ) \frac{1}{k\delta_{D}} +
\frac{2\pi l}{k \delta_{D}} ) \right) $$
for all $0 \leq l \leq \delta_{D}k - 1$, for all $m \in \mathbb{R}$.

We select an unbounded sector $U_{d}$ centered at 0, a small disc $D(0,r)$ and we assign the sector $S_{Q,R_{D}}$ in a way that the next
two conditions hold:\\
1) A constant $M_{1}>0$ can be found such that
\begin{equation}
|u - q_{l}(m)| \geq M_{1}(1 + |u|) \label{dist_Ud_Dr_qlm_M1} 
\end{equation}
for all $0 \leq l \leq k\delta_{D} - 1$, all $m \in \mathbb{R}$, whenever $u \in U_{d} \cup D(0,r)$.\\
2) There exists a constant $M_{2}>0$ with
\begin{equation}
|u - q_{l_0}(m)| \geq M_{2} |q_{l_0}(m)| \label{dist_Ud_Dr_qlm_M2}
\end{equation}
for some $0 \leq l_{0} \leq \delta_{D}k-1$, all $m \in \mathbb{R}$, all $u \in U_{d} \cup D(0,r)$.

In order to examine the first point 1), we observe that under the hypothesis (\ref{quotient_Q_RD_in_S}), the roots $q_{l}(m)$ are bounded
from below and satisfy $|q_{l}(m)| \geq 2r$ for all $m \in \mathbb{R}$, all $0 \leq l \leq \delta_{D}k-1$ for a suitable choice of the radii
$r_{Q,R_{D},1},r >0$. Besides, for all $m \in \mathbb{R}$, all $0 \leq l \leq \delta_{D}k-1$, these roots remain inside an union
$\mathcal{U}$ of unbounded sectors centered at 0 that do not cover a full neighborhood of 0 in $\mathbb{C}^{\ast}$ whenever the aperture
$\eta_{Q,R_{D}}>0$ is taken small enough. Therefore, we may choose a sector $U_{d}$ such that
$$ U_{d} \cap U = \emptyset $$
It has the property that for all $0 \leq l \leq \delta_{D}k-1$, the quotients $q_{l}(m)/u$ lay outside some small disc centered at 1 in
$\mathbb{C}$ for all $u \in U_{d}$, all $m \in \mathbb{R}$. As a consequence, (\ref{dist_Ud_Dr_qlm_M1}) follows.

With the sector $U_{d}$ and disc $D(0,r)$ chosen as above, the second point 2) then proceeds from the fact that for any fixed
$0 \leq l_{0} \leq \delta_{D}k-1$, the quotient $u/q_{l_{0}}(m)$ stays apart a small disc centered at 1 in $\mathbb{C}$ for all
$u \in U_{d} \cup D(0,r)$, all $m \in \mathbb{R}$.

The factorization (\ref{factor_P_m}) along with the lower bounds (\ref{dist_Ud_Dr_qlm_M1}) and (\ref{dist_Ud_Dr_qlm_M2}) provided above, permits
us to find lower bounds for $P_{m}(u)$, namely a constant $C_{P}>0$ (independent of $r_{Q,R_{D},1}$ and $r_{Q,R_{D},2}$) with
\begin{multline}
|P_{m}(u)| \geq M_{1}^{\delta_{D}k-1} M_{2} |R_{D}(im)| k^{\delta_{D}} (k')^{m_D}
( \frac{|Q(im)|}{|R_{D}(im)|k^{\delta_D}(k')^{m_D}} )^{\frac{1}{k\delta_D}} (1 + |u|)^{k\delta_{D} - 1} \\
\geq C_{P}(r_{Q,R_{D},1})^{\frac{1}{k \delta_D}}|R_{D}(im)|(1 + |u|)^{k\delta_{D}-1} \label{low_bds_Pmu}
\end{multline}
for all $u \in U_{d} \cup D(0,r)$, all $m \in \mathbb{R}$.

For later requirement, we already display the next upper bounds. There exists
a constant $C_{P,R_{D}}>0$ (depending on $k,k',\delta_{D},m_{D},Q,R_{D},r_{Q,R_{D},2}$) such that
\begin{equation}
\left| \frac{P_{m_1}(u)R_{D}(im)}{P_{m}(u)R_{D}(im_{1})} \right| \leq C_{P,R_{D}} \label{bds_quotient_Pmu_RD}
\end{equation}
for all $u \in U_{d} \cup D(0,r)$, all $m,m_{1} \in \mathbb{R}$. Indeed, owing to the assumption (\ref{quotient_Q_RD_in_S})
along with (\ref{dist_Ud_Dr_qlm_M1}), the factorization
(\ref{factor_P_m}) yields the lower bounds
$$ |P_{m}(u)| \geq |R_{D}(im)|k^{\delta_D}(k')^{m_D}M_{1}^{k\delta_{D}}(1 + |u|)^{k\delta_{D}} $$
for all $u \in U_{d} \cup D(0,r)$, all $m \in \mathbb{R}$. On the other hand, having a glance again at (\ref{quotient_Q_RD_in_S}),
the triangular inequality allows us to write
$$ \left| \frac{P_{m_1}(u)}{R_{D}(im_{1})} \right| \leq \left| \frac{Q(im_{1})}{R_{D}(im_{1})} \right| +
k^{\delta_D}(k')^{m_D}|u|^{k\delta_{D}} \leq r_{Q,R_{D},2} + k^{\delta_D}(k')^{m_D}|u|^{k\delta_{D}} $$
for all $u \in \mathbb{C}$, all $m_{1} \in \mathbb{R}$.
Therefore, we deduce that
$$ 
\left| \frac{P_{m_1}(u)R_{D}(im)}{P_{m}(u)R_{D}(im_{1})} \right| \leq
\sup_{x \geq 0} \frac{r_{Q,R_{D},2} + k^{\delta_{D}} (k')^{m_D} x^{k\delta_{D}}}{k^{\delta_{D}}(k')^{m_D} M_{1}^{k\delta_{D}}
(1 + x)^{k\delta_{D}} }=C_{P,R_{D}}
$$
whenever $u \in U_{d} \cup D(0,r)$, $m,m_{1} \in \mathbb{R}$.

In the next proposition, we provide sufficient conditions in order to ensure the existence and uniqueness of a solution
$w^{d}(u,m,\epsilon)$ of the main integral equation (\ref{main_conv_w}) that belongs to the Banach space $F_{(\nu_{2},\beta,\mu,k_{1})}^{d}$.
\begin{prop} We take for granted that the next additional requirement
\begin{equation}
k_{1} \geq 1 \ \ , \ \ \mu > \mathrm{deg}(R_{\bf l}) + 1 \ \ , \ \ k\delta_{D} \geq kl_{1}(1 - \frac{k_{1}}{k'}) + k_{1}l_{2} + 1  \label{cond_exist_uniq_wd}
\end{equation}
holds for all ${\bf l}=(l_{1},l_{2}) \in I$. Then, for a proper choice of the radius $r_{Q,R_{D},1}>0$ (see \ref{quotient_Q_RD_in_S})
taken large enough and constants $C_{\bf l}>0$ (see \ref{defin_Cl}) sufficiently small for ${\bf l} \in I$, one can find a
constant $\varpi$ such that the equation
(\ref{main_conv_w}) possesses a unique solution $(u,m) \mapsto w^{d}(u,m,\epsilon)$ in the space $F_{(\nu_{2},\beta,\mu,k_{1})}^{d}$ with
the feature that
\begin{equation}
||w^{d}(u,m,\epsilon)||_{(\nu_{2},\beta,\mu,k_{1})} \leq \varpi \label{norm_wd_varpi} 
\end{equation}
for all $\epsilon \in D(0,\epsilon_{0})$, where
$\nu_{2}=(1/T_{0})^{k_1}$ and $T_{0},\beta,\mu,k_{1}$ are introduced in Section 3 within the construction of the map $\psi(u,m,\epsilon)$, see (\ref{psi_exp_bds}).
\end{prop}
\begin{proof}
We initiate the proof with a lemma which studies a shrinking map that allows us to apply a classical fixed point theorem.
\begin{lemma} Under the constraints (\ref{cond_exist_uniq_wd}), one can select constants $r_{Q,R_{D},1}>0$, $C_{\bf l}$ for ${\bf l} \in I$
and $\varpi>0$ in a way that
for all $\epsilon \in D(0,\epsilon_{0})$, the map $\mathcal{H}_{\epsilon}$ defined as
\begin{multline}
\mathcal{H}_{\epsilon}(w(u,m)) :=  \sum_{1 \leq p \leq m_{D}-1} R_{D}(im) k^{\delta_D} A_{m_{D},p}
\frac{u^{k'}}{\Gamma(m_{D}-p)P_{m}(u)}\\
\times \int_{0}^{u^{k'}} (u^{k'}-s)^{m_{D}-p - 1}
(k')^{p} s^{p} w(s^{1/k'},m) \frac{ds}{s}\\
+ R_{D}(im) k^{\delta_D} \sum_{q=1}^{m_{D}-1} a_{q,m_{D}} \left( \frac{u^{k'}}{\Gamma(\frac{d_{q,k,D}}{k'})P_{m}(u)}
\int_{0}^{u^{k'}} (u^{k'} - s)^{\frac{d_{q,k,D}}{k'} - 1} (k')^{q} s^{q} w(s^{1/k'},m) \frac{ds}{s} \right.
\\
\left. + \sum_{1 \leq p \leq q-1} A_{q,p} \frac{u^{k'}}{\Gamma( \frac{d_{q,k,D} + k'(q-p)}{k'} )P_{m}(u)}
\int_{0}^{u^{k'}} (u^{k'} - s)^{\frac{d_{q,k,D} + k'(q-p)}{k'} - 1} (k')^{p} s^{p} w(s^{1/k'},m) \frac{ds}{s} \right)\\
+ \sum_{{\bf l} = (l_{1},l_{2}) \in I} \epsilon^{\Delta_{{\bf l}} - kl_{1}} \frac{1}{(2\pi)^{1/2}}
\int_{-\infty}^{+\infty} C_{{\bf l}}(m-m_{1},\epsilon) k^{l_1} R_{{\bf l}}(im_{1}) \\
\times \sum_{q=1}^{l_2} a_{q,l_{2}}
\left( \frac{u^{k'}}{\Gamma(\frac{e_{q,k,l_{1}}}{k'})P_{m}(u)} \int_{0}^{u^{k'}} (u^{k'} - s)^{\frac{e_{q,k,l_{1}}}{k'}-1}
(k')^{q} s^{q} w(s^{1/k'},m_{1}) \frac{ds}{s} \right.\\
\left. + \sum_{1 \leq p \leq q-1} A_{q,p} \frac{u^{k'}}{\Gamma( \frac{e_{q,k,l_{1}} + k'(q-p)}{k'} )P_{m}(u)}
\int_{0}^{u^{k'}} (u^{k'} - s)^{\frac{e_{q,k,l_{1}} + k'(q-p)}{k'} - 1} (k')^{p} s^{p} w(s^{1/k'},m_{1}) \frac{ds}{s} \right) dm_{1}\\
+ \frac{\psi(u,m,\epsilon)}{P_{m}(u)}
\end{multline}
verifies the next properties.\\
{\bf i)} The following inclusion
\begin{equation}
\mathcal{H}_{\epsilon}(\bar{B}(0,\varpi)) \subset \bar{B}(0,\varpi) \label{H_epsilon_B_into_B} 
\end{equation}
where $\bar{B}(0,\varpi)$ is the closed ball of radius $\varpi$ centered at 0 in $F_{(\nu_{2},\beta,\mu,k_{1})}^{d}$, for all
$\epsilon \in D(0,\epsilon_{0})$.\\
{\bf ii)} The map $\mathcal{H}_{\epsilon}$ is shrinking, namely
\begin{equation}
|| \mathcal{H}_{\epsilon}(w_{2}(u,m)) - \mathcal{H}_{\epsilon}(w_{1}(u,m)) ||_{(\nu_{2},\beta,\mu,k_{1})} \leq \frac{1}{2}
||w_{2}(u,m) - w_{1}(u,m)||_{(\nu_{2},\beta,\mu,k_{1})} \label{H_epsilon_shrink}
\end{equation}
whenever $w_{1},w_{2} \in \bar{B}(0,\varpi)$, for all $\epsilon \in D(0,\epsilon_{0})$.
\end{lemma}
\begin{proof} According to the first and second bounds in (\ref{psi_exp_bds}) together with (\ref{low_bds_Pmu}), we can find a constant $\zeta_{\psi}>0$ (depending on
$\nu_{2}$ and $k_{1}$) with
\begin{equation}
|| \frac{\psi(u,m,\epsilon)}{P_{m}(u)} ||_{(\nu_{2},\beta,\mu,k_{1})} \leq \frac{K_{0}\zeta_{\psi}}{C_{P}(r_{Q,R_{D},1})^{\frac{1}{k\delta_{D}}}
\min_{m \in \mathbb{R}} |R_{D}(im)| } \label{norm_bds_ball_1}
\end{equation}
where $K_{0}>0$ is a constant introduced in the condition (\ref{norm_bds_psi_n}), for all $\epsilon \in D(0,\epsilon_{0})$.

We focus on the first feature (\ref{H_epsilon_B_into_B}). Let us take $w(u,m)$ in $F_{(\nu_{2},\beta,\mu,k_{1})}^{d}$ and assume that
$||w(u,m)||_{(\nu_{2},\beta,\mu,k_{1})} \leq \varpi$. In accordance with the first condition of (\ref{cond_exist_uniq_wd}) and
(\ref{rel_k_deltaD_mD_k_prime}) together with
the lower bounds (\ref{low_bds_Pmu}), Proposition 1 gives rise to a constant $C_{1.1}>0$ (depending on $\nu_{2},k,\delta_{D},m_{D},k',k_{1},r$) such that
\begin{multline}
||R_{D}(im) \frac{u^{k'}}{P_{m}(u)} \int_{0}^{u^{k'}} (u^{k'} -s)^{m_{D}-p-1}s^{p-1}w(s^{1/k'},m) ds ||_{(\nu_{2},\beta,\mu,k_{1})}\\
\leq \frac{C_{1.1}}{C_{P}(r_{Q,R_{D},1})^{\frac{1}{k\delta_{D}}}} ||w(u,m)||_{(\nu_{2},\beta,\mu,k_{1})}
\leq \frac{C_{1.1}}{C_{P}(r_{Q,R_{D},1})^{\frac{1}{k\delta_{D}}}} \varpi \label{norm_bds_ball_2}
\end{multline}
for all $1 \leq p \leq m_{D}-1$. Again the first constraint of (\ref{cond_exist_uniq_wd}) and
(\ref{rel_k_deltaD_mD_k_prime}) along with (\ref{low_bds_Pmu}), allows us to
apply Proposition 1 in order to get a constant $C_{1.2}>0$ (depending on $\nu_{2},k,\delta_{D},m_{D},k',k_{1},r$) for which
\begin{multline}
||R_{D}(im)\frac{u^{k'}}{P_{m}(u)} \int_{0}^{u^{k'}} (u^{k'}-s)^{\frac{d_{q,k,D}}{k'} - 1}
s^{q-1}w(s^{1/k'},m) ds ||_{(\nu_{2},\beta,\mu,k_{1})} \\
\leq \frac{C_{1.2}}{C_{P}(r_{Q,R_{D},1})^{\frac{1}{k\delta_{D}}}}
||w(u,m)||_{(\nu_{2},\beta,\mu,k_{1})} \leq \frac{C_{1.2}}{C_{P}(r_{Q,R_{D},1})^{\frac{1}{k\delta_{D}}}} \varpi \label{norm_bds_ball_3}
 \end{multline}
for all $1 \leq q \leq m_{D}-1$ together with
\begin{multline}
||R_{D}(im)\frac{u^{k'}}{P_{m}(u)} \int_{0}^{u^{k'}} (u^{k'} - s)^{\frac{d_{q,k,D} + k'(q-p)}{k'} - 1} s^{p-1} w(s^{1/k'},m) ds
||_{(\nu_{2},\beta,\mu,k_{1})} \leq \\
\frac{C_{1.2}}{C_{P}(r_{Q,R_{D},1})^{\frac{1}{k\delta_{D}}}} ||w(u,m)||_{(\nu_{2},\beta,\mu,k_{1})}
\leq \frac{C_{1.2}}{C_{P}(r_{Q,R_{D},1})^{\frac{1}{k\delta_{D}}}} \varpi \label{norm_bds_ball_4}
\end{multline}
whenever $1 \leq q \leq m_{D}-1$, $1 \leq p \leq q-1$.

We handle now the terms arising in the sum over the set $I$. Owing to the bounds (\ref{bds_quotient_Pmu_RD}) together with the second
condition of (\ref{cond_exist_uniq_wd}), Proposition 2 grants the existence of a constant $C_{2}>0$
(depending on $R_{D}$, $R_{\bf l}$ and $\mu$) such that
\begin{multline}
|| \int_{-\infty}^{+\infty} C_{\bf l}(m-m_{1},\epsilon) R_{\bf l}(im_{1}) \frac{u^{k'}}{P_{m}(u)}
\int_{0}^{u^{k'}} (u^{k'} - s)^{\frac{e_{q,k,l_{1}}}{k'}-1} s^{q-1} \\
\times w(s^{1/k'},m_{1}) ds dm_{1} ||_{(\nu_{2},\beta,\mu,k_{1})}
\\= ||\frac{1}{R_{D}(im)} \int_{-\infty}^{+\infty} C_{\bf l}(m-m_{1},\epsilon)R_{\bf l}(im_{1})
\left( \frac{P_{m_1}(u)R_{D}(im)}{P_{m}(u)R_{D}(im_{1})} \right)
\frac{u^{k'}R_{D}(im_{1})}{P_{m_1}(u)} \\
\times \int_{0}^{u^{k'}} (u^{k'}-s)^{\frac{e_{q,k,l_{1}}}{k'}-1} s^{q-1} w(s^{1/k'},m_{1}) ds dm_{1}
||_{(\nu_{2},\beta,\mu,k_{1})} \leq C_{2}C_{P,R_{D}} ||C_{\bf l}(m,\epsilon)||_{(\beta,\mu)}\\
\times ||\frac{u^{k'}R_{D}(im)}{P_{m}(u)} \int_{0}^{u^{k'}} (u^{k'}-s)^{\frac{e_{q,k,l_{1}}}{k'}-1} s^{q-1}w(s^{1/k'},m) ds
||_{(\nu_{2},\beta,\mu,k_{1})}  \label{norm_bds_ball_5}
\end{multline}
for all ${\bf l} = (l_{1},l_{2}) \in I$, all $1 \leq q \leq l_{2}$
together with
\begin{multline}
|| \int_{-\infty}^{+\infty} C_{\bf l}(m-m_{1},\epsilon) R_{\bf l}(im_{1}) \frac{u^{k'}}{P_{m}(u)}
\int_{0}^{u^{k'}} (u^{k'} - s)^{\frac{e_{q,k,l_{1}} + k'(q-p)}{k'}-1} s^{p-1} \\
\times w(s^{1/k'},m_{1}) ds dm_{1} ||_{(\nu_{2},\beta,\mu,k_{1})}
\\= ||\frac{1}{R_{D}(im)} \int_{-\infty}^{+\infty} C_{\bf l}(m-m_{1},\epsilon)R_{\bf l}(im_{1})
\left( \frac{P_{m_1}(u)R_{D}(im)}{P_{m}(u)R_{D}(im_{1})} \right)
\frac{u^{k'}R_{D}(im_{1})}{P_{m_1}(u)} \\
\times \int_{0}^{u^{k'}} (u^{k'}-s)^{\frac{e_{q,k,l_{1}}+k'(q-p)}{k'}-1} s^{p-1} w(s^{1/k'},m_{1}) ds dm_{1}
||_{(\nu_{2},\beta,\mu,k_{1})} \leq C_{2}C_{P,R_{D}} ||C_{\bf l}(m,\epsilon)||_{(\beta,\mu)}\\
\times ||\frac{u^{k'}R_{D}(im)}{P_{m}(u)} \int_{0}^{u^{k'}} (u^{k'}-s)^{\frac{e_{q,k,l_{1}}+k'(q-p)}{k'}-1} s^{p-1}w(s^{1/k'},m) ds
||_{(\nu_{2},\beta,\mu,k_{1})} \label{norm_bds_ball_6}
\end{multline}
whenever ${\bf l}=(l_{1},l_{2}) \in I$, $1 \leq q \leq l_{2}$, $1 \leq p \leq q-1$.

Furthermore, under the third requirement of (\ref{cond_exist_uniq_wd}) and keeping in mind the lower bounds
(\ref{low_bds_Pmu}), we obtain a constant $C_{1.3}>0$ (depending on $\nu_{2},{\bf l}$,$k$,$k'$,$k_{1}$,$\delta_{D}$,$r$) such that
\begin{multline}
||\frac{u^{k'}R_{D}(im)}{P_{m}(u)} \int_{0}^{u^{k'}} (u^{k'}-s)^{\frac{e_{q,k,l_{1}}}{k'}-1} s^{q-1}w(s^{1/k'},m) ds
||_{(\nu_{2},\beta,\mu,k_{1})} \\
\leq \frac{C_{1.3}}{C_{P}(r_{Q,R_{D},1})^{\frac{1}{k\delta_{D}}}} ||w(u,m)||_{(\nu_{2},\beta,\mu,k_{1})}
\leq \frac{C_{1.3}}{C_{P}(r_{Q,R_{D},1})^{\frac{1}{k\delta_{D}}}} \varpi \label{norm_bds_ball_7}
\end{multline}
along with
\begin{multline}
||\frac{u^{k'}R_{D}(im)}{P_{m}(u)} \int_{0}^{u^{k'}} (u^{k'}-s)^{\frac{e_{q,k,l_{1}}+k'(q-p)}{k'}-1} s^{p-1}w(s^{1/k'},m) ds
||_{(\nu_{2},\beta,\mu,k_{1})} \\
\leq \frac{C_{1.3}}{C_{P}(r_{Q,R_{D},1})^{\frac{1}{k\delta_{D}}}} ||w(u,m)||_{(\nu_{2},\beta,\mu,k_{1})}
\leq \frac{C_{1.3}}{C_{P}(r_{Q,R_{D},1})^{\frac{1}{k\delta_{D}}}} \varpi \label{norm_bds_ball_8}
\end{multline}
provided that ${\bf l}=(l_{1},l_{2}) \in I$, $1 \leq q \leq l_{2}$, $1 \leq p \leq q-1$.

Now, we assign the radius $r_{Q,R_{D},1}>0$ to be large enough and the constants $C_{\bf l}>0$, for
${\bf l} \in I$, to be sufficiently tiny in order to find
a constant $\varpi>0$ with
\begin{multline}
\sum_{1 \leq p \leq m_{D}-1} \frac{k^{\delta_D}|A_{m_{D},p}|(k')^{p}}{\Gamma(m_{D}-p)}
\frac{C_{1.1}}{C_{P}(r_{Q,R_{D},1})^{\frac{1}{k\delta_{D}}}} \varpi \\
+ \sum_{q=1}^{m_{D}-1} k^{\delta_D} |a_{q,m_{D}}| \left(
\frac{C_{1.2}(k')^{q}}{\Gamma(\frac{d_{q,k,D}}{k'})C_{P}(r_{Q,R_{D},1})^{\frac{1}{k\delta_{D}}}} \varpi \right. \\
+
\left. \sum_{1 \leq p \leq q-1} |A_{q,p}| \frac{C_{1.2}(k')^{p}}{\Gamma(\frac{d_{q,k,D} + k'(q-p)}{k'})C_{P}(r_{Q,R_{D},1})^{\frac{1}{k\delta_{D}}}} \varpi \right) + \sum_{{\bf l} = (l_{1},l_{2}) \in I} |\epsilon|^{\Delta_{\bf l} - kl_{1}}
\frac{1}{(2\pi)^{1/2}} k^{l_1}\\
\times \sum_{q=1}^{l_2} |a_{q,l_{2}}| \left( \frac{(k')^{q}}{\Gamma(\frac{e_{q,k,l_{1}}}{k'})}C_{2}C_{P,R_{D}}C_{{\bf l}}
\frac{C_{1.3}}{C_{P}(r_{Q,R_{D},1})^{\frac{1}{k\delta_{D}}}} \varpi \right.\\
\left. + \sum_{1 \leq p \leq q-1} |A_{q,p}| \frac{(k')^{p}}{\Gamma(\frac{e_{q,k,l_{1}}+k'(q-p)}{k'})}C_{2}C_{P,R_{D}}C_{{\bf l}}
\frac{C_{1.3}}{C_{P}(r_{Q,R_{D},1})^{\frac{1}{k\delta_{D}}}} \varpi \right) \\
+ 
\frac{K_{0}\zeta_{\psi}}{C_{P}(r_{Q,R_{D},1})^{\frac{1}{k\delta_{D}}} \min_{m \in \mathbb{R}} |R_{D}(im)|} \leq \varpi \label{cond_Hepsilon_inclusion}
\end{multline}
At last, if one gathers the above norms bounds (\ref{norm_bds_ball_1}), (\ref{norm_bds_ball_2}), (\ref{norm_bds_ball_3}),
(\ref{norm_bds_ball_4}) in a row with
(\ref{norm_bds_ball_5}), (\ref{norm_bds_ball_6}), (\ref{norm_bds_ball_7}), (\ref{norm_bds_ball_8}) under the restriction
(\ref{cond_Hepsilon_inclusion}), the inclusion (\ref{H_epsilon_B_into_B}) follows.\medskip

In the next part of the proof, we turn to the explanation of the second property (\ref{H_epsilon_shrink}). Indeed, take
$w_{1}(\tau,m)$ and $w_{2}(\tau,m)$ inside the ball $\bar{B}(0,\varpi)$ from $F_{(\nu_{2},\beta,\mu,k_{1})}^{d}$. Returning back to the
inequalities (\ref{norm_bds_ball_2}), (\ref{norm_bds_ball_3}), (\ref{norm_bds_ball_4}) allows us to get the next bounds
\begin{multline}
||R_{D}(im) \frac{u^{k'}}{P_{m}(u)} \int_{0}^{u^{k'}} (u^{k'} -s)^{m_{D}-p-1}s^{p-1}(w_{2}(s^{1/k'},m) - w_{1}(s^{1/k'},m)) ds
||_{(\nu_{2},\beta,\mu,k_{1})}\\
\leq \frac{C_{1.1}}{C_{P}(r_{Q,R_{D},1})^{\frac{1}{k\delta_{D}}}} ||w_{2}(u,m) - w_{1}(u,m)||_{(\nu_{2},\beta,\mu,k_{1})}
\label{norm_bds_shrink_1}
\end{multline}
for all $1 \leq p \leq m_{D}-1$ along with
\begin{multline}
||R_{D}(im)\frac{u^{k'}}{P_{m}(u)} \int_{0}^{u^{k'}} (u^{k'}-s)^{\frac{d_{q,k,D}}{k'} - 1}
s^{q-1}(w_{2}(s^{1/k'},m) - w_{1}(s^{1/k'},m)) ds ||_{(\nu_{2},\beta,\mu,k_{1})} \\
\leq \frac{C_{1.2}}{C_{P}(r_{Q,R_{D},1})^{\frac{1}{k\delta_{D}}}}
||w_{2}(u,m) - w_{1}(u,m)||_{(\nu_{2},\beta,\mu,k_{1})} \label{norm_bds_shrink_2}
 \end{multline}
for all $1 \leq q \leq m_{D}-1$ together with
\begin{multline}
||R_{D}(im)\frac{u^{k'}}{P_{m}(u)} \int_{0}^{u^{k'}} (u^{k'} - s)^{\frac{d_{q,k,D} + k'(q-p)}{k'} - 1} s^{p-1}\\
\times 
(w_{2}(s^{1/k'},m) - w_{1}(s^{1/k'},m)) ds
||_{(\nu_{2},\beta,\mu,k_{1})} \leq 
\frac{C_{1.2}}{C_{P}(r_{Q,R_{D},1})^{\frac{1}{k\delta_{D}}}} ||w_{2}(u,m) - w_{1}(u,m)||_{(\nu_{2},\beta,\mu,k_{1})}
\label{norm_bds_shrink_3}
\end{multline}
whenever $1 \leq q \leq m_{D}-1$, $1 \leq p \leq q-1$.

Furthermore, the inequalities (\ref{norm_bds_ball_5}) combined with (\ref{norm_bds_ball_7}) and
(\ref{norm_bds_ball_6}) coupled with (\ref{norm_bds_ball_8}) give rise to the next two bounds
\begin{multline}
|| \int_{-\infty}^{+\infty} C_{\bf l}(m-m_{1},\epsilon) R_{\bf l}(im_{1}) \frac{u^{k'}}{P_{m}(u)}
\int_{0}^{u^{k'}} (u^{k'} - s)^{\frac{e_{q,k,l_{1}}}{k'}-1} s^{q-1} \\
\times (w_{2}(s^{1/k'},m_{1}) - w_{1}(s^{1/k'},m_{1})) ds dm_{1} ||_{(\nu_{2},\beta,\mu,k_{1})} \leq C_{2}C_{P,R_{D}}
||C_{\bf l}(m,\epsilon)||_{(\beta,\mu)}\\
\times
\frac{C_{1.3}}{C_{P}(r_{Q,R_{D},1})^{\frac{1}{k\delta_{D}}}} ||w_{2}(u,m) - w_{1}(u,m)||_{(\nu_{2},\beta,\mu,k_{1})} \label{norm_bds_shrink_4}
\end{multline}
for all ${\bf l}=(l_{1},l_{2}) \in I$, all $1 \leq q \leq l_{2}$ in a row with
\begin{multline}
 || \int_{-\infty}^{+\infty} C_{\bf l}(m-m_{1},\epsilon) R_{\bf l}(im_{1}) \frac{u^{k'}}{P_{m}(u)}
\int_{0}^{u^{k'}} (u^{k'} - s)^{\frac{e_{q,k,l_{1}} + k'(q-p)}{k'}-1} s^{p-1} \\
\times (w_{2}(s^{1/k'},m_{1}) - w_{1}(s^{1/k'},m_{1})) ds dm_{1} ||_{(\nu_{2},\beta,\mu,k_{1})}
\leq C_{2}C_{P,R_{D}} ||C_{\bf l}(m,\epsilon)||_{(\beta,\mu)} \\
\times 
\frac{C_{1.3}}{C_{P}(r_{Q,R_{D},1})^{\frac{1}{k\delta_{D}}}} ||w_{2}(u,m) - w_{1}(u,m)||_{(\nu_{2},\beta,\mu,k_{1})} \label{norm_bds_shrink_5}
\end{multline}
for all ${\bf l}=(l_{1},l_{2}) \in I$, $1 \leq q \leq l_{2}$, $1 \leq p \leq q-1$.

Then, we choose the radius $r_{Q,R_{D},1}>0$ large enough and control the constant $C_{\bf l}>0$, for ${\bf l} \in I$, close to 0 in a
way that
\begin{multline}
 \sum_{1 \leq p \leq m_{D}-1} \frac{k^{\delta_D}|A_{m_{D},p}|(k')^{p}}{\Gamma(m_{D}-p)}
\frac{C_{1.1}}{C_{P}(r_{Q,R_{D},1})^{\frac{1}{k\delta_{D}}}} \\
+ \sum_{q=1}^{m_{D}-1} k^{\delta_D} |a_{q,m_{D}}| \left(
\frac{C_{1.2}(k')^{q}}{\Gamma(\frac{d_{q,k,D}}{k'})C_{P}(r_{Q,R_{D},1})^{\frac{1}{k\delta_{D}}}} \right. \\
+
\left. \sum_{1 \leq p \leq q-1} |A_{q,p}| \frac{C_{1.2}(k')^{p}}{\Gamma(\frac{d_{q,k,D} + k'(q-p)}{k'})C_{P}(r_{Q,R_{D},1})^{\frac{1}{k\delta_{D}}}} \right) + \sum_{{\bf l} = (l_{1},l_{2}) \in I} |\epsilon|^{\Delta_{\bf l} - kl_{1}}
\frac{1}{(2\pi)^{1/2}} k^{l_1}\\
\times \sum_{q=1}^{l_2} |a_{q,l_{2}}| \left( \frac{(k')^{q}}{\Gamma(\frac{e_{q,k,l_{1}}}{k'})}C_{2}C_{P,R_{D}}C_{{\bf l}}
\frac{C_{1.3}}{C_{P}(r_{Q,R_{D},1})^{\frac{1}{k\delta_{D}}}} \right.\\
\left. + \sum_{1 \leq p \leq q-1} |A_{q,p}| \frac{(k')^{p}}{\Gamma(\frac{e_{q,k,l_{1}}+k'(q-p)}{k'})}C_{2}C_{P,R_{D}}C_{{\bf l}}
\frac{C_{1.3}}{C_{P}(r_{Q,R_{D},1})^{\frac{1}{k\delta_{D}}}} \right) \leq \frac{1}{2} \label{cond_Hepsilon_shrink}
\end{multline}
Lastly, we collect the norms estimates overhead (\ref{norm_bds_shrink_1}), (\ref{norm_bds_shrink_2}), (\ref{norm_bds_shrink_3}) along with
(\ref{norm_bds_shrink_4}), (\ref{norm_bds_shrink_5}) under the requirement (\ref{cond_Hepsilon_shrink}) which leads to the contractive property
(\ref{H_epsilon_shrink}).

Conclusively, we select the radius $r_{Q,R_{D},1}>0$ and the constants $C_{\bf l}>0$, for ${\bf l} \in I$, in order that
(\ref{cond_Hepsilon_inclusion}) and (\ref{cond_Hepsilon_shrink}) are both achieved. Lemma 5 follows.
\end{proof}
We return to the proof of Proposition 3. For $\varpi>0$ chosen as in the lemma above, we set the closed ball
$\bar{B}(0,\varpi) \subset F_{(\nu_{2},\beta,\mu,k_{1})}^{d}$ which represents a complete metric space for the distance
$d(x,y) = ||x - y||_{(\nu_{2},\beta,\mu,k_{1})}$. According to the same lemma, we observe that $\mathcal{H}_{\epsilon}$ induces a contractive
application from $(\bar{B}(0,\varpi),d)$ into itself. Then, according to the classical contractive mapping theorem, the map
$\mathcal{H}_{\epsilon}$ possesses a unique fixed point that we set as $w^{d}(u,m,\epsilon)$, meaning that
\begin{equation}
\mathcal{H}_{\epsilon}(w^{d}(u,m,\epsilon) ) = w^{d}(u,m,\epsilon), \label{fixed_pt_Hepsilon}
\end{equation}
that belongs to the ball $\bar{B}(0,\varpi)$, for all $\epsilon \in D(0,\epsilon_{0})$. Furthermore, the function $w^{d}(u,m,\epsilon)$
depends holomorphically on $\epsilon$ in $D(0,\epsilon_{0})$. If one displaces the term
$$ R_{D}(im)k^{\delta_D}(k')^{m_D}u^{k\delta_{D}}w(u,m,\epsilon) $$
from the right to the left handside of (\ref{main_conv_w}) and then divide by the polynomial $P_{m}(u)$ defined in (\ref{defin_Pm}), we check
that (\ref{main_conv_w}) can be exactly recast as the equation (\ref{fixed_pt_Hepsilon}) above. As a result, the unique fixed point
$w^{d}(u,m,\epsilon)$ of $\mathcal{H}_{\epsilon}$ obtained overhead in $\bar{B}(0,\varpi)$ precisely solves the equation (\ref{main_conv_w}).
\end{proof}

\section{Solving the first auxiliary integro-differential equation}

The main purpose of this section is the construction of a solution of the integro-differential equation (\ref{main_diff_conv_W})
for vanishing initial data expressed as Laplace transform of order $k'$ that belongs to the Banach space disclosed in Definition 4.

\begin{prop} Let $w^{d}(u,m,\epsilon)$ be the unique solution of the integral equation (\ref{main_conv_w}) within the Banach space
$F^{d}_{(\nu_{2},\beta,\mu,k_{1})}$ built up in Proposition 3. We set up
\begin{equation}
W^{d}(\tau,m,\epsilon) = k' \int_{L_{\gamma}} w^{d}(u,m,\epsilon) \exp( -(\frac{u}{\tau})^{k'} ) \frac{du}{u} \label{Laplace_Wd}
\end{equation}
as the Laplace transform of $w^{d}(u,m,\epsilon)$ of order $k'$ in direction $d$ where
the halfline of integration $L_{\gamma} = \mathbb{R}_{+}e^{\sqrt{-1}\gamma}$ belongs to the sector $U_{d} \cup \{ 0 \}$. Then, for
all $\epsilon \in D(0,\epsilon_{0})$, the map $(\tau,m) \mapsto W^{d}(\tau,m,\epsilon)$ appertains to the Banach space
$E^{d}_{(\nu_{1},\beta,\mu,\kappa_{1})}$ where $S_{d}$ stands for an unbounded sector with bisecting direction $d$ and opening $\theta_{k'}$
that needs to fulfill
\begin{equation}
0 < \theta_{k'} < \frac{\pi}{k'} + \mathrm{Ap}(U_{d}) \label{cond_thetak'}
\end{equation}
for $\mathrm{Ap}(U_{d})$ defined as the aperture of the sector $U_{d}$. The real number $\nu_{1}>0$ is properly chosen and satisfies
$\nu_{1} > (1/T_{0})^{\kappa_1}$ for $T_{0}$ given in
(\ref{norm_bds_psi_n}) and $\kappa_{1}$ is introduced after (\ref{maj_quotient_Gamma_k_prime_k1}) under the condition
(\ref{kappa1_less_k}). Additionally, a constant $\varrho>0$ can be chosen with the bounds
\begin{equation}
|| W^{d}(\tau,m,\epsilon)||_{(\nu_{1},\beta,\mu,\kappa_{1})} \leq \varrho \label{norm_Wd_varrho}
\end{equation}
for all $\epsilon \in D(0,\epsilon_{0})$. Furthermore, $W^{d}(\tau,m,\epsilon)$ fulfills the first auxiliary integro-differential equation
(\ref{main_diff_conv_W}) on the domain $S_{d} \times \mathbb{R} \times D(0,\epsilon_{0})$.
\end{prop}
\begin{proof} According to the bounds (\ref{norm_wd_varpi}) and the very definition of the norm, we know in particular that
\begin{equation}
|w^{d}(u,m,\epsilon)| \leq \varpi (1 + |m|)^{-\mu} e^{-\beta |m|} |u| \exp( \nu_{2} |u|^{k_1} ) 
\end{equation}
holds for all $(u,m,\epsilon) \in (U_{d} \cup D(0,r)) \times \mathbb{R} \times D(0,\epsilon_{0})$. From the integral representation
(\ref{Laplace_Wd}) we deduce that
\begin{multline}
|W^{d}(\tau,m,\epsilon)| \leq k' \int_{0}^{+\infty} \varpi (1 + |m|)^{-\mu} e^{-\beta |m|}
\exp( \nu_{2}r^{k_1} ) \exp( - \frac{r^{k'}}{|\tau|^{k'}} \cos( k'(\gamma - \mathrm{arg}(\tau)) ) dr\\
\leq k'\int_{0}^{+\infty} \varpi (1 + |m|)^{-\mu} e^{-\beta |m|} \exp( \nu_{2} r^{k_1} ) \exp( -\frac{r^{k'}}{|\tau|^{k'}} \delta_{2} ) dr
\end{multline}
provided that $\tau \in S_{d}$ and that the direction $\gamma$ is well chosen (and may depend on $\tau$) in a way that
$\cos(k'(\gamma - \mathrm{arg}(\tau))) \geq \delta_{2}$
for some fixed constant $0 < \delta_{2} < 1$, close to 0, which is realizable under the condition (\ref{cond_thetak'}).

In the next step of the proof, we are reduced to supply bounds for the auxiliary function
$$ \mathbb{E}(x) = \int_{0}^{+\infty} \exp( \nu_{2}r^{k_1} ) \exp( -\frac{r^{k'}}{x} ) dr $$
when $x > 0$, especially for large values of $x$. Indeed, we first expand
$$ \exp( \nu_{2} r^{k_1} ) = \sum_{n \geq 0} \nu_{2}^{n} r^{k_{1}n}/n! $$
for all $r \geq 0$. By dominated convergence, we deduce that
$$ \mathbb{E}(x) = \sum_{n \geq 0} \frac{\nu_{2}^{n}}{n!} \int_{0}^{+\infty} r^{k_{1}n} \exp( -\frac{r^{k'}}{x} ) dr $$
for all $x > 0$. This last expression, allows us to compute explicitely the series expansion $\mathbb{E}(x)$ w.r.t $x$ in terms of the Gamma function. Namely, by performing
the change of variable $r^{k'}/x=\tilde{r}$, we get that
$$ \int_{0}^{+\infty} r^{k_{1}n} \exp( -\frac{r^{k'}}{x} ) dr = \frac{1}{k'} x^{\frac{k_1}{k'}n + \frac{1}{k'}}
\int_{0}^{+\infty} (\tilde{r})^{\frac{k_1}{k'}n + \frac{1}{k'} - 1} e^{-\tilde{r}} d\tilde{r} =
\frac{1}{k'} x^{\frac{1}{k'} + \frac{k_{1}}{k'}n}\Gamma( \frac{k_1}{k'}n + \frac{1}{k'})
$$
for all $n \geq 0$, by definition of the Gamma function. Therefore, we can recast
$$ \mathbb{E}(x) = \frac{1}{k'} x^{1/k'} \sum_{n \geq 0} \frac{\Gamma(\frac{k_1}{k'}n + \frac{1}{k'})}{\Gamma(n+1)} (\nu_{2}
x^{\frac{k_1}{k'}})^{n} $$
for all $x > 0$. Bearing in mind the inequality (\ref{quotient_Gamma_alpha_beta}) for the special case
$$ \alpha = \frac{k_1}{k'}n + \frac{1}{k'} \ \ , \ \ \beta = (1 - \frac{k_1}{k'})n + 1 - \frac{1}{k'} $$
we observe that
$$ \frac{\Gamma(\frac{k_1}{k'}n + \frac{1}{k'})}{\Gamma(n+1)} \leq \frac{1}{\Gamma((1 - \frac{k_1}{k'})n + 1 - \frac{1}{k'})} $$
for all $n \geq \max( (k'-1)/k_{1}, 1/(k' - k_{1}) )$. Henceforth, we can bound $\mathbb{E}(x)$ by a Wiman function as follows
\begin{equation}
\mathbb{E}(x) \leq \frac{C_{k_{1},k'}^{1}}{k'} x^{1/k'} \sum_{n \geq 0}
\frac{(\nu_{2}x^{\frac{k_1}{k'}})^{n}}{\Gamma( (1 - \frac{k_1}{k'})n + 1 - \frac{1}{k'})} \label{bds_mathbbE_small_x}
\end{equation}
for some constant $C_{k_{1},k'}^{1}>0$ (depending on $k_{1},k'$), for all $x > 0$. We again require the bounds for the Wiman function $E_{\alpha,\beta}(z)$ for large values of $z$ already
mentioned above, see (\ref{bds_E_alpha_beta}). As a result, a constant $C_{k_{1},k'}>0$ (depending on $k_{1},k'$) can be found such that
\begin{equation}
\mathbb{E}(x) \leq C_{k_{1},k'} \nu_{2}^{\frac{1}{k' - k_{1}}} x^{\frac{1}{k'} + \frac{k_{1}}{k'(k'-k_{1})}}
\exp \left( \nu_{2}^{\frac{1}{1 - \frac{k_1}{k'}}} x^{\frac{k_1}{k' - k_{1}}} \right) \label{bds_mathbbE_large_x}
\end{equation}
whenever $x \geq (1/\nu_{2})^{k'/k_{1}}$.

These two last upper bounds (\ref{bds_mathbbE_small_x}) and (\ref{bds_mathbbE_large_x}) give rise to estimates for
$W^{d}(\tau,m,\epsilon)$. Namely, we get that
\begin{equation}
|W^{d}(\tau,m,\epsilon)| \leq \varpi C_{k_{1},k'}^{1} (1 + |m|)^{-\mu} e^{-\beta|m|} \frac{|\tau|}{\delta_{2}^{1/k'}}
\sum_{n \geq 0} \frac{(\nu_{2} \frac{|\tau|^{k_1}}{(\delta_{2})^{k_{1}/k'}})^{n} }{\Gamma( (1 - \frac{k_1}{k'})n + 1 - \frac{1}{k'})}
\label{bds_Wd_small_tau}
\end{equation}
for all $\tau \in S_{d}$, all $m \in \mathbb{R}$, all $\epsilon \in D(0,\epsilon_{0})$ together with
\begin{multline}
|W^{d}(\tau,m,\epsilon)| \leq \varpi k' C_{k_{1},k'} (\nu_{2})^{\frac{1}{k' - k_{1}}}
(\frac{1}{\delta_{2}})^{\frac{1}{k'} + \frac{k_{1}}{k'(k'-k_{1})}}
(1 + |m|)^{-\mu} e^{-\beta |m|} \\
\times |\tau|^{1 + \frac{k_{1}}{k' - k_{1}}}
\exp \left( \nu_{2}^{\frac{1}{1 - \frac{k_{1}}{k'}}} (\frac{1}{\delta_{2}})^{\frac{k_{1}}{k' - k_{1}}} |\tau|^{\frac{k'k_{1}}{k' - k_{1}}}
\right)\\
= \varpi k' C_{k_{1},k'} (\nu_{2})^{\frac{1}{k' - k_{1}}}
(\frac{1}{\delta_{2}})^{\frac{1}{k'} + \frac{k_{1}}{k'(k'-k_{1})}}
(1 + |m|)^{-\mu} e^{-\beta |m|} |\tau|^{1 + \frac{k_{1}}{k' - k_{1}}}\\
\times \exp \left( (\frac{1}{T_{0}})^{\kappa_1} (\frac{1}{\delta_{2}})^{\frac{k_{1}}{k' - k_{1}}} |\tau|^{\kappa_1} \right)
\label{bds_Wd_large_tau}
\end{multline}
provided that $\tau \in S_{d}$ with $|\tau| \geq \delta_{2}^{1/k'}(1/\nu_{2})^{1/k_{1}}$, $m \in \mathbb{R}$ and
$\epsilon \in D(0,\epsilon_{0})$.

Collecting the bounds (\ref{bds_Wd_small_tau}) for small values of $|\tau|$ and (\ref{bds_Wd_large_tau}) for large values of
$|\tau|$ implies that for all $\epsilon \in D(0,\epsilon_{0})$, the function $(\tau,m) \mapsto W^{d}(\tau,m,\epsilon)$ belongs to
$E^{d}_{(\nu_{1},\beta,\mu,\kappa_{1})}$ when $\nu_{1}>0$ is taken such that
$$ \nu_{1} > (\frac{1}{T_{0}})^{\kappa_1} (\frac{1}{\delta_{2}})^{\frac{k_{1}}{k' - k_{1}}} $$
Moreover, we can find a constant $\varrho>0$ with the estimates (\ref{norm_Wd_varrho}) uniformly in $\epsilon \in D(0,\epsilon_{0})$.

In order to check that $W^{d}(\tau,m,\epsilon)$ fulfills the equation (\ref{main_diff_conv_W}), we follow backwards step by step the construction displayed in Section 3. Namely, since $w^{d}(u,m,\epsilon)$ solves (\ref{main_conv_w}) and belongs to $F^{d}_{(\nu_{2},\beta,\mu,k_{1})}$, the map
$W^{d}(\tau,m,\epsilon)$ solves the integral equation in prepared form (\ref{main_diff_conv_prep_form}) according to the identities of Lemma 2.
Owing to the formulas (\ref{Tahara_form_1}), (\ref{Tahara_form_2}) and (\ref{Tahara_form_3}), we deduce that $W^{d}(\tau,m,\epsilon)$
solves (\ref{main_diff_conv_1}). At last, Lemma 3 allows us to write (\ref{main_diff_conv_1}) in the form
(\ref{main_diff_conv_W}) and we can conclude that $W^{d}(\tau,m,\epsilon)$ is a solution of the first main integro-differential equation
(\ref{main_diff_conv_W}) on the domain $S_{d} \times \mathbb{R} \times D(0,\epsilon_{0})$.
\end{proof}

\section{Analytic solutions on sectors to the main initial value problem}

We revisit the first step of the formal constructions realized in Section 3 in view of the progress made in solving the two auxiliary
problems (\ref{main_conv_w}) and (\ref{main_diff_conv_W}) throughout the above sections 4 and 5.

We need to remind the reader the definition of a good covering in $\mathbb{C}^{\ast}$ and we introduce an adapted version of a so-called
associated sets of sectors to a good covering as proposed in our previous work, \cite{lama1}.

\begin{defin} Let $\varsigma \geq 2$ be an integer. We consider a set $\underline{\mathcal{E}}$ of open sectors
$\mathcal{E}_{p}$ centered at 0, with radius $\epsilon_{0}>0$ for all $0 \leq p \leq \varsigma - 1$
owning the next three properties:\\
i) the intersection $\mathcal{E}_{p} \cap \mathcal{E}_{p+1}$ is not empty for all
$0 \leq p \leq \varsigma-1$ (with the convention that $\mathcal{E}_{\varsigma} = \mathcal{E}_{0}$),\\
ii) the intersection of any three elements of $\underline{\mathcal{E}}$ is empty,\\
iii) the union $\cup_{p=0}^{\varsigma - 1} \mathcal{E}_{p}$ equals $\mathcal{U} \setminus \{ 0 \}$ for some neighborhood $\mathcal{U}$ of 0
in $\mathbb{C}$.\\ 
Then, the set of sectors $\underline{\mathcal{E}}$ is called a good covering of $\mathbb{C}^{\ast}$.
\end{defin}

\begin{defin} We select\\
a) a good covering $\underline{\mathcal{E}} = \{ \mathcal{E}_{p} \}_{0 \leq p \leq \varsigma-1}$ of $\mathbb{C}^{\ast}$,\\
b) a set $\underline{U}$ of unbounded sectors $U_{\mathfrak{d}_{p}}$, $0 \leq p \leq \varsigma-1$ centered at 0 with bisecting direction
$\mathfrak{d}_{p} \in \mathbb{R}$ and small opening $\theta_{U_{\mathfrak{d}_{p}}}>0$, \\
c) a set $\underline{S}$ of unbounded sectors $S_{\mathfrak{d}_{p}}$, $0 \leq p \leq \varsigma-1$ centered at 0 with bisecting direction
$\mathfrak{d}_{p} \in \mathbb{R}$ and aperture $0< \theta_{S_{\mathfrak{d}_{p}}} < \frac{\pi}{k'} + \theta_{U_{\mathfrak{d}_{p}}}$ for
some integer $k' \geq 1$,\\
d) a fixed bounded sector $\mathcal{T}$ centered at 0 with radius $r_{\mathcal{T}}>0$ and a disc $D(0,r)$,\\
suitably selected in a way that the next features are conjointly satisfied:\\
1) the bounds (\ref{dist_Ud_Dr_qlm_M1}) and (\ref{dist_Ud_Dr_qlm_M2}) are fulfilled provided that $u \in U_{\mathfrak{d}_{p}} \cup D(0,r)$,
for all $0 \leq p \leq \varsigma-1$,\\
2) the set $\underline{S}$ fulfills the next properties:\\
2.1) the intersection $S_{\mathfrak{d}_{p}} \cap S_{\mathfrak{d}_{p+1}}$ is not empty for all
$0 \leq p \leq \varsigma-1$ (with the convention that $S_{\mathfrak{d}_{\varsigma}} = S_{\mathfrak{d}_{0}}$),\\
2.2) the union $\cup_{p=0}^{\varsigma - 1} S_{\mathfrak{d}_{p}}$ equals $\mathbb{C} \setminus \{ 0 \}$.\\
3) for all $\epsilon \in \mathcal{E}_{p}$, all $t \in \mathcal{T}$,
\begin{equation}
\epsilon t \in S_{\mathfrak{d}_{p},\theta_{k,k'},\epsilon_{0}r_{\mathcal{T}}} \label{epsilon_t_in_Sdpthetak}
\end{equation}
where $S_{\mathfrak{d}_{p},\theta_{k,k'},\epsilon_{0}r_{\mathcal{T}}}$ stands for a bounded sector with bisecting direction
$\mathfrak{d}_{p}$, opening $\theta_{k,k'}>0$ that fulfills $0 < \theta_{k,k'} < \frac{\pi}{k} + \theta_{S_{\mathfrak{d}_{p}}}$
and radius $\epsilon_{0}r_{\mathcal{T}}$, for all $0 \leq p \leq \varsigma-1$.

When the above properties are fulfilled, we say that the set of data
$\{ \underline{\mathcal{E}}, \underline{U}, \underline{S}, \mathcal{T}, D(0,r) \}$ is admissible. 
\end{defin}

We state now the first main result of the work. We build up a family of actual holomorphic solutions to the main initial value problem
(\ref{main_PDE_u}) defined on sectors $\mathcal{E}_{p}$, $0 \leq p \leq \varsigma-1$, of a good covering in $\mathbb{C}^{\ast}$. Upper
control for the difference between consecutive solutions on the intersections $\mathcal{E}_{p} \cap \mathcal{E}_{p+1}$ is also given.

\begin{theo} Take for granted that next list of requirements (\ref{rel_k_deltaD_mD_k_prime}), (\ref{rel_kl_one_l_two_k_prime}),
(\ref{rel_Deltal_k}), (\ref{constraints_degree_coeff}), (\ref{defin_Cl}), (\ref{norm_bds_psi_n}), (\ref{kappa1_less_k})
(\ref{quotient_Q_RD_in_S}) and (\ref{cond_exist_uniq_wd}) is fulfilled. We fix an admissible set of data
$$ \underline{\mathcal{A}} = \{ \underline{\mathcal{E}} = \{ \mathcal{E}_{p} \}_{0 \leq p \leq \varsigma-1},
\underline{U} = \{ U_{\mathfrak{d}_{p}} \}_{0 \leq p \leq \varsigma-1},
\underline{S} = \{ S_{\mathfrak{d}_{p}} \}_{0 \leq p \leq \varsigma-1}, \mathcal{T}, D(0,r) \} $$
as described in Definition 7.

Then, whenever the inner radius $r_{Q,R_{D},1}>0$ (see \ref{quotient_Q_RD_in_S}) is selected large enough
and the constants $C_{\bf l}>0$ (see \ref{defin_Cl}) are chosen close enough to 0 for all ${\bf l} \in I$, a collection
$\{ u_{p}(t,z,\epsilon) \}_{0 \leq p \leq \varsigma - 1}$ of genuine solutions of (\ref{main_PDE_u}) can be set up. In particular, each function
$u_{p}(t,z,\epsilon)$ defines a bounded holomorphic application on the product $(\mathcal{T} \cap D(0,\sigma)) \times H_{\beta'} \times
\mathcal{E}_{p}$ for any given $0 < \beta' < \beta$ and suitable tiny radius $\sigma>0$. Furthermore, $u_{p}(t,z,\epsilon)$ can be
expressed as a Laplace transform of order $k$ and Fourier inverse transform
\begin{equation}
u_{p}(t,z,\epsilon) = \frac{k}{(2\pi)^{1/2}} \int_{-\infty}^{+\infty} \int_{L_{\gamma_p}} W^{\mathfrak{d}_{p}}(\tau,m,\epsilon)
\exp( -(\frac{\tau}{\epsilon t})^{k} ) e^{izm} \frac{d\tau}{\tau} dm \label{Laplace_up_theo}
\end{equation}
along a halfline $L_{\gamma_p} = \mathbb{R}_{+}e^{\sqrt{-1}\gamma_{p}} \subset S_{\mathfrak{d}_{p}} \cup \{ 0 \}$. The map
$(\tau,m) \mapsto W^{\mathfrak{d}_{p}}(\tau,m,\epsilon)$ represents a function that belongs to the Banach space
$E^{\mathfrak{d}_{p}}_{(\nu_{1},\beta,\mu,\kappa_{1})}$ for a well chosen $\nu_{1} > (1/T_{0})^{\kappa_{1}}$ for all
$\epsilon \in D(0,\epsilon_{0})$ and can itself be recast as a Laplace transform of order $k'$
\begin{equation}
W^{\mathfrak{d}_p}(\tau,m,\epsilon) = k' \int_{L_{\gamma_{p}'}} w^{\mathfrak{d}_{p}}(u,m,\epsilon)
\exp( -(\frac{u}{\tau})^{k'} ) \frac{du}{u}
\end{equation}
where the integration path $L_{\gamma_{p}'}$ is taken inside $U_{\mathfrak{d}_{p}} \cup \{ 0 \}$ and where
$(u,m) \mapsto w^{\mathfrak{d}_p}(u,m,\epsilon)$ stands for a function built within the Banach space
$F^{\mathfrak{d}_p}_{(\nu_{2},\beta,\mu,k_{1})}$ for all $\epsilon \in D(0,\epsilon_{0})$.

In addition, one can choose constants $K_{p},M_{p}>0$ and a radius $0 < \sigma' < \sigma$ (independent of $\epsilon$) with
\begin{equation}
 \sup_{t \in \mathcal{T} \cap D(0,\sigma'), z \in H_{\beta'}} |u_{p+1}(t,z,\epsilon) - u_{p}(t,z,\epsilon)| \leq
 K_{p} \exp( - \frac{M_{p}}{|\epsilon|^{\frac{kk'}{k+k'}}} ) \label{exp_flat_difference_up}
\end{equation}
for all $\epsilon \in \mathcal{E}_{p+1} \cap \mathcal{E}_{p}$, all $0 \leq p \leq \varsigma-1$, where by convention, we set
$u_{\varsigma}(t,z,\epsilon) = u_{0}(t,z,\epsilon)$.
\end{theo}
\begin{proof} At the onset, we depart from an admissible set of data $\underline{\mathcal{A}}$. Under the conditions asked in the statement
of Theorem 1, we can apply Proposition 4 in order to find, for all $0 \leq p \leq \varsigma -1$, a function
\begin{equation}
W^{\mathfrak{d}_{p}}(\tau,m,\epsilon) = k'\int_{L_{\gamma_{p}'}} w^{\mathfrak{d}_{p}}(u,m,\epsilon) \exp( -(\frac{u}{\tau})^{k'} ) \frac{du}{u}
\label{Wdp_Laplace_repres}
\end{equation}
written as a Laplace transform of order $k'$ in direction $\gamma_{p}'$ with
$L_{\gamma_{p}'}=\mathbb{R}_{+}e^{\sqrt{-1}\gamma_{p}'} \subset U_{\mathfrak{d}_{p}} \cup \{ 0 \}$ of a map
$w^{\mathfrak{d}_{p}}(u,m,\epsilon)$ which turns out to be holomorphic w.r.t $u$ on $U_{\mathfrak{d}_p} \cup D(0,r)$ and
w.r.t $\epsilon$ on $D(0,\epsilon_{0})$, continuous w.r.t $m$
on $\mathbb{R}$, with the property that a constant $\varpi^{\mathfrak{d}_p}>0$ can be singled out with
\begin{equation}
|w^{\mathfrak{d}_p}(u,m,\epsilon)| \leq \varpi^{\mathfrak{d}_p} (1 + |m|)^{-\mu} e^{-\beta |m|}|u| \exp( \nu_{2} |u|^{k_1} )
\label{bds_w_frac_dp}
\end{equation}
for all $u \in U_{\mathfrak{d}_p} \cup D(0,r)$, $m \in \mathbb{R}$, $\epsilon \in D(0,\epsilon_{0})$, where
$\nu_{2} = (1/T_{0})^{k_1}$. The function $W^{\mathfrak{d}_p}(\tau,m,\epsilon)$ is built in a way that it solves the first auxiliary
integro-differential equation (\ref{main_diff_conv_W}) on the domain $S_{\mathfrak{d}_{p}} \times \mathbb{R} \times D(0,\epsilon_{0})$ and is
submitted to the next bounds
\begin{equation}
|W^{\mathfrak{d}_{p}}(\tau,m,\epsilon)| \leq \varrho^{\mathfrak{d}_p} (1 + |m|)^{-\mu} e^{-\beta |m|} |\tau|
\exp( \nu_{1} |\tau|^{\kappa_1} ) \label{bds_W_frac_dp}
\end{equation}
whenever $\tau \in S_{\mathfrak{d}_{p}}$, $m \in \mathbb{R}$ and $\epsilon \in D(0,\epsilon_{0})$, for a well chosen
$\nu_{1} > (1/T_{0})^{\kappa_1}$.

We now turn back to the first step of the formal construction discussed in Section 3. We consider the next Laplace transform and Fourier
inverse transform
$$ U_{\gamma_{p}}(T,z,\epsilon) = \frac{k}{(2\pi)^{1/2}} \int_{-\infty}^{+\infty} \int_{L_{\gamma_p}}
W^{\mathfrak{d}_{p}}(\tau,m,\epsilon) \exp( -(\frac{\tau}{T})^{k} ) e^{izm} \frac{d\tau}{\tau} dm $$
along a halfline $L_{\gamma_p} = \mathbb{R}_{+}e^{\sqrt{-1}\gamma_{p}} \subset S_{\mathfrak{d}_{p}} \cup \{ 0 \}$. According to the
upper bounds (\ref{bds_W_frac_dp}) and the basic properties of Laplace and Fourier inverse transforms outlined in Section 2, we get that
$U_{\gamma_{p}}(T,z,\epsilon)$ defines\\
1) a holomorphic bounded function w.r.t $T$ on a sector $S_{\mathfrak{d}_{p},\theta_{k,k'},\upsilon}$ with bisecting direction $\mathfrak{d}_{p}$,
aperture $0 < \theta_{k,k'} < \frac{\pi}{k} + \mathrm{Ap}(S_{\mathfrak{d}_{p}})$, for some small radius $\upsilon>0$, where
$\mathrm{Ap}(S_{\mathfrak{d}_{p}})$ stands for the aperture of $S_{\mathfrak{d}_p}$,\\
2) a holomorphic bounded application w.r.t $z$ on $H_{\beta'}$,\\
3) a holomorphic bounded map w.r.t $\epsilon$ on $D(0,\epsilon_{0})$.\\
Furthermore, since $W^{\mathfrak{d}_p}(\tau,m,\epsilon)$ fulfills the equation (\ref{main_diff_conv_W}), Lemma 1 allows us to assert that
$U_{\gamma_p}(T,z,\epsilon)$ must solve the equation (\ref{main_PDE_U}) on
$S_{\mathfrak{d}_{p},\theta_{k,k'},\upsilon} \times H_{\beta'} \times D(0,\epsilon_{0})$. As a result, the function
$$ u_{p}(t,z,\epsilon) = U_{\gamma_p}(\epsilon t,z,\epsilon) $$
represents a bounded holomorphic function w.r.t $t$ on $\mathcal{T} \cap D(0,\sigma)$ for some $\sigma>0$ small enough,
$\epsilon \in \mathcal{E}_{p}$, $z \in H_{\beta'}$ for any given $0 < \beta' < \beta$, keeping in mind that the sectors
$\mathcal{E}_{p}$ and $\mathcal{T}$ suffer the restriction (\ref{epsilon_t_in_Sdpthetak}). Moreover, $u_{p}(t,z,\epsilon)$ solves the main
initial value problem (\ref{main_PDE_u}) on the domain $(\mathcal{T} \cap D(0,\sigma)) \times H_{\beta'} \times \mathcal{E}_{p}$, for all
$0 \leq p \leq \varsigma-1$.

In the second part of the proof, we concentrate on the exponential bounds (\ref{exp_flat_difference_up}). For $l=p,p+1$, the map
$\tau \mapsto W^{\mathfrak{d}_{l}}(\tau,m,\epsilon)\exp( -(\frac{\tau}{\epsilon t})^{k} )/\tau$ is holomorphic on the sector
$S_{\mathfrak{d}_{l}}$. As a result, we can deform each straight halfline $L_{\gamma_{l}}$, for $l=p,p+1$, into the union of three pieces with suitable orientation,
described as follows:\\
a) a halfline $L_{\gamma_{l},r_{1}} = [r_{1},+\infty) \exp( \sqrt{-1}\gamma_{l} )$ for a given real number $r_{1}>0$,\\
b) an arc of circle with radius $r_{1}$ denoted $C_{r_{1},\gamma_{l},\gamma_{p,p+1}}$ joining the point $r_{1}\exp( \sqrt{-1}\gamma_{p,p+1})$
which is taken inside the intersection $S_{\mathfrak{d}_{p}} \cap S_{\mathfrak{d}_{p+1}}$
(that is assumed to be non empty, see Definition 7, 2.1) to the halfline $L_{\gamma_{l},r_{1}}$,\\
c) a segment $L_{\gamma_{p,p+1},0,r_{1}} = [0,r_{1}] \exp( \sqrt{-1} \gamma_{p,p+1} )$.

See Figure~\ref{fig1} for the configuration of the deformation of the integration paths.

\begin{figure}[h]
	\centering
		\includegraphics[width=0.48\textwidth]{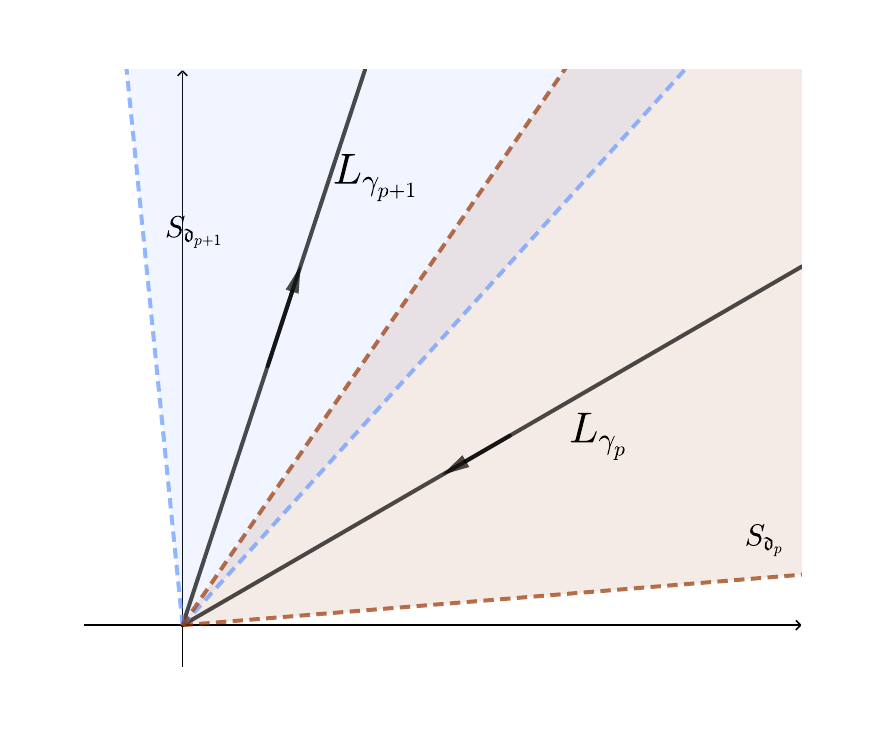}
		\includegraphics[width=0.48\textwidth]{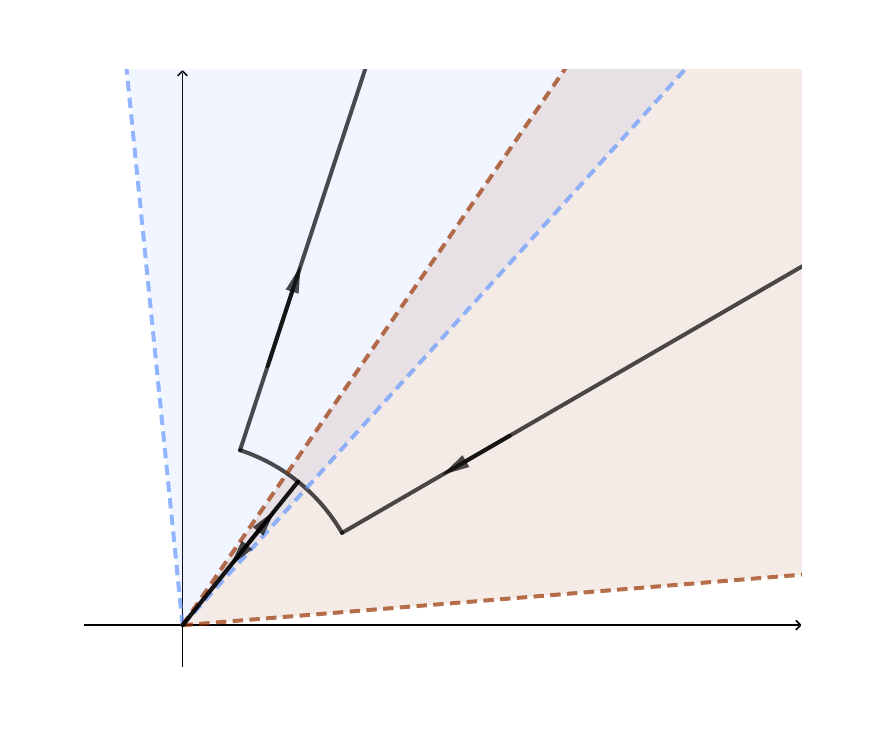}
	\caption{Initial (left) and deformation (right) of the integration paths}
			\label{fig1}
\end{figure}

We notice that the deformation paths are similar to those performed in the proof of Theorem 1 from \cite{lama2}.
Consequently, we are able to split the difference $u_{p+1}-u_{p}$ into five parts, namely
\begin{multline}
u_{p+1}(t,z,\epsilon) - u_{p}(t,z,\epsilon) = \frac{k}{(2\pi)^{1/2}}
\int_{-\infty}^{+\infty} \int_{L_{\gamma_{p+1}},r_{1}} W^{\mathfrak{d}_{p+1}}(\tau,m,\epsilon)
\exp( -(\frac{\tau}{\epsilon t})^{k} ) e^{izm} \frac{d\tau}{\tau} dm\\
-
\frac{k}{(2\pi)^{1/2}}
\int_{-\infty}^{+\infty} \int_{L_{\gamma_{p}},r_{1}} W^{\mathfrak{d}_{p}}(\tau,m,\epsilon)
\exp( -(\frac{\tau}{\epsilon t})^{k} ) e^{izm} \frac{d\tau}{\tau} dm \\
+
\frac{k}{(2\pi)^{1/2}}
\int_{-\infty}^{+\infty} \int_{C_{r_{1},\gamma_{p+1},\gamma_{p,p+1}}} W^{\mathfrak{d}_{p+1}}(\tau,m,\epsilon)
\exp( -(\frac{\tau}{\epsilon t})^{k} ) e^{izm} \frac{d\tau}{\tau} dm\\
- \frac{k}{(2\pi)^{1/2}}
\int_{-\infty}^{+\infty} \int_{C_{r_{1},\gamma_{p},\gamma_{p,p+1}}} W^{\mathfrak{d}_{p}}(\tau,m,\epsilon)
\exp( -(\frac{\tau}{\epsilon t})^{k} ) e^{izm} \frac{d\tau}{\tau} dm\\
+ \frac{k}{(2\pi)^{1/2}}
\int_{-\infty}^{+\infty} \int_{L_{\gamma_{p,p+1},0,r_{1}}}
\left( W^{\mathfrak{d}_{p+1}}(\tau,m,\epsilon) - W^{\mathfrak{d}_{p}}(\tau,m,\epsilon) \right)
\exp( -(\frac{\tau}{\epsilon t})^{k} ) e^{izm} \frac{d\tau}{\tau} dm \label{split_diff_up_five_parts}
\end{multline}
We first provide estimates for the quantity
$$
I_{1} = \left|  \frac{k}{(2\pi)^{1/2}}
\int_{-\infty}^{+\infty} \int_{L_{\gamma_{p+1}},r_{1}} W^{\mathfrak{d}_{p+1}}(\tau,m,\epsilon)
\exp( -(\frac{\tau}{\epsilon t})^{k} ) e^{izm} \frac{d\tau}{\tau} dm   \right|
$$
Observe that the direction $\gamma_{p+1}$ (which relies on $\epsilon t$) is chosen in a way that
$$ \cos(k(\gamma_{p+1} - \mathrm{arg}(\epsilon t))) \geq \delta_{1} $$
for all $\epsilon \in \mathcal{E}_{p+1} \cap \mathcal{E}_{p}$, all $t \in \mathcal{T}$, for some fixed $\delta_{1}>0$. Owing to the
bounds (\ref{bds_W_frac_dp}), we deduce
\begin{multline}
I_{1} \leq \frac{k}{(2\pi)^{1/2}} \int_{-\infty}^{+\infty} \int_{r_{1}}^{+\infty}
\varrho^{\mathfrak{d}_{p+1}} (1 + |m|)^{-\mu} e^{-\beta |m|} r \exp( \nu_{1} r^{\kappa_{1}} )\\
\times 
\exp\left( -\frac{ \cos( k(\gamma_{p+1} - \mathrm{arg}(\epsilon t)) ) }{|\epsilon t|^{k}} r^{k} \right)
\exp(-m \mathrm{Im}(z) ) \frac{dr}{r} dm \leq \frac{k \varrho^{\mathfrak{d}_{p+1}}}{(2\pi)^{1/2}}
\int_{-\infty}^{+\infty} e^{-(\beta - \beta')|m|} dm\\
\times 
\int_{r_1}^{+\infty} \exp \left( - (\frac{\delta_{1}}{|t|^{k}} - \nu_{1}|\epsilon|^{k}r^{\kappa_{1}-k}) \frac{r^k}{|\epsilon|^{k}} \right)
dr \leq \frac{2k \varrho^{\mathfrak{d}_{p+1}}}{(2\pi)^{1/2}} \int_{0}^{+\infty} e^{-(\beta - \beta')m} dm\\
\times 
\int_{r_1}^{+\infty} \exp \left( - (\frac{\delta_{1}}{|t|^{k}} - \nu_{1}|\epsilon|^{k}r_{1}^{\kappa_{1}-k}) \frac{r^k}{|\epsilon|^{k}} \right)
dr \leq \frac{2k \varrho^{\mathfrak{d}_{p+1}}}{(2\pi)^{1/2}(\beta - \beta')}
\int_{r_{1}}^{+\infty} \left( \frac{|\epsilon|^{k}}{\frac{\delta_{1}}{|t|^{k}} - \nu_{1} |\epsilon|^{k} r_{1}^{\kappa_{1} - k}}
\frac{1}{k r_{1}^{k-1}} \right)\\
\times \frac{ \frac{\delta_1}{|t|^{k}} - \nu_{1} |\epsilon|^{k} r_{1}^{\kappa_{1} - k} }{|\epsilon|^k}kr^{k-1}
\exp \left( -(\frac{\delta_{1}}{|t|^{k}} - \nu_{1}|\epsilon|^{k} r_{1}^{\kappa_{1} - k}) \frac{r^k}{|\epsilon|^{k}} \right) dr
\leq \frac{2k \varrho^{\mathfrak{d}_{p+1}}}{(2\pi)^{1/2}(\beta - \beta')}\\
\times 
\frac{|\epsilon|^{k}}{\frac{\delta_{1}}{|t|^{k}} - \nu_{1} |\epsilon|^{k} r_{1}^{\kappa_{1} - k}}\frac{1}{k r_{1}^{k-1}}
\exp \left( -( \frac{\delta_{1}}{|t|^{k}} - \nu_{1}|\epsilon|^{k}r_{1}^{\kappa_{1}-k}) \frac{r_{1}^{k}}{|\epsilon|^{k}} \right)\\
\leq \frac{2k \varrho^{\mathfrak{d}_{p+1}}}{(2\pi)^{1/2}(\beta - \beta')} \frac{|\epsilon|^{k}}{\delta_{2} k r_{1}^{k-1}}
\exp( - \delta_{2} \frac{r_{1}^{k}}{|\epsilon|^{k}} ) \label{bds_I1}
\end{multline}
for all $t \in \mathcal{T}$ and $|\mathrm{Im}(z)| \leq \beta'$ under the requirement that
\begin{equation}
|t| < (\frac{\delta_{1}}{\delta_{2} + \nu_{1} \epsilon_{0}^{k} r_{1}^{\kappa_{1}-k}})^{1/k} \label{t_small_in_mathcalT}
\end{equation}
for some given $\delta_{2}>0$, for all $\epsilon \in \mathcal{E}_{p+1} \cap \mathcal{E}_{p}$.

In a similar manner, we can supply upper bounds for the next term
$$
I_{2} = \left|  \frac{k}{(2\pi)^{1/2}}
\int_{-\infty}^{+\infty} \int_{L_{\gamma_{p}},r_{1}} W^{\mathfrak{d}_{p}}(\tau,m,\epsilon)
\exp( -(\frac{\tau}{\epsilon t})^{k} ) e^{izm} \frac{d\tau}{\tau} dm   \right|
$$
Indeed, the direction $\gamma_{p}$ (which depends on $\epsilon t$) is taken in order that
$$ \cos( k(\gamma_{p} - \mathrm{arg}(\epsilon t)) ) \geq \delta_{1} $$
for all $\epsilon \in \mathcal{E}_{p+1} \cap \mathcal{E}_{p}$, all $t \in \mathcal{T}$, for some fixed $\delta_{1}>0$. Again with the
estimates (\ref{bds_W_frac_dp}), the same steps as above (\ref{bds_I1}) yield
\begin{equation}
I_{2} \leq \frac{2k \varrho^{\mathfrak{d}_{p}}}{(2\pi)^{1/2}(\beta - \beta')} \frac{|\epsilon|^{k}}{\delta_{2} k r_{1}^{k-1}}
\exp( - \delta_{2} \frac{r_{1}^{k}}{|\epsilon|^{k}} ) \label{bds_I2}
\end{equation}
provided that $t \in \mathcal{T}$ and $|\mathrm{Im}(z)| \leq \beta'$ under the constraint (\ref{t_small_in_mathcalT}) for some $\delta_{2}>0$.

In the next step, we control the first integral along an arc of circle
$$
I_{3} = \left| \frac{k}{(2\pi)^{1/2}}
\int_{-\infty}^{+\infty} \int_{C_{r_{1},\gamma_{p+1},\gamma_{p,p+1}}} W^{\mathfrak{d}_{p+1}}(\tau,m,\epsilon)
\exp( -(\frac{\tau}{\epsilon t})^{k} ) e^{izm} \frac{d\tau}{\tau} dm \right|
$$
By construction, the arc of circle $C_{r_{1},\gamma_{p+1},\gamma_{p,p+1}}$ is built in order that
$$ \cos( k(\theta - \mathrm{arg}(\epsilon t)) ) \geq \delta_{1} $$
for all $\theta \in [\gamma_{p+1},\gamma_{p,p+1}]$ (if $\gamma_{p+1} < \gamma_{p,p+1}$) or $\theta \in [\gamma_{p,p+1},\gamma_{p+1}]$
(if $\gamma_{p,p+1} < \gamma_{p+1}$), whenever $t \in \mathcal{T}$, $\epsilon \in \mathcal{E}_{p} \cap \mathcal{E}_{p+1}$,
for some fixed $\delta_{1}>0$. Keeping in mind (\ref{bds_W_frac_dp}), we obtain
\begin{multline}
I_{3} \leq \frac{k}{(2\pi)^{1/2}} \int_{-\infty}^{+\infty} \left| \int_{\gamma_{p+1}}^{\gamma_{p,p+1}}
\varrho^{\mathfrak{d}_{p+1}} (1 + |m|)^{-\mu} e^{-\beta |m|} r_{1} \exp( \nu_{1} r_{1}^{\kappa_1} ) \right.\\
\times 
\left. \exp \left( - \frac{\mathrm{cos}(k(\theta - \mathrm{arg}(\epsilon t)))}{|\epsilon t|^{k}} r_{1}^{k} \right)
e^{-m \mathrm{Im}(z)} d\theta \right| dm \leq \frac{k \varrho^{\mathfrak{d}_{p+1}}}{(2\pi)^{1/2}}
\int_{-\infty}^{+\infty} e^{-(\beta - \beta')|m|} dm
|\gamma_{p+1} - \gamma_{p,p+1}| \\
\times r_{1} \exp \left( - (\frac{\delta_{1}}{|t|^{k}} - \nu_{1} |\epsilon|^{k} r_{1}^{\kappa_{1} - k} )
\frac{r_{1}^{k}}{|\epsilon|^{k}} \right) \leq \frac{2 k \varrho^{\mathfrak{d}_{p+1}}}{(2\pi)^{1/2}(\beta - \beta')}
|\gamma_{p+1} - \gamma_{p,p+1}|r_{1} \exp( -\delta_{2} \frac{r_{1}^{k}}{|\epsilon|^{k}} ) \label{bds_I3}
\end{multline}
for all $t \in \mathcal{T}$, $|\mathrm{Im}(z)| \leq \beta'$ submitted to (\ref{t_small_in_mathcalT}) for some fixed $\delta_{2}>0$,
whenever $\epsilon \in \mathcal{E}_{p+1} \cap \mathcal{E}_{p}$.

The second integral along an arc of circle
$$
I_{4} = \left| \frac{k}{(2\pi)^{1/2}}
\int_{-\infty}^{+\infty} \int_{C_{r_{1},\gamma_{p},\gamma_{p,p+1}}} W^{\mathfrak{d}_{p}}(\tau,m,\epsilon)
\exp( -(\frac{\tau}{\epsilon t})^{k} ) e^{izm} \frac{d\tau}{\tau} dm \right|
$$
can be estimated from above in a similar way. Namely, the arc of circle $C_{r_{1},\gamma_{p},\gamma_{p,p+1}}$ is again shaped in order that
$$ \mathrm{cos}(k (\theta - \mathrm{arg}(\epsilon t)) ) \geq \delta_{1} $$
for all $\theta \in [\gamma_{p},\gamma_{p,p+1}]$ (if $\gamma_{p} < \gamma_{p,p+1}$) or $\theta \in [\gamma_{p,p+1},\gamma_{p}]$
(if $\gamma_{p,p+1} < \gamma_{p}$), provided that $t \in \mathcal{T}$, $\epsilon \in \mathcal{E}_{p} \cap \mathcal{E}_{p+1}$,
for some fixed $\delta_{1}>0$. The bounds (\ref{bds_W_frac_dp}) along with the same arguments as above (\ref{bds_I3}) yield
\begin{equation}
I_{4} \leq \frac{2 k \varrho^{\mathfrak{d}_{p}}}{(2\pi)^{1/2}(\beta - \beta')}
|\gamma_{p} - \gamma_{p,p+1}|r_{1} \exp( -\delta_{2} \frac{r_{1}^{k}}{|\epsilon|^{k}} ) \label{bds_I4}
\end{equation}
for all $t \in \mathcal{T}$, $|\mathrm{Im}(z)| \leq \beta'$ obeying (\ref{t_small_in_mathcalT}) for some fixed $\delta_{2}>0$, whenever
$\epsilon \in \mathcal{E}_{p+1} \cap \mathcal{E}_{p}$.

In the ultimate part of the proof, it remains to examine the integral along the segment
$$
I_{5} = \left| \frac{k}{(2\pi)^{1/2}}
\int_{-\infty}^{+\infty} \int_{L_{\gamma_{p,p+1},0,r_{1}}}
\left( W^{\mathfrak{d}_{p+1}}(\tau,m,\epsilon) - W^{\mathfrak{d}_{p}}(\tau,m,\epsilon) \right)
\exp( -(\frac{\tau}{\epsilon t})^{k} ) e^{izm} \frac{d\tau}{\tau} dm \right|
$$
We need some preliminary ground work. We depart from a lemma that displays exponential upper bounds for the difference
$W^{\mathfrak{d}_{p+1}} - W^{\mathfrak{d}_{p}}$.
\begin{lemma} For every $0 \leq p \leq \varsigma-1$, we can single out two constants $K_{p}^{W},M_{p}^{W}>0$ such that
\begin{equation}
|W^{\mathfrak{d}_{p+1}}(\tau,m,\epsilon) - W^{\mathfrak{d}_{p}}(\tau,m,\epsilon)| \leq K_{p}^{W}
\exp( - \frac{M_{p}^{W}}{|\tau|^{k'}} ) (1 + |m|)^{-\mu} e^{-\beta |m|} \label{bds_difference_Wdp} 
\end{equation}
for all $\epsilon \in D(0,\epsilon_{0})$, all $m \in \mathbb{R}$, all $\tau \in S_{\mathfrak{d}_{p+1}} \cap S_{\mathfrak{d}_{p}} \cap D(0,r_{1})$
provided that
\begin{equation}
0 < r_{1} \leq (\frac{ \delta_{1} - \delta_{2}}{\nu_{2} (r/2)^{k_{1}-k'}})^{1/k'} \label{r1_small} 
\end{equation}
for some fixed $0 < \delta_{2} < \delta_{1}$, with the convention that $W^{\mathfrak{d}_{\varsigma}} = W^{\mathfrak{d}_{0}}$.
\end{lemma}
\begin{proof} By construction, we first notice that all the maps $u \mapsto w^{\mathfrak{d}_{p}}(u,m,\epsilon)$, $0 \leq p \leq \varsigma-1$,
are analytic continuations on the sector $U_{\mathfrak{d}_p}$ of a unique holomorphic function that we call $u \mapsto w(u,m,\epsilon)$
on the disc $D(0,r)$ which fulfills the same bounds (\ref{bds_w_frac_dp}). Furthermore, the application
$u \mapsto w(u,m,\epsilon) \exp( -(u/\tau)^{k'} )/u$ is holomorphic on $D(0,r)$ when
$\tau \in S_{\mathfrak{d}_{p+1}} \cap S_{\mathfrak{d}_{p}}$ and its integral is thus vanishing along an oriented path shaped as the union of\\
a) a segment departing from 0 to $(r/2)\exp(\sqrt{-1}\gamma_{p+1}')$\\
b) an arc of circle with radius $r/2$ joining the points $(r/2)\exp(\sqrt{-1}\gamma_{p+1}')$ and
$(r/2)\exp(\sqrt{-1}\gamma_{p}')$\\
c) a segment connecting $(r/2)\exp(\sqrt{-1}\gamma_{p}')$ and the origin.

As a result, by turning back to the integral representations (\ref{Wdp_Laplace_repres}) of $W^{\mathfrak{d}_{p+1}}$ and
$W^{\mathfrak{d}_{p}}$, we can recast the difference $W^{\mathfrak{d}_{p+1}} - W^{\mathfrak{d}_{p}}$ as a sum of three integrals
\begin{multline}
W^{\mathfrak{d}_{p+1}}(\tau,m,\epsilon) - W^{\mathfrak{d}_{p}}(\tau,m,\epsilon) = k'\int_{L_{\gamma_{p+1}',r/2}}
w^{\mathfrak{d}_{p+1}}(u,m,\epsilon) \exp( -(\frac{u}{\tau})^{k'} ) \frac{du}{u} \\
-
k'\int_{L_{\gamma_{p}',r/2}}
w^{\mathfrak{d}_{p}}(u,m,\epsilon) \exp( -(\frac{u}{\tau})^{k'} ) \frac{du}{u} + k'\int_{C_{r/2,\gamma_{p}',\gamma_{p+1}'}}
w(u,m,\epsilon) \exp( -(\frac{u}{\tau})^{k'} ) \frac{du}{u} \label{diff_W_dp_3_integrals}
\end{multline}
where the integrations paths are two halflines and an arc of circle staying away from the origin that are described as follows
\begin{multline*}
L_{\gamma_{p+1}',r/2} = [r/2,+\infty) \exp( \sqrt{-1}\gamma_{p+1}' ),
L_{\gamma_{p}',r/2} = [r/2,+\infty) \exp( \sqrt{-1}\gamma_{p}' ),\\
C_{r/2,\gamma_{p}',\gamma_{p+1}'} = \{ \frac{r}{2} \exp( \sqrt{-1} \theta ) : \theta \in [\gamma_{p}',\gamma_{p+1}'] \}
\end{multline*}

See Figure~\ref{fig2} for the configuration of the deformation of the integration paths.

\begin{figure}[h]
	\centering
		\includegraphics[width=0.48\textwidth]{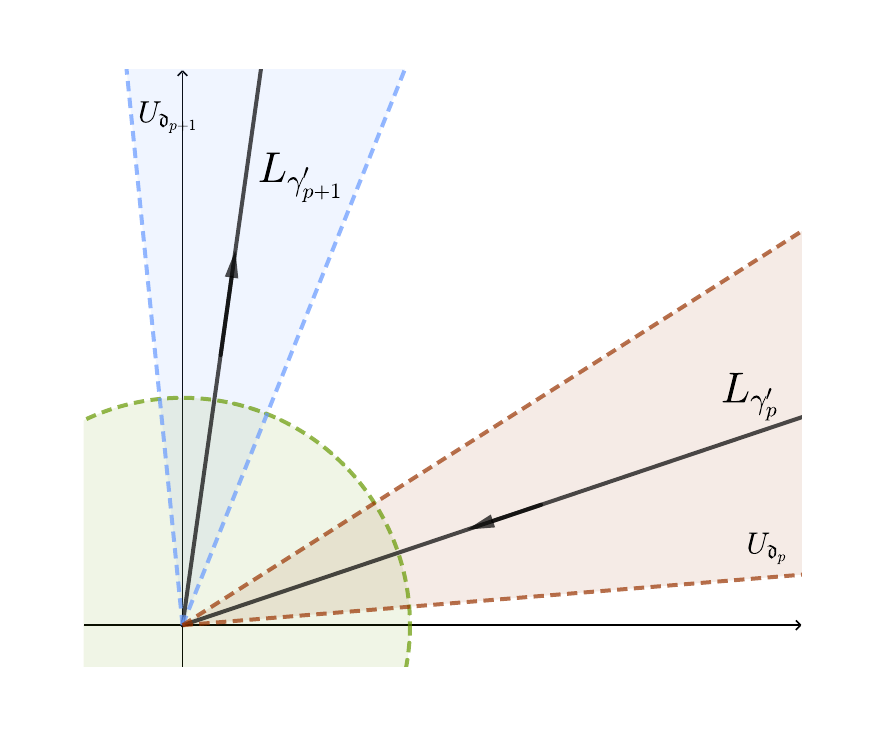}
		\includegraphics[width=0.48\textwidth]{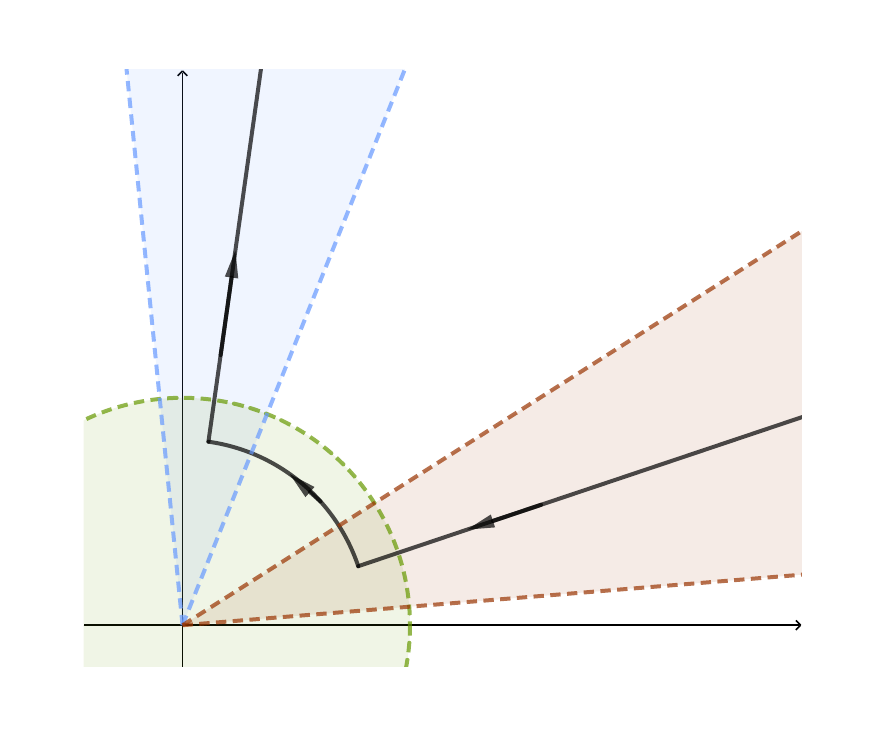}
	\caption{Initial (left) and deformation (right) of the integration paths}
			\label{fig2}
\end{figure}

We deal with the first integral along a halfline
$$ J_{1} = \left| k'\int_{L_{\gamma_{p+1}',r/2}}
w^{\mathfrak{d}_{p+1}}(u,m,\epsilon) \exp( -(\frac{u}{\tau})^{k'} ) \frac{du}{u} \right| $$
The direction $\gamma_{p+1}'$ (which depends on $\tau$) is suitably chosen in order that
$$ \cos( k'(\gamma_{p+1}' - \mathrm{arg}(\tau))) \geq \delta_{1} $$
for all $\tau \in S_{\mathfrak{d}_{p+1}} \cap S_{\mathfrak{d}_{p}}$, for some fixed $\delta_{1}>0$. Bearing in mind the estimates
(\ref{bds_w_frac_dp}) leads to
\begin{multline}
J_{1} \leq k' \int_{r/2}^{+\infty} \varpi^{\mathfrak{d}_{p+1}} (1 + |m|)^{-\mu} e^{-\beta|m|}
s \exp( \nu_{2} s^{k_1} ) \exp \left( - \frac{\cos( k'(\gamma_{p+1}' - \mathrm{arg}(\tau)))}{|\tau|^{k'}} s^{k'} \right) \frac{ds}{s}\\
\leq k' \varpi^{\mathfrak{d}_{p+1}} (1 + |m|)^{-\mu} e^{-\beta |m|}
\int_{r/2}^{+\infty} \exp \left( -\frac{s^{k'}}{|\tau|^{k'}}( \delta_{1} - \nu_{2}|\tau|^{k'}s^{k_{1} - k'} ) \right) ds\\
\leq k' \varpi^{\mathfrak{d}_{p+1}} (1 + |m|)^{-\mu} e^{-\beta |m|}
\int_{r/2}^{+\infty} \exp \left( -\frac{s^{k'}}{|\tau|^{k'}}( \delta_{1} - \nu_{2}|\tau|^{k'}(r/2)^{k_{1} - k'} ) \right) ds\\
\leq k' \varpi^{\mathfrak{d}_{p+1}} (1 + |m|)^{-\mu} e^{-\beta |m|}
\int_{r/2}^{+\infty} \exp \left( -\frac{s^{k'}}{|\tau|^{k'}} \delta_{2} \right) ds\\
\leq k' \varpi^{\mathfrak{d}_{p+1}} (1 + |m|)^{-\mu} e^{-\beta |m|} \frac{|\tau|^{k'}}{\delta_{2}}
\frac{1}{k'(r/2)^{k'-1}} \int_{r/2}^{+\infty} \frac{\delta_{2}}{|\tau|^{k'}} k' s^{k'-1}
\exp \left( -\frac{s^{k'}}{|\tau|^{k'}} \delta_{2} \right) ds\\
= k' \varpi^{\mathfrak{d}_{p+1}} (1 + |m|)^{-\mu} e^{-\beta |m|} \frac{|\tau|^{k'}}{\delta_{2}}
\frac{1}{k'(r/2)^{k'-1}} \exp \left( - \frac{ (r/2)^{k'} }{|\tau|^{k'}} \delta_{2} \right) \label{bds_J1}
\end{multline}
for all $\epsilon \in D(0,\epsilon_{0})$, all $m \in \mathbb{R}$, provided that $\tau \in S_{\mathfrak{d}_{p+1}} \cap S_{\mathfrak{d}_{p}}$ with
\begin{equation}
 |\tau| \leq ( \frac{\delta_{1} - \delta_{2}}{\nu_{2}(r/2)^{k_{1}-k'}} )^{1/k'} \label{tau_small_in_Sdp}
\end{equation}
for given $0 < \delta_{2} < \delta_{1}$.

In a similar manner, we supply bounds for the second integral over a halfline
$$ J_{2} = \left|  k'\int_{L_{\gamma_{p}',r/2}}
w^{\mathfrak{d}_{p}}(u,m,\epsilon) \exp( -(\frac{u}{\tau})^{k'} ) \frac{du}{u} \right| $$
Indeed, the direction $\gamma_{p}'$ (that relies on $\tau$) is properly chosen in order that
$$ \cos( k'(\gamma_{p}' - \mathrm{arg}(\tau)) ) \geq \delta_{1} $$
for all $\tau \in S_{\mathfrak{d}_{p+1}} \cap S_{\mathfrak{d}_{p}}$, for some fixed $\delta_{1}>0$. The use of (\ref{bds_w_frac_dp}) together
with a list of bounds akin to (\ref{bds_J1}) allows
\begin{equation}
J_{2} \leq k' \varpi^{\mathfrak{d}_{p}} (1 + |m|)^{-\mu} e^{-\beta |m|} \frac{|\tau|^{k'}}{\delta_{2}}
\frac{1}{k'(r/2)^{k'-1}} \exp \left( - \frac{ (r/2)^{k'} }{|\tau|^{k'}} \delta_{2} \right) \label{bds_J2}
\end{equation}
to hold whenever $\epsilon \in D(0,\epsilon_{0})$, $m \in \mathbb{R}$, $\tau \in S_{\mathfrak{d}_{p+1}} \cap S_{\mathfrak{d}_{p}}$
restricted to (\ref{tau_small_in_Sdp}) for given $0 < \delta_{2} < \delta_{1}$.

In the final part of the lemma, we evaluate the third integral along an arc of circle
$$ J_{3} = \left|  k'\int_{C_{r/2,\gamma_{p}',\gamma_{p+1}'}}
w(u,m,\epsilon) \exp( -(\frac{u}{\tau})^{k'} ) \frac{du}{u} \right| $$
The circle $C_{r/2,\gamma_{p}',\gamma_{p+1}'}$ satisfies the lower bounds
$$ \cos( k'( \theta - \mathrm{arg}(\tau))) \geq \delta_{1} $$
for all $\theta \in [\gamma_{p}',\gamma_{p+1}']$ (if $\gamma_{p}'<\gamma_{p+1}'$) or $\theta \in [\gamma_{p+1}',\gamma_{p}']$
(if $\gamma_{p+1}' < \gamma_{p}'$) granting that $\tau \in S_{\mathfrak{d}_{p+1}} \cap S_{\mathfrak{d}_{p}}$. Again, the estimates (\ref{bds_w_frac_dp}) lead to
\begin{multline}
J_{3} \leq k' | \int_{\gamma_{p}'}^{\gamma_{p+1}'} \max( \varpi^{\mathfrak{d}_{p}}, \varpi^{\mathfrak{d}_{p+1}} )
(1 + |m|)^{-\mu} e^{-\beta |m|} \frac{r}{2} \exp( \nu_{2} (\frac{r}{2})^{k_1} )\\
\times 
\exp \left( - \frac{\cos(k'(\theta - \mathrm{arg}(\tau)))}{|\tau|^{k'}} (\frac{r}{2})^{k'} \right) d\theta | \leq
k' \max( \varpi^{\mathfrak{d}_{p}}, \varpi^{\mathfrak{d}_{p+1}} ) |\gamma_{p+1}' - \gamma_{p}'|
(1 + |m|)^{-\mu} e^{-\beta |m|} \frac{r}{2}\\
\times \exp \left( -\frac{(r/2)^{k'}}{|\tau|^{k'}} ( \delta_{1} - |\tau|^{k'}\nu_{2}(r/2)^{k_{1}-k'}) \right) \leq
k' \max( \varpi^{\mathfrak{d}_{p}}, \varpi^{\mathfrak{d}_{p+1}} ) |\gamma_{p+1}' - \gamma_{p}'|
(1 + |m|)^{-\mu} e^{-\beta |m|} \frac{r}{2}\\
\times \exp ( -\frac{(r/2)^{k'}}{|\tau|^{k'}} \delta_{2} ) \label{bds_J3}
\end{multline}
for all $\epsilon \in D(0,\epsilon_{0})$, $m \in \mathbb{R}$ and $\tau \in S_{\mathfrak{d}_{p+1}} \cap S_{\mathfrak{d}_{p}}$
withstanding (\ref{tau_small_in_Sdp}) for given $0 < \delta_{2} < \delta_{1}$.

By collecting the above inequalities (\ref{bds_J1}), (\ref{bds_J2}) and (\ref{bds_J3}) applied to the decomposition
(\ref{diff_W_dp_3_integrals}), we reach the forecast bounds (\ref{bds_difference_Wdp}).
\end{proof}
From now on, we assume that the real number $r_{1}>0$ chosen above in the deformation a)b)c) of the straight halflines $L_{\gamma_{l}}$, $l=p,p+1$
is submitted to the restriction (\ref{r1_small}). As observed above, the direction $\gamma_{p,p+1}$ fulfills the lower estimates
$$ \cos( k(\gamma_{p,p+1} - \mathrm{arg}(\epsilon t))) \geq \delta_{1} $$
provided that $t \in \mathcal{T}$, $\epsilon \in \mathcal{E}_{p} \cap \mathcal{E}_{p+1}$ for some fixed $\delta_{1}>0$. The upper bounds 
(\ref{bds_difference_Wdp}) allow us to show that
\begin{multline}
I_{5} \leq \frac{k}{(2\pi)^{1/2}} \int_{-\infty}^{+\infty} \int_{0}^{r_1}
K_{p}^{W} (1 + |m|)^{-\mu} e^{-\beta |m|} \exp( - \frac{M_{p}^{W}}{r^{k'}} )\\
\times
\exp \left( - \frac{\cos(k (\gamma_{p,p+1} - \mathrm{arg}(\epsilon t)) )}{|\epsilon t|^{k}} r^{k} \right)
e^{-m \mathrm{Im}(z)} \frac{dr}{r} dm \leq \frac{2k K_{p}^{W}}{(2\pi)^{1/2}(\beta - \beta')} \tilde{I}_{5}(\epsilon t)
\label{bds_I5_1}
\end{multline}
where
\begin{equation}
\tilde{I}_{5}(\epsilon t)=\int_{0}^{r_1} \exp( -\frac{M_{p}^{W}}{r^{k'}} ) \exp( - \frac{\delta_{1}}{|\epsilon t|^{k}} r^{k} ) \frac{dr}{r}
\label{defin_tilde_I5}
\end{equation}
for all $\epsilon \in \mathcal{E}_{p} \cap \mathcal{E}_{p+1}$, $t \in \mathcal{T}$, $|\mathrm{Im}(z)| \leq \beta'$.

The study of estimates for $\tilde{I}_{5}(\epsilon t)$ as $\epsilon$ comes close to 0 has already been done in the proof of Theorem 1 from
our previous work \cite{lama2}. However, we display the full details of the arguments in order to keep them self contained. Namely, the bounds lean on the next two lemmas.

\begin{lemma}[Watson's Lemma. Exercise 4, page 16 in~\cite{ba}]

Let $b>0$ and $f:[0,b] \rightarrow \mathbb{C} $ be a continuous function having the formal expansion
$\sum_{n \geq 0}a_{n}t^n \in \mathbb{C}[[t]]$ as its asymptotic expansion of Gevrey order $\kappa>0$ at 0, meaning there exist
$C,M>0$ such that
$$\left |f(t)-\sum_{n=0}^{N-1}a_{n}t^n \right| \leq CM^{N}N!^{\kappa}|t|^{N},$$
for every $N \geq 1$ and $t\in [0,\delta]$, for some $0<\delta<b$. Then, the function
$$I(x)=\int_{0}^{b}f(s)e^{-\frac{s}{x}}ds$$
admits the formal power series $\sum_{n \geq 0} a_{n}n!x^{n+1} \in \mathbb{C}[[x]]$ as its asymptotic expansion
of Gevrey order $\kappa+1$ at 0, it is to say, there exist $\tilde{C},\tilde{K}>0$ such that
$$\left |I(x)-\sum_{n=0}^{N-1}a_{n}n!x^{n+1}\right| \leq \tilde{C}\tilde{K}^{N+1}(N+1)!^{1+\kappa}|x|^{N+1},$$
for every $N \geq 0$ and $x \in [0,\delta']$ for some $0<\delta'<b$.
\end{lemma}
\begin{lemma}[Exercise 3, page 18 in~\cite{ba}]
Let $\delta,q>0$, and $\psi:[0,\delta] \rightarrow \mathbb{C} $ be a continuous function. The following assertions are equivalent:
\begin{enumerate}
\item There exist $C,M>0$ such that $|\psi(x)|\le CM^{n}n!^{q}|x|^{n},$ for every $n \in \mathbb{N}$, $n\ge 0$ and
$x \in [0,\delta]$.
\item There exist $C',M'>0$ such that $|\psi(x)|\le C'e^{-M'/x^{\frac{1}{q}}}$, for every $x \in (0,\delta]$.
\end{enumerate}
\end{lemma}
We perform the change of variable $r^{k}=s$ into the integral (\ref{defin_tilde_I5}) and we get
$$ \tilde{I}_{5}(\epsilon t) = \frac{1}{k} \int_{0}^{r_{1}^{k}} \exp( - \frac{M_{p}^{W}}{s^{k'/k}} )
\exp( - \frac{\delta_1}{|\epsilon t|^{k}} s ) \frac{ds}{s} $$
We set $\psi_{k,k'}(s) = \exp( - M_{p}^{W}/s^{k'/k} )/s$. According to Lemma 8, two constants $C,M>0$ can be singled out with
$$ |\psi_{k,k'}(s)| \leq CM^{n}(n!)^{k/k'}|s|^{n} $$
for all $n \geq 0$, all $s \in [0,r_{1}^{k}]$. Owing to Lemma 7, we deduce that the function
$$ \tilde{I}(x) = \int_{0}^{r_{1}^{k}} \psi_{k,k'}(s) e^{-s/x} ds $$
has the formal series $\hat{0} \in \mathbb{C}[[x]]$ as asymptotic expansion of Gevrey order $\frac{k}{k'}+1$ on some segment $[0,\delta']$
with $0 < \delta' < r_{1}^{k}$. A second application of Lemma 8 implies the existence of two constants $C',M'>0$ with
$$ \tilde{I}(x) \leq C' \exp( -\frac{M'}{x^{\frac{k'}{k+k'}}} ) $$
for all $x \in [0,\delta']$. Finally, we deduce the existence of two constants $C_{I_5},M_{I_5}>0$ with
\begin{equation}
\tilde{I}_{5}(\epsilon t) \leq C_{I_5} \exp( - \frac{M_{I_5}}{|\epsilon t|^{\frac{kk'}{k+k'}}} ) \label{bds_tildeI5}
\end{equation}
for all $\epsilon \in \mathcal{E}_{p} \cap \mathcal{E}_{p+1}$, all $t \in \mathcal{T} \cap D(0,h_{p})$ for some $h_{p}>0$.

Gathering these last inequalities (\ref{bds_I5_1}) and (\ref{bds_tildeI5}) gives rise to the bounds
\begin{equation}
I_{5} \leq \frac{2k K_{p}^{W} C_{I_5}}{(2\pi)^{1/2}(\beta - \beta')}
\exp \left( - \frac{M_{I_5}}{h_{p}^{\frac{kk'}{k+k'}} |\epsilon|^{\frac{kk'}{k+k'}} } \right) \label{bds_I5}
\end{equation}
for all $\epsilon \in \mathcal{E}_{p} \cap \mathcal{E}_{p+1}$, all $t \in \mathcal{T} \cap D(0,h_{p})$ whenever $|\mathrm{Im}(z)| \leq \beta'$.

At last, the record of estimates (\ref{bds_I1}), (\ref{bds_I2}), (\ref{bds_I3}), (\ref{bds_I4}) and (\ref{bds_I5}) together with the breakup
(\ref{split_diff_up_five_parts}) yield the next inequality
\begin{multline}
\sup_{t \in \mathcal{T} \cap D(0,\sigma'), z \in H_{\beta'}} |u_{p+1}(t,z,\epsilon) - u_{p}(t,z,\epsilon)|
\leq \frac{2k (\varrho^{\mathfrak{d}_{p+1}} + \varrho^{\mathfrak{d}_{p}}) }{(2\pi)^{1/2}(\beta - \beta')} \frac{|\epsilon|^{k}}{\delta_{2} k r_{1}^{k-1}} \exp( - \delta_{2} \frac{r_{1}^{k}}{|\epsilon|^{k}} ) \\
+ \frac{2 k (\varrho^{\mathfrak{d}_{p+1}}|\gamma_{p+1} - \gamma_{p,p+1}| + \varrho^{\mathfrak{d}_{p}}|\gamma_{p} - \gamma_{p,p+1}|)
}{(2\pi)^{1/2}(\beta - \beta')} r_{1} \exp( -\delta_{2} \frac{r_{1}^{k}}{|\epsilon|^{k}} ) \\
+
\frac{2k K_{p}^{W} C_{I_5}}{(2\pi)^{1/2}(\beta - \beta')}
\exp \left( - \frac{M_{I_5}}{h_{p}^{\frac{kk'}{k+k'}} |\epsilon|^{\frac{kk'}{k+k'}} } \right)
\end{multline}
for some $\sigma'>0$ small enough, for all $\epsilon \in \mathcal{E}_{p} \cap \mathcal{E}_{p+1}$. Since $\frac{kk'}{k+k'} < k$, we finally
conclude that (\ref{exp_flat_difference_up}) holds.
\end{proof}

\section{Gevrey asymptotic expansions of the solutions in the perturbation parameter}

\subsection{Gevrey asymptotic expansions of order $1/\kappa$, $\kappa-$summable formal series and a Ramis-Sibuya theorem}

We first recall the definition of $\kappa-$summability of formal series with coefficients in a Banach space as introduced in
classical textbooks such as \cite{ba}.

\begin{defin} We set $(\mathbb{F},||.||_{\mathbb{F}})$ as a complex Banach space and we single out a real number $\kappa$ strictly larger than
$1/2$. A formal series
$$\hat{a}(\epsilon) = \sum_{j=0}^{\infty}  a_{j}  \epsilon^{j} \in \mathbb{F}[[\epsilon]]$$
with coefficients taken in $( \mathbb{F}, ||.||_{\mathbb{F}} )$ is said to be $\kappa-$summable
with respect to $\epsilon$ in the direction $d \in \mathbb{R}$ if \medskip

{\bf i)} a radius $\rho \in \mathbb{R}_{+}$ can be chosen in a way that the formal series, called formal
Borel transform of order $\kappa$ of $\hat{a}$,
$$ B_{\kappa}(\hat{a})(\tau) = \sum_{j=0}^{\infty} \frac{ a_{j} \tau^{j}  }{ \Gamma(1 + \frac{j}{\kappa}) } \in \mathbb{F}[[\tau]],$$
converge absolutely for $|\tau| < \rho$. \medskip

{\bf ii)} One can find an aperture $2\delta > 0$ in order that the series $B_{\kappa}(\hat{a})(\tau)$ can be analytically continued with
respect to $\tau$ on the unbounded sector
$S_{d,\delta} = \{ \tau \in \mathbb{C}^{\ast} : |d - \mathrm{arg}(\tau) | < \delta \} $. Moreover, there exist suitable $C >0$ and $K >0$
with the bounds
$$ ||B_{\kappa}(\hat{a})(\tau)||_{\mathbb{F}}
\leq C e^{ K|\tau|^{\kappa}} $$
whenever $\tau \in S_{d, \delta}$.
\end{defin}
If the constraints above are fulfilled, the vector valued Laplace transform of order $\kappa$ of $B_{\kappa}(\hat{a})(\tau)$
in the direction $d$ is set as
$$ L^{d}_{\kappa}(B_{\kappa}(\hat{a}))(\epsilon) = \epsilon^{-\kappa} \int_{L_{\gamma}}
B_{\kappa}(\hat{a})(u) e^{ - ( u/\epsilon )^{\kappa} } \kappa u^{\kappa-1}du,$$
along a half-line $L_{\gamma} = \mathbb{R}_{+}e^{\sqrt{-1}\gamma} \subset S_{d,\delta} \cup \{ 0 \}$, where $\gamma$ relies on
$\epsilon$ and is sort in such a way to satisfy
$\cos(k(\gamma - \mathrm{arg}(\epsilon))) \geq \delta_{1} > 0$, for some fixed $\delta_{1}$, for all
$\epsilon$ in a sector
$$ S_{d,\theta,R^{1/\kappa}} = \{ \epsilon \in \mathbb{C}^{\ast} : |\epsilon| < R^{1/\kappa} \ \ , \ \ |d - \mathrm{arg}(\epsilon) |
< \theta/2 \},$$
where the angle $\theta$ and radius $R$ withstand $0 < \theta < \frac{\pi}{\kappa} + 2\delta$ and $0 < R < \delta_{1}/K$.

It is worth noting that this Laplace transform of
order $\kappa$ differs slightly from the one displayed in Definition 1 which appears to be more suitable for the computations
related to the problems under study in this work.

The function $L^{d}_{\kappa}(B_{\kappa}(\hat{a}))(\epsilon)$
is called the $\kappa-$sum of the formal series $\hat{a}(\epsilon)$ in the direction $d$. It represents a bounded and holomorphic function on the sector $S_{d,\theta,R^{1/\kappa}}$ and turns out to be the \emph{unique} such function that possesses the formal series $\hat{a}(\epsilon)$ as Gevrey asymptotic
expansion of order $1/\kappa$ with respect to $\epsilon$ on $S_{d,\theta,R^{1/\kappa}}$. It means that for all
$0 < \theta_{1} < \theta$, there exist $C,M > 0$ such that
$$ ||L^{d}_{\kappa}(B_{\kappa}(\hat{a}))(\epsilon) - \sum_{p=0}^{n-1}
a_{p} \epsilon^{p}||_{\mathbb{F}} \leq CM^{n}\Gamma(1+ \frac{n}{\kappa})|\epsilon|^{n} $$
for all $n \geq 1$, all $\epsilon \in S_{d,\theta_{1},R^{1/\kappa}}$.\medskip

In the sequel, we state a cohomological criterion for the existence of Gevrey asymptotics of order $1/\kappa$ for proper families of sectorial
holomorphic functions and $k-$summability of formal series with coefficients in Banach spaces (see
\cite{ba2}, p. 121 or \cite{hssi}, Lemma XI-2-6) which is known as the Ramis-Sibuya theorem. This result
plays a central role in the proof of our second main statement (Theorem 2).\medskip

\noindent {\bf Theorem (RS)} {\it We consider a Banach space $(\mathbb{F},||.||_{\mathbb{F}})$ over $\mathbb{C}$ and
a good covering $\{ \mathcal{E}_{p} \}_{0 \leq p \leq \varsigma-1}$ in $\mathbb{C}^{\ast}$ (as explained in Definition 6). For all
$0 \leq p \leq \varsigma - 1$, let $G_{p}$ be a holomorphic function from $\mathcal{E}_{p}$ into
the Banach space $(\mathbb{F},||.||_{\mathbb{F}})$. We denote the cocycle $\Theta_{p}(\epsilon) = G_{p+1}(\epsilon) - G_{p}(\epsilon)$,
$0 \leq p \leq \varsigma-1$, which represents a holomorphic function from the sector $Z_{p} = \mathcal{E}_{p+1} \cap \mathcal{E}_{p}$ into
$\mathbb{F}$ (with the convention that $\mathcal{E}_{\varsigma} = \mathcal{E}_{0}$ and $G_{\varsigma} = G_{0}$).
We ask for the following requirements.

\noindent {\bf 1)} The functions $G_{p}(\epsilon)$ remain bounded as $\epsilon \in \mathcal{E}_{p}$ comes close to the origin
in $\mathbb{C}$, for all $0 \leq p \leq \varsigma - 1$.

\noindent {\bf 2)} The functions $\Theta_{p}(\epsilon)$ are exponentially flat of order $\kappa$ on $Z_{p}$, for all
$0 \leq p \leq \varsigma-1$, for some real number $\kappa > 1/2$. In other words, there exist constants $C_{p},A_{p}>0$ such that
$$ ||\Theta_{p}(\epsilon)||_{\mathbb{F}} \leq C_{p}e^{-A_{p}/|\epsilon|^{\kappa}} $$
for all $\epsilon \in Z_{p}$, all $0 \leq p \leq \varsigma-1$.

Then, for all $0 \leq p \leq \varsigma - 1$, the functions $G_{p}(\epsilon)$ share a common formal power series
$\hat{G}(\epsilon) \in \mathbb{F}[[\epsilon]]$ as Gevrey asymptotic expansion of order $1/\kappa$ on $\mathcal{E}_{p}$.

Moreover, for the special configuration where the aperture of one sector
$\mathcal{E}_{p_0}$ can be chosen slightly larger than $\pi/\kappa$, the function $G_{p_0}(\epsilon)$ is promoted as
the $\kappa-$sum of $\hat{G}(\epsilon)$ on $\mathcal{E}_{p_0}$.}

\subsection{Gevrey asymptotic expansion in the perturbation parameter for the analytic solutions to the initial value problem}

Throughout this subsection, we disclose the second central result of our work. We establish the existence of a formal power series
in the parameter $\epsilon$ whose coefficients are bounded holomorphic
functions on the product of a sector $\mathcal{T}$ with small radius centered at 0 and a strip $H_{\beta'}$ in $\mathbb{C}^2$, which
represent the common Gevrey asymptotic expansion of order $1/\kappa$, for some real number $\kappa > 1/2$ of the actual solutions
$u_{p}(t,z,\epsilon)$ of (\ref{main_PDE_u}) constructed in Theorem 1.\medskip

\noindent The second main result of this work can be stated as follows.

\begin{theo} Let $k,k'$ be the two integers considered in Theorem 1. We set
\begin{equation}
\kappa = \frac{kk'}{k+k'} \label{defin_kappa}
\end{equation}
We denote $\mathbb{F}$ the Banach space of complex valued bounded holomorphic functions on the product
$(\mathcal{T} \cap D(0,\sigma')) \times H_{\beta'}$ endowed with the supremum norm where the sector $\mathcal{T}$, radius $\sigma'>0$ and
width $\beta'>0$ are determined in Theorem 1. For all $0 \leq p \leq \varsigma - 1$, the holomorphic and bounded functions
$\epsilon \mapsto u_{p}(t,z,\epsilon)$
from $\mathcal{E}_{p}$ into $\mathbb{F}$ built up in Theorem 1 possess a common formal power series
$$ \hat{u}(t,z,\epsilon) = \sum_{m \geq 0} h_{m}(t,z) \epsilon^{m} \in \mathbb{F}[[\epsilon ]]$$
as Gevrey asymptotic expansion of order $1/\kappa$. More precisely, for all $0 \leq p \leq \varsigma-1$, we can single out two constants
$C_{p},M_{p}>0$ with
$$ \sup_{t \in \mathcal{T} \cap D(0,\sigma'),z \in H_{\beta'}} |u_{p}(t,z,\epsilon) - \sum_{m=0}^{n-1} h_{m}(t,z) \epsilon^{m}|
\leq C_{p}M_{p}^{n}\Gamma(1 + \frac{n}{\kappa}) |\epsilon|^{n}
$$
for all $n \geq 1$, whenever $\epsilon \in \mathcal{E}_{p}$.

Furthermore, if the aperture of one sector $\mathcal{E}_{p_0}$ can be taken
slightly larger than
$\pi/\kappa$, then the map $\epsilon \mapsto u_{p_0}(t,z,\epsilon)$ becomes the $\kappa-$sum of $\hat{u}(t,z,\epsilon)$ on $\mathcal{E}_{p_0}$.
\end{theo}
\begin{proof} We first observe that according to the assumptions made in Theorem 1, the inequalities $k \geq 1$ and $k' > k_{1} \geq 1$
imply that $\kappa \geq 2/3 > 1/2$. We aim attention at the family of functions $u_{p}(t,z,\epsilon)$, $0 \leq p \leq \varsigma-1$ constructed in Theorem 1.
For all $0 \leq p \leq \varsigma-1$, we define $G_{p}(\epsilon) := (t,z) \mapsto u_{p}(t,z,\epsilon)$, which represents by construction a
holomorphic and bounded function from $\mathcal{E}_{p}$ into the Banach space $\mathbb{F}$ of bounded holomorphic functions on
$(\mathcal{T} \cap D(0,\sigma')) \times H_{\beta'}$ equipped with the supremum norm, where $\mathcal{T}$ is a bounded sector selected in Theorem 1,
the radius $\sigma'>0$ is taken small enough and $H_{\beta'}$ is a horizontal strip of width $0 < \beta' < \beta$. In accordance with the
bounds (\ref{exp_flat_difference_up}), we deduce that the cocycle
$\Theta_{p}(\epsilon) = G_{p+1}(\epsilon) - G_{p}(\epsilon)$ is exponentially flat of order $\kappa$ on
$Z_{p} = \mathcal{E}_{p} \cap \mathcal{E}_{p+1}$, for any $0 \leq p \leq \varsigma-1$.

Owing to Theorem (RS) described overhead, we obtain a formal power series $\hat{G}(\epsilon) \in \mathbb{F}[[\epsilon]]$
which represents the Gevrey asymptotic expansion of order $1/\kappa$ of each $G_{p}(\epsilon)$ on $\mathcal{E}_{p}$, for
$0 \leq p \leq \varsigma - 1$. Furthermore, if the aperture of one sector $\mathcal{E}_{p_0}$ can be slightly chosen larger than $\pi/\kappa$,
then the function $G_{p_0}(\epsilon)$ represents
the $\kappa-$sum of $\hat{G}(\epsilon)$ on $\mathcal{E}_{p_0}$ as described within Definition 8.
\end{proof}

\noindent {\bf Example:} In order to show that summability can actually occur, we exhibit a configuration of an admissible set of
data $\underline{\mathcal{A}}$ which allows
$6/5-$summability on one sector for the case $k=3$ and $k'=2$ through the next example of equation (\ref{main_PDE_u})
which corresponds to the settings $\delta_{D}=2$, $m_{D}=3$, $k_{1}=1$, $\kappa_{1}=2$, $\kappa_{2}=6$ and $I = \{ (1,1), (1,0) \}$,
\begin{multline}
Q(\partial_{z})u(t,z,\epsilon) = R_{D}(\partial_{z})\epsilon^{6}(t^{4}\partial_{t})^{2}(t\partial_{t})^{3}u(t,z,\epsilon) \\
+
\epsilon^{3}c_{1,1}(z,\epsilon)R_{(1,1)}(\partial_{z})(t^{4}\partial_{t})(t\partial_{t})u(t,z,\epsilon)
+ \epsilon^{3}c_{1,0}(z,\epsilon)R_{(1,0)}(\partial_{z})t^{4}\partial_{t}u(t,z,\epsilon) + f(t,z,\epsilon)
\end{multline}

A possible configuration for the sets $\underline{S}$ and $\underline{\mathcal{E}}$ is displayed in Figure~\ref{fig3}, when assuming that $\mathcal{T}$ is a sector with bisecting direction $\theta=0$, and small opening. Observe that $\kappa$-summability is obtained on one of the sectors in $\underline{\mathcal{E}}$, with opening slightly larger than $5\pi/6$. Moreover, observe that the opening of the corresponding element in $\underline{S}$ is of opening strictly larger than $\pi/2$.

\begin{figure}[h]
	\centering
		\includegraphics[width=0.48\textwidth]{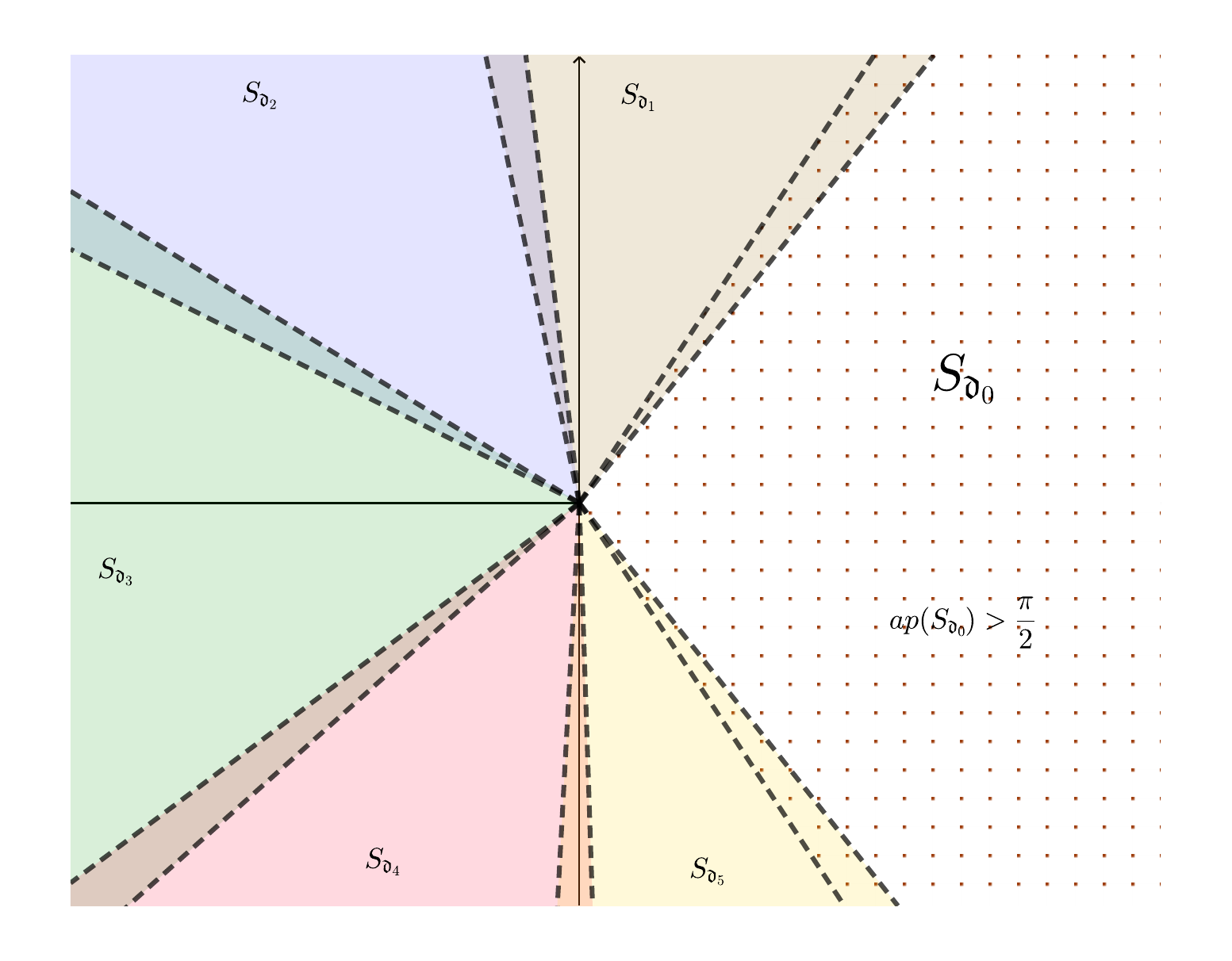}
		\includegraphics[width=0.48\textwidth]{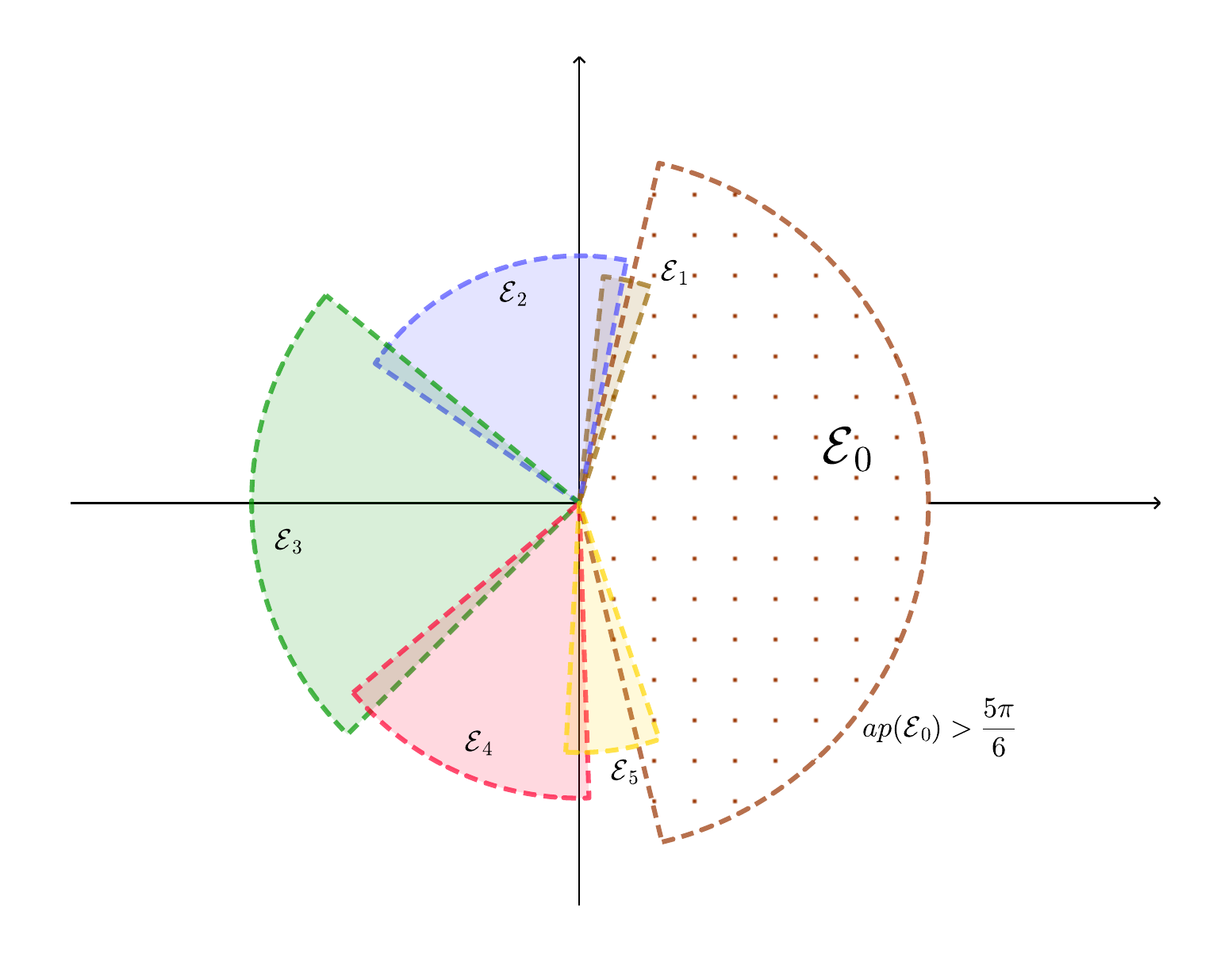}
	\caption{A configuration for summability in the Example: $S_{\mathfrak{d}_0}$ in $\underline{S}$ (left) and $\mathcal{E}_0$ in $\underline{\mathcal{E}}$ (right)}
			\label{fig3}
\end{figure}

\noindent \textbf{Acknowledgements.}  A. Lastra and S. Malek are supported by the Spanish Ministerio de Econom\'{\i}a, Industria y Competitividad under the Project MTM2016-77642-C2-1-P.


\begin{thebibliography}{99}
\bibitem{ba} W. Balser, \emph{From divergent power series to analytic functions. Theory and application
of multisummable power series.} Lecture Notes in Mathematics, 1582. Springer-Verlag, Berlin, 1994. x+108 pp.
\bibitem{ba1} W. Balser, \emph{Multisummability of complete formal solutions for
non-linear systems of meromorphic ordinary differential equations.} Complex Variables Theory Appl. 34 (1997), no. 1-2, 19--24.
\bibitem{ba2} W. Balser, \emph{Formal power series and linear systems of meromorphic ordinary differential equations.}
Universitext. Springer-Verlag, New York, 2000. xviii+299 pp.
\bibitem{ba4} W. Balser, \emph{Multisummability of formal power series solutions of partial differential
equations with constant coefficients.} J. Differential Equations 201 (2004), no. 1, 63--74.
\bibitem{babrrasi} W. Balser, B. Braaksma, J.-P. Ramis, Y. Sibuya, \emph{Multisummability of formal power
 series solutions of linear ordinary differential equations.} Asymptotic Anal. 5 (1991), no. 1, 27--45.
\bibitem{br} B. Braaksma, \emph{Multisummability of formal power series
solutions of nonlinear meromorphic differential equations.} Ann. Inst. Fourier (Grenoble) 42 (1992), no. 3, 517--540.
\bibitem{erd} A. Erdelyi, \emph{Higher transcendental functions.} Vol III. McGraw-Hill, New-York, 1953.
\bibitem{cota} O. Costin, S. Tanveer, \emph{Existence and uniqueness for a class of nonlinear higher-order partial
differential equations in the complex plane.} Comm. Pure Appl. Math. 53 (2000), no. 9, 1092--1117.
\bibitem{cota2} O. Costin, S. Tanveer, \emph{Short time existence and Borel summability in
the Navier-Stokes equation in $\mathbb{R}^{3}$}, Comm. Partial Differential Equations 34 (2009), no. 7-9, 785--817.
\bibitem{geta}  R. G\'{e}rard, H. Tahara, \emph{Singular nonlinear partial differential equations.} Aspects of Mathematics.
Friedr. Vieweg and Sohn, Braunschweig, 1996. viii+269 pp.
\bibitem{hssi} P. Hsieh, Y. Sibuya, \emph{Basic theory of ordinary differential equations}. Universitext. Springer-Verlag, New York,
1999.
\bibitem{ich} K. Ichinobe, \emph{On k-summability of formal solutions for certain higher order partial differential operators with polynomial coefficients.} Analytic, algebraic and geometric aspects of differential equations, 351--368, Trends Math., Birkh\"{a}user/Springer, Cham, 2017.
\bibitem{ich1} K. Ichinobe, \emph{On k-summability of formal solutions for a class of partial differential operators with time dependent coefficients.} J. Differential Equations 257 (2014), no. 8, 3048--3070.
\bibitem{lama1} A. Lastra, S. Malek, \emph{Parametric Gevrey asymptotics for some nonlinear initial value Cauchy problems},
J. Differential Equations 259 (2015), no. 10, 5220--5270.
\bibitem{lama2} A. Lastra, S. Malek, \emph{On parametric multisummable formal solutions to some nonlinear initial value Cauchy
problems,} Advances in Difference Equations 2015, 2015:200.
\bibitem{lama3} A. Lastra, S. Malek, \emph{Parametric Gevrey asymptotics for initial value problems with infinite order irregular
singularity and linear fractional transforms}, Advances in Difference Equations (2018), 2018:386.
\bibitem{lamasa1} A. Lastra, S. Malek, J. Sanz, \emph{On Gevrey solutions of threefold singular nonlinear partial differential
equations.} J. Differential Equations 255 (2013), no. 10, 3205--3232.
\bibitem{lod} M. Loday--Richaud, \emph{Divergent series, summability and resurgence. II. Simple and multiple summability.}
With prefaces by Jean-Pierre Ramis, \'{E}ric Delabaere, Claude Mitschi and David Sauzin. Lecture Notes in Mathematics, 2154. Springer, Cham,
2016. xxiii+272 pp.
\bibitem{lori} M. Loday-Richaud, \emph{Stokes phenomenon, multisummability and differential Galois groups.} Ann. Inst.
Fourier (Grenoble) 44 (1994), no. 3, 849--906.
\bibitem{ma2} S. Malek, \emph{On Gevrey asymptotics for some nonlinear integro-differential equations}.
J. Dyn. Control Syst. 16 (2010), no. 3, 377--406.
\bibitem{malra} B. Malgrange, J.-P. Ramis, \emph{Fonctions multisommables}. (French) [Multisummable functions]
Ann. Inst. Fourier (Grenoble) 42 (1992), no. 1-2, 353--368.
\bibitem{man2} T. Mandai, \emph{Existence and nonexistence of null-solutions for some non-Fuchsian partial differential
operators with $T$-dependent coefficients.} Nagoya Math. J. 122 (1991), 115--137.
\bibitem{mi} S. Michalik, \emph{On the multisummability of divergent solutions of linear partial
differential equations with constant coefficients.} J. Differential Equations 249 (2010), no. 3, 551--570.
\bibitem{mi1} S. Michalik, \emph{Multisummability of formal solutions of inhomogeneous
linear partial differential equations with constant coefficients.} J. Dyn. Control Syst. 18 (2012), no. 1, 103--133.
\bibitem{rasi} J.-P. Ramis, Y. Sibuya, \emph{A new proof of multisummability of formal
solutions of nonlinear meromorphic differential equations.} Ann. Inst. Fourier (Grenoble) 44 (1994), no. 3, 811--848.
\bibitem{taya} H. Tahara, H. Yamazawa, \emph{Multisummability of formal solutions to the Cauchy problem for some linear
partial differential equations}, Journal of Differential equations, Volume 255, Issue 10, 15 November 2013, pages 3592--3637.
\end{thebibliography}
\end{document}